\documentclass[a4paper,11pt]{article}

\usepackage[utf8]{inputenc}
\usepackage{geometry}
\usepackage{graphicx}
\usepackage{amsmath,amssymb,amsthm}

\usepackage{enumerate}
\usepackage{theoremref}
\usepackage[colorlinks=true, linkcolor=blue,citecolor=blue, urlcolor=blue]{hyperref}

\usepackage{tikz}
\usetikzlibrary{arrows,calc,plotmarks,positioning}
\tikzset{>=angle 60}

\newcommand{\tn}[1]{\textnormal{#1}}
\newcommand{\mc}[1]{\mathcal{#1}}

\newcommand{\N}{\mathbb{N}}
\newcommand{\R}{\mathbb{R}}

\newcommand{\Z}{\mathbb{Z}}

\newcommand{\veps}{\varepsilon}

\newcommand{\RadNik}[2]{\tfrac{\diff #1}{\diff #2}}

\newcommand{\la}{\left \langle}
\newcommand{\ra}{\right \rangle}

\newcommand{\diff}{\tn{d}}

\DeclareMathOperator*{\argmin}{arg min}

\DeclareMathOperator{\ddiv}{div}

\DeclareMathOperator{\Lip}{Lip}
\DeclareMathOperator{\spt}{spt}

\DeclareMathOperator{\KL}{KL}

\theoremstyle{plain}
\newtheorem{theorem}{Theorem}[section]
\newtheorem{lemma}[theorem]{Lemma}
\newtheorem{proposition}[theorem]{Proposition}

\theoremstyle{definition}
\newtheorem{definition}[theorem]{Definition}
\newtheorem{assumption}[theorem]{Assumption}

\newtheorem{remark}[theorem]{Remark}
\newtheorem{example}[theorem]{Example}

\usepackage{algpseudocode}

\algrenewcommand{\algorithmiccomment}[1]{\hfill$//$ #1}

\numberwithin{equation}{section}

\newcommand{\assign}{:=}

\newcommand{\meas}{\mc{M}}
\newcommand{\prob}{\mc{M}_1}

\usepackage{bm}

\usepackage{subcaption}

\newcommand{\energy}{\mc{E}}

\DeclareMathOperator{\ddivBanach}{Div}
\newcommand{\proj}{P}

\newcommand{\sigmaSet}{\mc{S}_{a,b}}
\newcommand{\fullSpace}{\mc{C}_{a,b}}

\newcommand{\dball}[1]{B^d(#1)}
\newcommand{\fullX}{\mathcal{X}}

\newcommand{\normalize}[1]{\mc{N}\left(#1\right)}

\usepackage{comment}

\title{Data-driven entropic spatially inhomogeneous evolutionary games}
\author{Mauro Bonafini, Massimo Fornasier, and Bernhard Schmitzer}
\date{\today}

\begin{document}
\maketitle
\begin{abstract}
We introduce novel multi-agent interaction models of entropic spatially inhomogeneous evolutionary undisclosed games and their quasi-static limits. These evolutions vastly generalize
first and second order dynamics. 
Besides the well-posedness of these novel forms of multi-agent interactions, we are concerned with the learnability of individual payoff functions from observation data.
We formulate the payoff learning as a variational problem, minimizing the discrepancy between the observations and the predictions by the payoff function.
The inferred payoff function can then be used to simulate further evolutions, which are fully data-driven.
We prove convergence of minimizing solutions obtained from a finite number of observations to a mean-field limit and the minimal value provides a quantitative error bound on the data-driven evolutions.
The abstract framework is fully constructive and numerically implementable. We illustrate this on computational examples where a ground truth payoff function is known and on examples where this is not the case, including a model for pedestrian movement.
\end{abstract}

\tableofcontents
\section{Introduction}
\subsection{Multi-agent interaction models in biological, social and economical dynamics}

In the course of the past two decades there has been an explosion of research on models of multi-agent interactions \cite{CuckerDong11,cucker-mordecki,CS,GC04, MR2000132,KeMinAuWan02,vicsek}, to describe phenomena beyond physics, e.g., in biology, such as cell aggregation and motility \cite{CDFSTB03,kese70,KocWhi98,be07}, coordinated animal motion \cite{BCCCCGLOPPVZ09,MR2507454,ChuDorMarBerCha07,crpito10,CouFra02,CKFL05,CS,Niw94,PE99,ParVisGru02,Rom96,TonTu95,YEECBKMS09}, coordinated human \cite{crpito11,CucSmaZho04,MR2438215} and synthetic agent behaviour and interactions, such as 
cooperative robots \cite{ChuHuaDorBer07,LeoFio01,PerGomElo09,SugSan97}. We refer to \cite{CCH13,CCP16,cafotove10,viza12} for recent surveys.

Two main mechanisms are considered in such models to define the dynamics.
The first takes inspiration from physical laws of motion and is based on pairwise forces encoding observed ``first principles'' 
of biological, social, or economical interactions, e.g., repulsion-attraction, alignment, self-propulsion/friction et cetera.
Typically this leads to Newton-like first or second order equations with ``social interaction'' forces, see \eqref{eq:NewtonModel} below.
In the second mechanism the dynamics are driven by an evolutive game where players simultaneously optimize their cost. Examples are game theoretic models of evolution \cite{zbMATH01169593} or 
mean-field games, introduced in \cite{lali07} and independently under the name
Nash Certainty Equivalence (NCE) in \cite{HCM03}, later greatly popularized, e.g., within consensus problems, for instance in \cite{NCM10,NCM11}.

More recently the notion of spatially inhomogeneous evolutionary games has been proposed \cite{almi2020alternate,AmFoMoSa2018,doi:10.1137/19M1273426} where the transport field for the agent population is directed by an evolutionary game on their available strategies.
Unlike mean-field games, the optimal dynamics are not chosen via an underlying non-local optimal control problem but by the agents' local (in time and space) decisions, see \eqref{eq:GameModel} below.

One fundamental goal of these studies is to reveal the relationship between the simple pairwise forces or incentives acting at the individual level and the emergence of global patterns in the behaviour of a given population.

\paragraph{Newtonian models.}

A common simple class of models for interacting multi-agent dynamics with pairwise interactions is inspired by Newtonian mechanics. The evolution of $N$ agents with time-dependent locations $x_1(t),\ldots,x_N(t)$ in $\R^d$ is described by the ODE system
\begin{align}
	\label{eq:NewtonModel}
	\partial_t x_i(t) = \frac{1}{N} \sum_{j=1}^N f(x_i(t),x_j(t))
	\qquad \tn{for } i=1,\ldots,N,
\end{align}
where $f$ is a pre-determined pairwise interaction force between pairs of agents.
The system is well-defined for sufficiently regular $f$ (e.g.~for $f$ Lipschitz continuous).
In this article, we will refer to such models as \emph{Newtonian models}.

First order models of the form \eqref{eq:NewtonModel} are ubiquitous in the literature and have, for instance, been used to model multiagent interactions in opinion formation \cite{kr00,hekr02}, vehicular traffic flow \cite{garavello2006traffic}, pedestrian motion \cite{crpito14}, and synchronisation of chemical and biological oscillators in neuroscience \cite{kuramoto2003chemical}.

Often one is interested in studying the behaviour of a very large number of agents. We may think of the agents at time $t$ to be distributed according to a probability measure $\mu(t)$ over $\R^d$. The limit-dynamics of \eqref{eq:NewtonModel} as $N \to \infty$ can under suitable conditions be expressed directly in terms of the evolution of the distribution $\mu(t)$ according to a \emph{mean-field equation}.
The mean-field limit of \eqref{eq:NewtonModel} is formally given by
\begin{align}
	\label{eq:NewtonModelMeanField}
	\partial_t \mu(t) + \ddiv\big( v(\mu(t)) \cdot \mu(t) \big) = 0
	\qquad \tn{where} \qquad
	v(\mu(t))(x) \assign \int_{\R^d} f(x,x')\,\diff \mu(t)(x').
\end{align}
Here $v(\mu(t))$ is a velocity field and intuitively $v(\mu(t))(x)$ gives the velocity of infinitesimal mass particles of $\mu$ at time $t$ and location $x$ \cite{Carrillo2011, Carrillo2014}.

While strikingly simple, such models only exhibit limited modeling capabilities.
For instance, the resulting velocity $\partial_t x_i(t)$ is simply a linear combination of the influences of all the other agents. `Importance' and `magnitude' of these influences cannot be specified separately: agent $j$ cannot tell agent $i$ to remain motionless, regardless of what other agents suggest.
Generally, agent $i$ has no sophisticated mechanism of finding a `compromise' between various potentially conflicting influences and merely uses their linear average.
The applicability of such models to economics or sociology raises concerns as it is questionable whether a set of static interaction forces can describe the behaviour of rational agents who are able to anticipate and counteract undesirable situations.

\paragraph{Spatially inhomogeneous evolutionary games.}
In a game dynamic the vector field $v(\mu(t))$ is not induced by a rigid force law but by the optimization of an individual payoff by each agent. In mean-field games this optimization is global in time, each agent plans their whole trajectory in advance, optimization can be thought of as taking place over repetitions of the same game (e.g.~the daily commute).
Alternatively, in spatially inhomogeneous evolutionary games \cite{AmFoMoSa2018} the agents continuously update their mixed strategies in a process that is local in time, which may be more realistic in some scenarios. This is implemented by the well-known replicator dynamics \cite{zbMATH01169593}.

As above, agents may move in $\R^d$. Let $U$ be the set of pure strategies. The map $e : \R^d \times U \to \R^d$ describes the spatial velocity $e(x,u)$ for an agent at position $x \in \R^d$ with a pure strategy $u \in U$.
For example, one can pick $U \subset \R^d$ to be a set of admissible velocity vectors and simply set $e(x,u) = u$.
$e$ acts as dictionary between strategies and velocities and can therefore assumed to be known. $e(x,u)$ may be more complicated, for instance, when certain velocities are inadmissible at certain locations due to obstacles.
A function $J:(\mathbb R^d\times U)^2 \to \mathbb R$ describes the individual benefit $J(x,u,x',u')$ for an agent at $x$ for choosing strategy $u$ when another agent sits at $x'$ and intends to choose strategy $u'$.
For example, $J(x,u,x',u')$ may be high when $e(x,u)$ points in the direction that $x$ wants to move, but could be lowered when the other agent's velocity $e(x',u')$ suggests an impending collision.
The state of each agent is given by their spatial position $x \in \R^d$ and a mixed strategy $\sigma \in \prob(U)$ (where $\prob(U)$ denotes probability measures over $U$). As hinted at, the evolution of $\sigma$ is driven by a replicator dynamic involving the payoff function $J$, averaged over the benefit obtained from all other agents, the spatial velocity $\partial_t x$ is obtained by averaging $\int_U e(x,u)\,\diff \sigma(u)$. The equations of motion are given by
\begin{subequations}
\label{eq:GameModel}
\begin{align}
	\partial_t x_i(t) & = \int_U e(x_i(t),u)\,\diff \sigma_i(t)(u), \label{eq:motion}\\
	\partial_t \sigma_i(t) & = \frac{1}{N} \sum_{j=1}^N f^J\big(x_i(t),\sigma_i(t),x_j(t),\sigma_j(t)\big) \label{eq:replicator}
\end{align}
\end{subequations}
where for $x,x' \in \R^d$ and $\sigma,\sigma' \in \prob(U)$
\begin{align}
f^J\big(x,\sigma,x',\sigma'\big)
		 \assign \left[ 
			\int_U J(x,\cdot,x',u')\,\diff \sigma'(u') - \int_U \int_U J(x,v,x',u')\,\diff \sigma'(u')\,\diff \sigma(v) \right] \cdot \sigma. \label{eq:rateJ}
\end{align}
Intuitively, the strategy $\sigma_i$ tends to concentrate on the strategies with the highest benefit for agent $i$ via \eqref{eq:rateJ}, which will then determine their movement via \eqref{eq:motion}.

Equation \eqref{eq:motion} describes the motion of agents in $\mathbb R^d$ and resembles \eqref{eq:NewtonModel}. The main difference is that the vector field is determined by the resulting solution of the replicator equation \eqref{eq:replicator}, which promotes strategies that perform best with respect to the individual payoff $J$ with rate \eqref{eq:rateJ}. Despite the different nature of the variables $(x_i,\sigma_i)$ of the model, the system \eqref{eq:GameModel} can also be re-interpreted as a Newtonian-like dynamics on the space $\R^d \times \prob(U)$ where each agent $y_i(t)=(x_i(t),\sigma_i(t))$ is characterized by its spatial location $x_i(t) \in \R^d$ and its mixed strategy $\sigma_i(t) \in \prob(U)$.
Accordingly, equations \eqref{eq:GameModel} can be more concisely expressed as
\begin{align}
	\label{eq:GameModelNewton}
	\partial_t y_i(t) = \frac{1}{N} \sum_{j=1}^N f\big(y_i(t),y_j(t)\big)
	\quad \tn{where} \quad
	f((x,\sigma),(x',\sigma')) \assign \left( {\textstyle \int_U e(x,u)\,\diff \sigma(u)}, f^J(x,\sigma,x',\sigma') \right).
\end{align}

Similarly to mean-field limit results in \cite{Carrillo2011, Carrillo2014}, the main contribution of \cite{AmFoMoSa2018} is to show that the large particle limit for $N\to \infty$ of solutions of \eqref{eq:GameModel} converges in the sense of probability measures to $\Sigma(t) \in \prob(\R^d \times \prob(U))$, which is, in analogy to \eqref{eq:NewtonModelMeanField}, solution of a nonlinear transport equation of the type
\begin{align}
	\label{eq:GameModelMeanField}
	\partial_t \Sigma(t) + \ddivBanach\big( v(\Sigma(t)) \cdot \Sigma(t) \big) = 0
	\quad \tn{where} \quad	
	v(\Sigma(t))(y) \assign \int f(y,y')\,\diff \Sigma(t)(y').
\end{align}

Note that \eqref{eq:GameModelMeanField} is a partial differential equation having as domain a (possibly infinite-dimensional) Banach space containing $\R^d \times \prob(U)$ and particular care must be applied to the use of an appropriate calculus needed to define differentials, in particular for the divergence operator.
In \cite{AmFoMoSa2018} Lagrangian and Eulerian notions of solutions to \eqref{eq:GameModelMeanField} are introduced.
As the technical details underlying the well-posedness of \eqref{eq:GameModelMeanField} do not play a central role in this article we refer to \cite{AmFoMoSa2018} for more insights. Nevertheless, we apply some of the general statements in \cite{AmFoMoSa2018} to establish well-posedness of the models presented in this paper.

\subsection{Learning or inference of multi-agent interaction models}
An important challenge is to determine the precise form of the interaction force or payoff functions.
While physical laws can often be determined with high precision through experiments, such experiments are often either not possible, or the models are not as precisely determined in more complex systems from biology or social sciences.
Therefore, very often the governing rules are simply chosen ad hoc to reproduce at some major qualitative effects observed in reality.

Alternatively one can employ model selection and parameter estimation methods to determine the form of the governing functions. Data-driven estimations have been applied in continuum mechanics~\cite{KIRCHDOERFER201681,MR3799091} computational sociology~\cite{BFHM16,Lu14424,lu2019learning,ZHONG2020132542} or economics~\cite{DDvolatility,cre03,Egger_2005}.
However, even the problem of determining whether time shots of a linear dynamical system do fulfil physically meaningful models, in particular have Markovian dynamics, is computationally intractable~\cite{PhysRevLett.108.120503}. 
For nonlinear models, the intractability of learning the system corresponds to the complexity of determining the set of appropriate candidate functions to fit the data. In order to break the curse of dimensionality, one requires prior knowledge on the system and the potential structure of the governing equations.
For instance, in the sequence of recent papers~\cite{trwa16,sctrwa17-1,sctrwa17} the authors assume that the governing equations are of first order and can be written as sparse polynomials, i.e., linear combinations of few monomial terms.

In this work we present an approach for estimating the payoff function for spatially inhomogeneous evolutionary games from observed data. It is inspired by the papers \cite{BFHM16,Lu14424,lu2019learning,ZHONG2020132542}, in particular by the groundbreaking paper \cite{BFHM16} which is dedicated to data-driven evolutions of Newtonian models.
In these references the curse of dimensionality is remedied by assuming that the interaction function $f$ in \eqref{eq:NewtonModel} is parametrized by a lower-dimensional function $a$, the identification of which is more computationally tractable. A typical example for such a parametrization is given by functions $f=f^a$ of the type
\begin{align*}
	f^a(x,x') = a(|x-x'|) \cdot (x'-x)
	\qquad \tn{for some} \qquad
	a: \R \to \R.
\end{align*}
The corresponding model \eqref{eq:NewtonModel} is used, for instance, in opinion dynamics \cite{kr00,hekr02}, vehicular traffic \cite{garavello2006traffic} or pedestrian motion \cite{crpito14}.
The learning or inference problem is then about the determination of $a$, hence of $f=f^a$, from observations of real-life realizations of the model \eqref{eq:NewtonModel}.
Clearly the problem can be formulated as an optimal control problem. However, as pointed out in \cite{BFHM16}, in view of the nonlinear nature of the function mapping $a$ in the corresponding solution $x^{[a]}(t)$ of \eqref{eq:NewtonModel}, the control cost would be nonconvex and the optimal control problem rather difficult to solve.
Instead, in \cite{BFHM16} a convex formulation is obtained by considering an empirical risk minimization in least squares form: 
given an observed realization $(x_1^{N}(t),\dots,x_N^{N}(t))$ of the dynamics between $N$ agents, generated by \eqref{eq:NewtonModel} governed by $f^a$, we aim at identifying $a$ by minimizing
\begin{align}
	\label{eq:NewtonEnergy}
	\energy^{N}(\hat{a}) & \assign \frac{1}{T} \int_0^T \left[
		\frac{1}{N} \sum_{i=1}^N \left\|\partial_t x_i^{N}(t)- \frac{1}{N} \sum_{j=1}^N f^{\hat{a}}\big(x_i^{N}(t),x_j^{N}(t)\big) \right\|^2
			\right] \diff t.
\end{align}
That is, along the observed trajectories we aim to minimize the discrepancy between observed velocities and those predicted by the model.
The functional $\energy^{N}$ plays the role of a {\it loss function} in the learning or inference task.
A time-discrete formulation in the framework of statistical learning theory has been proposed also in \cite{Lu14424,lu2019learning,ZHONG2020132542}. Under the assumption that $a$ belongs to a suitable compact class $X$ of smooth functions, the main results in \cite{BFHM16,Lu14424,lu2019learning,ZHONG2020132542} establish that minimizers
$$
\hat a_N \in \arg\min_{\hat a \in V_N}\energy^{N}(\hat{a}),
$$
converge to $a$ in suitable senses for $N\to \infty$, where the ansatz spaces $V_N \subset X$ are suitable finite-dimensional spaces of smooth functions (such as finite element spaces).

\subsection{Contribution and outline}
The main scope of this paper is to provide a theoretical and practical framework to perform learning/inference of spatially inhomogeneous evolutionary games, so that these models could be used in real-life data-driven applications. 
First, we discuss some potential issues with model \eqref{eq:GameModel} and provide some adaptations. We further propose and study in Section \ref{sec:Inference} several learning functionals for inferring the payoff function $J$ from the observation of the dynamics, i.e.~extending the approach of \cite{BFHM16} to our modified version of \cite{AmFoMoSa2018}. Let us detail our contributions.

\medskip\noindent
The proposed changes to the model \eqref{eq:GameModel} are as follows:
\begin{itemize}
	\item In Section \ref{sec:ModelEntropic}, we add entropic regularization to the dynamics for the mixed strategies of the agents. This avoids degeneracy of the strategies and allows faster reactions to changes in the environment, see Figure \ref{fig:BasicModelComparison}. Entropic regularization of games was also considered in \cite{FlamEntropicGames2005}. We show that the adapted model is well-posed and has a consistent mean-field limit (Theorem \ref{thm:wellPosedEntropic}).
	\item For interacting agents the assumption that the mixed strategies of other agents are fully known may often be unrealistic. Therefore, we will focus our analysis on the case where $J(x,u,x',u')$ does not explicitly depend on $u'$, the `other' agent's pure strategy. We refer to this as the \emph{undisclosed} setting. (The general case is still considered to some extent.)
	
	For this undisclosed setting, we study the quasi-static \emph{fast-reaction} limit of \eqref{eq:GameModel} where the dynamics of mixed strategies $(\sigma_i)_{i=1}^N$ run at much faster time-scale than dynamics of locations $(x_i)_{i=1}^N$, i.e.,~agents quickly adjust their strategies to changes in the environment. This will also be important for designing practical inference functionals (see below). The \emph{undisclosed fast-reaction} model is introduced in Section \ref{sec:ModelNTFR}. Well-posedness and consistent mean-field limit are proved by Theorem \ref{thm:wellPosedFastReaction}, convergence to the fast-reaction limit or quasi-static evolution as the strategy time scale becomes faster is given by Theorem \ref{thm:LambdaQuasiStaticConvergence}.
\end{itemize}
We claim that the resulting undisclosed fast-reaction entropic model is a useful alternative of the Newtonian model \eqref{eq:NewtonModel} and support these considerations with theoretical and numerical examples.
In particular, we show that any Newtonian model can be described (approximately) as an undisclosed fast-reaction entropic model whereas the converse it not true (Examples \ref{ex:FastReactionVsNewton} and \ref{ex:FastReactionVsNewtonII}).

We then discuss several inference functionals for the modified game model in Section \ref{sec:Inference}. We start with a rather direct analogue of \eqref{eq:NewtonEnergy} (Section \ref{sec:InferenceDifferential}), which would require not only the observation of the spatial locations $(x_i)_{i=1}^N$ and velocities $(\partial_t x_i)_{i=1}^N$, but also of the mixed strategies $(\sigma_i)_{i=1}^N$, and their temporal derivatives $(\partial_t \sigma_i)_{i=1}^N$. Whether the latter two can be observed in practice is doubtful, in particular with sufficient accuracy. Therefore, as an alternative, we propose two inference functionals for the undisclosed fast-reaction setting. A functional based on observed spatial velocities is given in Section \ref{sec:InferenceNTFRVel}. A functional based on mixed strategies (but not their derivatives) is proposed in Section \ref{sec:InferenceNTFRSigma} and we discuss some options how the required data could be obtained in practice.
In Section \ref{sec:InferenceNTFRTheory} some properties of these functionals are established, such as existence of minimizers and consistency of their mean-field versions.

Numerical examples are given in Section \ref{sec:Numerics}. These include the inference on examples where observations were generated with a true underlying undisclosed fast-reaction entropic model, as well as inference of Newtonian models and of a model for pedestrian motion adapted from \cite{bailo2018pedestrian, Degond2013}. We close with a brief discussion in Section \ref{sec:Conclusion}. Some longer technical proofs have been moved to Appendix \ref{sec:Analysis}.
Before presenting the new model, we collect some notation in the next subsection.

\subsection{Setting and notation}
\label{sec:Notation}
\paragraph{General setting.} Let $(Y,d_Y)$ be a complete and separable metric space. We denote by $\meas(Y)$ the space of signed Borel measures with finite total variation and by $\prob(Y)$ the set of probability measures with bounded first moment, i.e.,
\[
\prob(Y) = \left\{ \mu \in \meas(Y) \mid \mu\geq 0, \mu(Y) = 1, \int_Y d_Y(y,\bar{y})\,\diff\mu(y) < \infty \text{ for some }\bar{y} \in Y \right\}.
\]
For a continuous function $\varphi \in C(Y)$ we denote by
\[
\Lip(\varphi) \assign \sup_{\substack{x,y \in Y\\x\neq y}} \frac{|\varphi(x)-\varphi(y)|}{d_Y(x,y)}
\]
its Lipschitz constant. The set of bounded Lipschitz functions is then denoted by $\Lip_b(Y) = \{ \varphi \in C(Y) \mid \|\varphi\|_\infty + \Lip(\varphi) < \infty \}$, with $\|\cdot\|_\infty$ the classical sup-norm.

\smallskip
For $\mu_1,\mu_2 \in \prob(Y)$, the $1$-Wasserstein distance $W_1(\mu_1,\mu_2)$ is defined by
\begin{equation}\label{eq:W1primal}
W_1(\mu_1,\mu_2) \assign \inf \left\{ \int_{Y\times Y} d_Y(y_1,y_2) \diff\gamma(y_1,y_2) \mid \gamma \in \Gamma(\mu_1,\mu_2) \right\}
\end{equation}
where $\Gamma(\mu_1,\mu_2)$ is the set of admissible coupling between $\mu_1$ and $\mu_2$, i.e., $\Gamma(\mu_1,\mu_2) = \{ \gamma \in \prob(Y \times Y) \mid \gamma(A \times Y) = \mu_1(A) \text{ and } \gamma(Y \times B) = \mu_2(B)\}$. Due to Kantorovitch duality, one can also consider the equivalent definition
\begin{equation}\label{eq:W1dual}
W_1(\mu_1,\mu_2) \assign \sup \left\{ \int_{Y} \varphi\,\diff(\mu_1-\mu_2) \mid \varphi \in \Lip_b(Y), \Lip(\varphi) \leq 1 \right\}
\end{equation}
The metric space $(\prob(Y), W_1)$ is complete because $(Y, d_Y)$ is complete, and a sequence $(\mu_n)_{n \in \mathbb{N}}\allowbreak \subset \prob(Y)$ converges to $\mu \in \prob(Y)$ with respect to the distance $W_1$ if and only if $\mu_n$ converges to $\mu$ in duality with bounded Lipschitz functions and the first moments converge, i.e., $\int_{Y} d_Y(\cdot,\bar{y})\,\diff\mu_n \to \int_{Y} d_Y(\cdot,\bar{y})\,\diff\mu$ for all $\bar{y} \in Y$.

\paragraph{Interaction setting.}
We fix the space of pure strategies $U$ to be a compact metric space with distance $\textup{d}_U$. Each agent's mixed strategy is then described by $\sigma \in \prob(U)$. Agents move in $\R^d$, and we denote with $\|\cdot\|$ the usual Euclidean norm. For $R > 0$, $\dball{R}$ is the open ball of radius $R$ and centre the origin in $\R^d$. For $N \in \mathbb{N}$ and $\mathbf{x} = (x_1,\dots,x_N) \in [\R^d]^N$ we introduce the scaled norm $\|\cdot\|_N$ defined as
\[
\|\mathbf{x}\|_N = \frac1N \sum_{i=1}^N \|x_i\|.
\]
For $\sigma_1,\sigma_2 \in \prob(U)$ we set the Kullback--Leibler divergence to be
\[
\KL(\sigma_1|\sigma_2) \assign \begin{cases}
\int_U \log \left( \frac{\diff\sigma_1}{\diff\sigma_2} \right)\,\diff\sigma_1 & \tn{if } \sigma_1 \ll \sigma_2, \\
+ \infty & \tn{else,}
\end{cases} 
\]
where $\diff\sigma_1/\diff\sigma_2$ is the Radon--Nikodym derivative of $\sigma_1$ with respect to $\sigma_2$.

\smallskip
Throughout the paper we assume $e \in \Lip_b(\R^d \times U; \R^d)$ to be fixed and known. Such a function maps each pure strategy $u \in U$ into a given velocity $e(\cdot,u) \in \R^d$. We may think of $e$ as a label or dictionary for strategies.
The sets of admissible interaction rules $J$ are described by
\[
\fullX = \Lip_b(\R^d \times U \times \R^d \times U) \quad \text{and} \quad X = \Lip_b(\R^d \times U \times \R^d).
\]
Here, $\fullX$ consists of all possible payoff functions, modeling the case where each agent has a complete knowledge of both positions and mixed strategies of the other agents. On the other hand, in what we call the \emph{undisclosed} setting, each agent only knows the physical locations of the others, and thus in this context the payoff function $J$ will no longer depend on the other's strategy, making it possible to restrict the analysis to $X$.

\section{From spatially inhomogeneous evolutionary games to undisclosed fast-reaction dynamics}
\label{sec:Model}
We describe in this section how we adapt the spatially inhomogeneous evolutionary games model \eqref{eq:GameModel} of \cite{AmFoMoSa2018} to make it more suitable in practice for the task of modeling and inferring interaction rules between rational agents.
In Section \ref{sec:ModelEntropic} we add entropic regularization to avoid degeneracy of the agents' mixed strategies.
In Section \ref{sec:ModelNTFR} we focus on the more realistic undisclosed setting, where agents are not aware of other agents' strategy choices and derive a fast-reaction limit, describing the regime where choice of the strategies happens at a much faster time scale than physical movement of the agents.

\subsection{Entropic regularization for spatially inhomogeneous evolutionary games}
\label{sec:ModelEntropic}
In the model \eqref{eq:GameModel}, mixed strategies $\sigma_i$ are prone to converging exponentially to very singular distributions so that agents cannot react quickly to changes in the environment and exhibit strong inertia, contradicting the notion of rational agents. This behaviour can be illustrated with a simple example.

\begin{figure}[tb]
	\centering
	\includegraphics{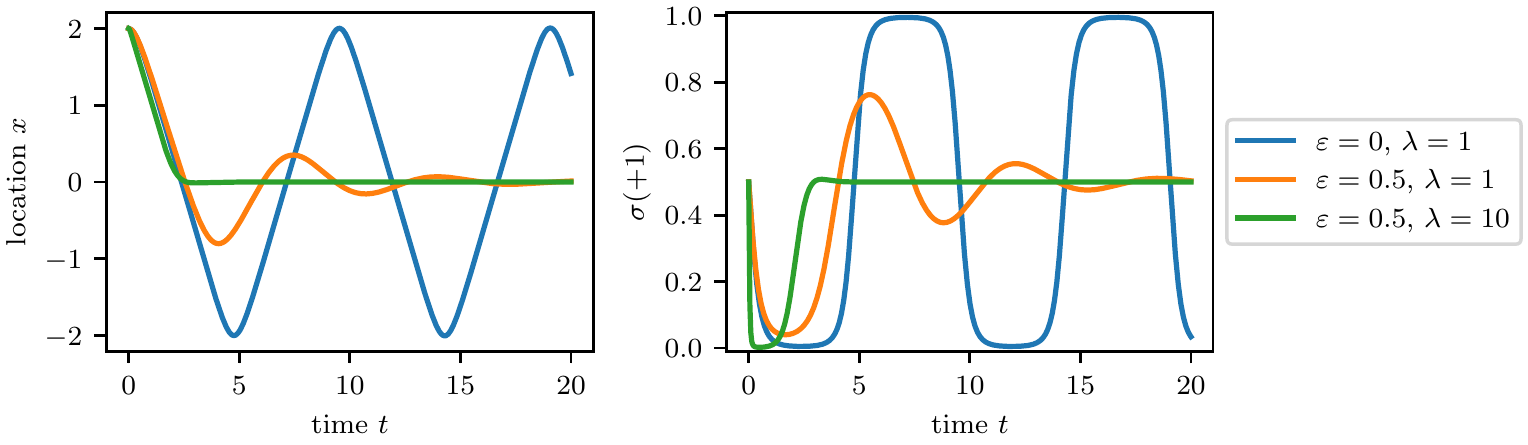}
	\caption{Comparison of original game dynamics and entropic regularization (possibly with accelerated time-scale for the strategy dynamics). For details see Examples \ref{ex:BasicModel} and \ref{ex:BasicModelEntropy}.}
	\label{fig:BasicModelComparison}
\end{figure}

\begin{example}[Slow reaction in spatially inhomogeneous evolutionary games]
\label{ex:BasicModel}
Let $d=1$, $U=\{-1,+1\}$, $e(x,u)=u$, so pure strategies correspond to moving left or right with unit speed. Let now $J(x,u,x',u')=-x \cdot u$, i.e.~agents prefer moving towards the origin, independently of the other agents' locations and actions. For simplicity we can set $N=1$ and choose $x(0)=2$, $\sigma(0)=(0.5,0.5)$, where we identify measures on $U=\{-1,+1\}$ with vectors of the simplex in $\R^2$. At $t=0$ the agent perceives strategy $-1$ as more attractive. The mixed strategy $\sigma$ starts pivoting towards this pure strategy and the agent starts moving towards the origin. The numerically simulated trajectory for the original model \eqref{eq:GameModel} is shown in Figure \ref{fig:BasicModelComparison} (corresponding to line $\varepsilon = 0$, $\lambda = 1$). We can see how $\sigma$ rapidly converges to the state $(0,1)$. Thus when the agent reaches the origin, it cannot react quickly and there is considerable overshoot.
\end{example}

As a remedy, we propose an adaptation of the model \eqref{eq:GameModel} by adding an entropic term to the dynamics of the mixed strategies $\sigma_i$.
The update rule \eqref{eq:replicator} modifies $\sigma_i$ towards maximizing the average benefit
\begin{equation*}
	\frac{1}{N} \sum_{j=1}^N \int_{U \times U} J(x_i,u,x_j,u')\,\diff \sigma_i(u)\,\diff \sigma_j(u').
\end{equation*}
To this objective we will now add the entropy $\veps\cdot \int_U \log(\RadNik{\sigma_i}{\eta})\,\diff \sigma_i$ where $\eta \in \prob(U)$ is a suitable `uniform' reference measure and $\veps>0$ is a weighting parameter.
This will pull $\sigma_i$ towards $\eta$ and thus prevent exponential degeneration.

\paragraph{Entropic regularization.} The entropic regularization implies that mixed strategies are no longer general probability measures over $U$, but they become densities with respect to a reference measure $\eta \in \prob(U)$, which we can assume without loss of generality to have full support, i.e.~$\spt(\eta) = U$. Thus, for a fixed $\varepsilon > 0$, set
\begin{align}
	\label{eq:SigmaSet}
	\sigmaSet \assign \left\{ \sigma\colon U \to \R_+, \;\sigma \text{ measurable,} \;\sigma(u) \in [a,b]\text{ } \eta\text{-a.e.},\;
		\int_U \sigma\,\diff \eta = 1 \right\}
\end{align}
where $0 < a < 1 < b < \infty$ (the bounds $a$ and $b$ are required for technical reasons, see below). For $\lambda > 0$, the modified entropic dynamics is then given by
\begin{subequations}
\label{eq:EntropicGameModel}
\begin{align}
	\label{eq:EntropicGameModelX}
	\partial_t x_i(t) & = \int_U e(x_i(t),u)\, \sigma_i(t)(u)\,\diff \eta(u), \\
	\label{eq:EntropicGameModelSigma}
	\partial_t \sigma_i(t) & = \lambda \cdot \left[ \frac{1}{N} \sum_{j=1}^N f^J\big(x_i(t),\sigma_i(t),x_j(t),\sigma_j(t)\big) +f^\veps\big(\sigma_i(t)\big) \right]
\end{align}
\end{subequations}
where the function $f^J$ now formally needs to be given in terms of densities, instead of measures:
\begin{multline}
	\label{eq:fJ}
	f^J\big(x,\sigma,x',\sigma'\big)
			 \assign \\ \left[ 
				\int_U J(x,\cdot,x',u')\,\sigma'(u')\,\diff \eta(u') - \int_U \int_U J(x,v,x',u')\,\sigma'(u')\,\sigma(v)\,\diff \eta(u')\,\diff \eta(v) \right] \cdot \sigma.
\end{multline}
The additional term $f^\veps$, corresponding to entropy regularization, is given by:
\begin{align}
	\label{eq:fH}
	f^\veps(\sigma) & \assign \veps \cdot \left[
		-\log\big(\sigma(\cdot)\big) + \int_U \log\big(\sigma(v)\big)\,\sigma(v)\,\diff \eta(v) \right] \cdot \sigma.
\end{align}
In \eqref{eq:EntropicGameModelSigma} we have explicitly added the factor $\lambda \in (0,\infty)$ to control the relative time-scale of the dynamics for mixed strategies and locations.

\begin{example}[Faster reactions with entropic regularization]
\label{ex:BasicModelEntropy}
We repeat here Example \ref{ex:BasicModel} with added entropy. We set $\eta=(0.5,0.5)$, keeping the symmetry between strategies $+1$ and $-1$ in the regularized system. Numerically simulated trajectories for this setup with different choices of $\lambda$ and $\veps$ are shown in Figure \ref{fig:BasicModelComparison}. For $(\lambda=1,\veps=0.5)$, the strategy does not get as close to $(0,1)$ as for $(\lambda=1,\veps=0)$. Therefore, upon crossing the origin the agent can react quicker, there is less overshoot and the trajectory eventually converges to $0$. Finally, for $(\lambda=10,\veps=0.5)$ the agent moves quickly and without overshoot to the origin, quickly adapting the mixed strategy.
\end{example}

\begin{remark}[Entropic gradient flow]
\label{rem:EntropicGradient}
Formally the contribution of $f^\veps$ to the dynamics \eqref{eq:EntropicGameModelSigma} corresponds to a gradient flow in the Hellinger--Kakutani metric over $\sigmaSet$ of the (negative) entropy
\begin{align}
	H(\sigma) \assign \int_U \log\big(\sigma\big)\,\sigma\,\diff \eta
\end{align}
of the density $\sigma$ with respect to the reference measure $\eta$.
\end{remark}

\begin{remark}[Multiple agent species]\label{rem:multipleSpecies}
	In the interaction models mentioned so far (\eqref{eq:NewtonModel}, \eqref{eq:GameModel}, \eqref{eq:EntropicGameModel}), all agents are of the same species (i.e.,~their movement is specified by the same functions). Often one wishes to model the interaction between agents of multiple species (e.g.,~predators and prey).
	The latter case can formally be subsumed into the former by expanding the physical space from $\R^d$ to $\R^{1+d}$ and using the first spatial dimension as `species label', i.e.,~an agent $\hat{x}_i=(n_i,x_i) \in \R^{1+d}$ describes an agent of species $n_i \in \Z$ at position $x_i \in \R^d$.
	The interaction force $\hat{f} : \R^{1+d} \times \R^{1+d} \to \R^{1+d}$ in the expanded version of \eqref{eq:NewtonModel} is then given by $\hat{f}\big((n,x),(n',x')) \assign (0,f_{n,n'}(x,x'))$ where $f_{n,n'} : \R^d \times \R^d \to \R^d$ is the interaction function for species $(n,n')$ and we set the first component of $\hat{f}$ to zero, such that the species of each agent remains unchanged. In an entirely analogous fashion functions $J$ and $e$ in models \eqref{eq:GameModel} and \eqref{eq:EntropicGameModel} can be extended.
\end{remark}

\paragraph{Well-posedness and mean-field limit of the regularized system.}
Following \cite{AmFoMoSa2018} (see \eqref{eq:GameModelNewton}), we analyse \eqref{eq:EntropicGameModel} as an interaction system in the Banach space $Y = \R^d \times L^p_\eta(U)$, $1 \leq p < \infty$. For $y = (x,\sigma) \in Y$ we set $\|y\|_Y = \|x\| + \|\sigma\|_{L^p_\eta(U)}$. Define
\begin{align}
	\label{eq:FullSpace}
	\fullSpace \assign \R^d \times \sigmaSet,
\end{align}
and for $y, y' \in \fullSpace$ let us set
\begin{align}
	\label{eq:BanachF}
f \colon \fullSpace \times \fullSpace \to Y, \quad f(y,y') = [f^{e}(x,\sigma), \lambda f^{J,\veps}(x,\sigma,x',\sigma')]
\end{align}
where
\begin{align}
	\label{eq:BanachFParts}
f^{e}(x,\sigma) = \int_U e(x,u)\, \sigma(u)\,\diff \eta(u) \quad \text{and} \quad f^{J,\veps}(x,\sigma,x',\sigma') = f^J\big(x,\sigma,x',\sigma'\big) +f^\veps\big(\sigma\big),
\end{align}
so that \eqref{eq:EntropicGameModel} takes the equivalent form
\begin{equation}\label{eq:BanachInteraction}
\partial_t y_i(t) = \frac{1}{N} \sum_{j=1}^N f(y_i(t), y_j(t)) \quad \text{ for }i = 1,\dots, N.
\end{equation}
Similar to \cite{AmFoMoSa2018}, well-posedness of \eqref{eq:BanachInteraction} relies on the Lipschitz continuity of $f$ and on the following compatibility condition (for the corresponding proofs see Appendix \ref{sec:Analysis}, Lemmas \ref{lem:fEpsLip}, \ref{lem:fJLip} and \ref{lem:CompatII}).
\begin{lemma}[Compatibility condition]
	\label{lem:Compat}
	For $J \in \fullX$ and $\varepsilon > 0$, let $f^J$ and $f^\varepsilon$ be defined as in \eqref{eq:fJ} and \eqref{eq:fH}. Then, there exist $a,b$ with $0 < a < 1 < b < \infty$ such that for any $(x,\sigma),(x',\sigma') \in \R^d \times \sigmaSet$ there exists some $\theta>0$ such that
	\begin{align}
	\label{eq:Compat}
	\sigma + \theta \lambda \left[f^J\big(x,\sigma,x',\sigma'\big) +f^\veps\big(\sigma\big) \right] \in \sigmaSet.
	\end{align}
\end{lemma}

Intuitively, for specific choices of the bounds $a$ and $b$, \eqref{eq:Compat} states that moving from $\sigma \in \sigmaSet$ into the direction generated by $f^J$ and $f^\veps$, we will always remain within $\sigmaSet$ for some finite time. Eventually, similar to \eqref{eq:GameModelMeanField}, a mean-field limit description of \eqref{eq:BanachInteraction} is formally given by
\begin{align}
	\label{eq:ContEqY}
	\partial_t \Sigma(t) + \ddivBanach\big(b(\Sigma(t)) \cdot \Sigma(t)\big) = 0
	\quad \text{with} \quad
	b(\Sigma(t))(y) \assign \int_{\fullSpace}\!\! f(y,y')\,\diff\Sigma(t)(y').
\end{align}
Like \eqref{eq:GameModelMeanField} this is a PDE whose domain is a Banach space and we refer to \cite{AmFoMoSa2018} for the technical details. We can then summarize the main result in the following theorem.

\begin{theorem}[Well-posedness and mean-field limit of entropic model]
	\label{thm:wellPosedEntropic}
	Let $J \in \fullX$, $\lambda, \veps > 0$, $T<+\infty$. Let $0 < a < 1 < b < \infty$ in accordance with Lemma \ref{lem:Compat}. Then:
	\begin{enumerate}
		\item Given $\bar{\mathbf{y}}^N = (\bar{y}_1^N,\dots, \bar{y}_N^N) \in \fullSpace^N$ there exists a unique trajectory $\mathbf{y}^N=(y_1^N,\ldots,y_N^N) \colon [0,T]\allowbreak \to \fullSpace^N$ of class $C^1$ solving \eqref{eq:BanachInteraction} with $\mathbf{y}^N(0) = \bar{\mathbf{y}}^N$. In particular, $\Sigma^N(t) \assign \tfrac{1}{N} \sum_{i=1}^N \delta_{y_i^N(t)}$ provides a solution of \eqref{eq:ContEqY} for initial condition $\bar{\Sigma}^N \assign \tfrac{1}{N} \sum_{i=1}^N \delta_{\bar{y}_i^N}$.
		\label{item:wellPosedEntropicDiscrete}
		\item Given $\bar{\Sigma} \in \prob(\fullSpace)$ there exists a unique $\Sigma \in C([0,T]; (\prob(\fullSpace), W_1) )$ satisfying in the weak sense the continuity equation \eqref{eq:ContEqY} with initial condition $\Sigma(0) = \bar{\Sigma}$.
		\label{item:wellPosedEntropicContinuous}
		\item For initial conditions $\bar{\Sigma}_1, \bar{\Sigma}_2 \in \prob(\fullSpace)$ and the respective solutions $\Sigma_1$ and $\Sigma_2$ of \eqref{eq:ContEqY} one has the stability estimate
		\begin{align}
			W_1(\Sigma_1(t),\Sigma_2(t)) \leq \exp\big(2\Lip(f)\,(t-s)\big) \cdot W_1(\Sigma_1(s),\Sigma_2(s))
		\end{align}
		for every $0 \leq s \leq t \leq T$.
		\label{item:wellPosedEntropicStability}
	\end{enumerate}
\end{theorem}
\begin{proof}
	The theorem follows by invoking Theorem 4.1 from \cite{AmFoMoSa2018}. On a more technical level, Theorem 5.3 of \cite{AmFoMoSa2018} provides the uniqueness of the solution to \eqref{eq:ContEqY} in a Eulerian sense.
	We now show that the respective requirements are met.
	
	First, the set $\fullSpace$, \eqref{eq:FullSpace}, is a closed convex subset of $Y$ with respect to $\|\cdot\|_Y$ for any $0 < a < b < +\infty$.
	Likewise, for any $0 < a < b < +\infty$ the map $f$ driving the evolution \eqref{eq:BanachInteraction} is Lipschitz continuous: indeed, one can prove that $f^J$ and $f^\varepsilon$ are Lipschitz continuous (see Lemmas \ref{lem:fEpsLip} and \ref{lem:fJLip} in Appendix \ref{sec:Analysis}) and Lipschitz continuity of $f^e$ follows from the fact that $e$ is Lipschitz continuous and bounded.
	Furthermore, as a consequence of Lemma \ref{lem:Compat}, one can choose $a$ and $b$ so that the extended compatibility condition
	\begin{align*}
	\forall\,R > 0 \, \exists \, \theta > 0: \quad
	y,y' \in \fullSpace \cap B_R(0)
	\quad \Rightarrow \quad
	y + \theta\,f(y,y') \in \fullSpace
	\end{align*}
	holds, cf.~\cite[Theorem B.1]{AmFoMoSa2018}.
	Therefore, we may invoke \cite[Theorem 4.1]{AmFoMoSa2018}. Combining this with \cite[Theorem 5.3]{AmFoMoSa2018}, which is applicable since $L^p_\eta(U)$ is separable and thus so is $Y$, the result follows.
\end{proof}

\subsection{Undisclosed setting and fast reactions}
\label{sec:ModelNTFR}
In the model \eqref{eq:EntropicGameModel} the decision process of each agent potentially involves knowledge of the mixed strategies of the other agents. Often it is reasonable to assume that there is no such knowledge, which can be reflected in the model by assuming that the payoff $J(x,u,x',u')$ does not actually depend on $u'$, the other agent's strategy.
Note that $J(x,u,x',u')$ may still depend on $x'$, the other agent's location.
We call this the \emph{undisclosed} setting.

Additionally, often it is plausible to assume that the (regularized) dynamic that governs the mixed strategies $(\sigma_i(t))_{i=1}^{N}$ of the agents runs at a much faster time-scale compared to the physical motion of the spatial locations $(x_i(t))_{i=1}^N$.
This corresponds to a large value of $\lambda$ in \eqref{eq:EntropicGameModel}.
Therefore, for the undisclosed setting we study the \emph{fast-reaction} limit $\lambda \to \infty$.

The main results of this section are Theorems \ref{thm:wellPosedFastReaction} and \ref{thm:LambdaQuasiStaticConvergence} which establish that the undisclosed fast-reaction limit is in itself a well-defined model (with a consistent mean-field limit as $N \to \infty$) and that the undisclosed model converges to this limit as $\lambda \to \infty$.

\paragraph{Undisclosed setting.}
In the undisclosed setting the general formulas \eqref{eq:EntropicGameModel} for the dynamics can be simplified as follows:
\begin{subequations}
\label{eq:NonTelepathicGameModel}
\begin{align}
	\label{eq:NonTelepathicGameModelX}
	\partial_t x_i(t) & = \int_U e(x_i(t),u)\, \sigma_i(t)(u)\,\diff \eta(u), \\
	\label{eq:NonTelepathicGameModelSigma}
	\partial_t \sigma_i(t) & = \lambda \cdot \left[ \frac{1}{N} \sum_{j=1}^N f^J\big(x_i(t),\sigma_i(t),x_j(t)\big) +f^\veps\big(\sigma_i(t)\big) \right]
\end{align}
\end{subequations}
where $f^\veps$ is as given in \eqref{eq:fH} and $f^J$ simplifies to
\begin{align}
	\label{eq:fJNonTelepathic}
	f^J\big(x,\sigma,x'\big)
			 \assign \left[ 
				J(x,\cdot,x') - \int_U J(x,v,x')\,\sigma(v)\,\diff \eta(v) \right] \cdot \sigma.
\end{align}
For finite $\lambda < +\infty$ the dynamics of this model are still covered by Theorem \ref{thm:wellPosedEntropic}.

\paragraph{Fast reactions of the agents.}
Intuitively, as $\lambda \to \infty$ in \eqref{eq:NonTelepathicGameModel}, at any given time $t$ the mixed strategies $(\sigma_i(t))_{i=1}^N$ will be in the unique steady state of the dynamics \eqref{eq:NonTelepathicGameModelSigma} for fixed spatial locations $(x_i(t))_{i=1}^N$.
For given locations $\mathbf{x} = (x_1,\ldots,x_N) \in [\R^d]^N$ this steady state is given by
\begin{subequations}
\label{eq:QuasiStaticGameModel}
\begin{align}
	\label{eq:QuasiStaticGameModelSigma}
	\sigma_{i}^J(\mathbf{x}) &\equiv \sigma_{i}^J(x_1,\ldots,x_N) \assign \frac{\exp\left(\tfrac{1}{\veps N}\sum_{j=1}^N J(x_i,\cdot,x_j)\right)}{
	\int_U \exp\left(\tfrac{1}{\veps N}\sum_{j=1}^N J(x_i,v,x_j)\right)\,\diff \eta(v)
	}.
\intertext{(This computation is explicitly shown in the proof of Theorem \ref{thm:LambdaQuasiStaticConvergence}.) The spatial agent velocities associated with this steady state are given by}
	\label{eq:QuasiStaticGameModelV}
	v_{i}^J(\mathbf{x}) &\equiv v_i^J(x_1,\dots,x_N) \assign \int_U e(x_i,u)\, \sigma_{i}^J(x_1,\ldots,x_N)(u)\,\diff \eta(u)
\intertext{and system \eqref{eq:NonTelepathicGameModel} turns into a purely spatial ODE in the form}
	\label{eq:QuasiStaticGameModelX}
	\partial_t x_i(t) &= v_i^J(x_1(t),\dots,x_N(t)) \quad \text{for } i = 1,\dots,N.
\end{align}
\end{subequations}
Unlike in Newtonian models over $\R^d$, \eqref{eq:NewtonModel}, here the driving velocity field $v_i^J(x_1(t),\ldots,x_N(t))$ is not a linear superposition of the contributions by each $x_j(t)$. This non-linearity is a consequence of the fast-reaction limit and allows for additional descriptive power of the model that cannot be captured by the Newton-type model \eqref{eq:NewtonModel}. This is illustrated in the two subsequent Examples \ref{ex:FastReactionVsNewton} and \ref{ex:FastReactionVsNewtonII}.

\begin{example}[Describing Newtonian models as undisclosed fast-reaction models]
\label{ex:FastReactionVsNewton}
Newtonian models can be approximated by undisclosed fast-reaction entropic game models. We give a sketch for this approximation procedure. For a model as in \eqref{eq:NewtonModel}, choose $U \subset \R^d$ such that it contains the range of $f$ (for simplicity we assume that it is compact), let $e(x,u) \assign u$ and then set
$$J(x,u,x') \assign -\|u-f(x,x')\|^2.$$
Accordingly, the stationary mixed strategy of agent $i$ is given by
\begin{align*}
\sigma_{i}^J(\mathbf{x})(u) & = \normalize{\exp\left(-\frac{1}{\veps N}\sum_{j=1}^N \|u-f(x_i,x_j)\|^2\right)}
\end{align*}
where $\normalize{\cdot}$ denotes the normalization operator for a density with respect to $\eta$. Observe now that
\begin{align*}
\frac{1}{N} \sum_{j=1}^N \|u-f(x_i,x_j)\|^2 &= \|u\|^2 - \frac2N \sum_{j=1}^N u \cdot f(x_i,x_j) + \frac1N \sum_{j=1}^N \|f(x_i,x_j)\|^2
\intertext{while}
\left\|u-\frac{1}{N}\sum_{j=1}^N f(x_i,x_j)\right\|^2 &= \|u\|^2 - \frac2N \sum_{j=1}^N u \cdot f(x_i,x_j) + \frac{1}{N^2} \left\| \sum_{j=1}^N f(x_i,x_j)\right\|^2.
\end{align*}
Hence, since the terms not depending on $u$ are cancelled by the normalization, we also have
\begin{align*}
\sigma_{i}^J(\mathbf{x})(u) & = \normalize{\exp\left(-\frac{1}{\veps}\left\|u-\frac{1}{N}\sum_{j=1}^N f(x_i,x_j)\right\|^2\right)}.
\end{align*}
So the maximum of the density $\sigma_{i}^J(\mathbf{x})$ is at $\tfrac{1}{N}\sum_{j=1}^N f(x_i,x_j)$, which is the velocity imposed by the Newtonian model.
For small $\veps$, $\sigma_{i}^J(\mathbf{x})$ will be increasingly concentrated around this point and hence the game model will impose a similar dynamic. Numerically this is demonstrated in Figure \ref{fig:newtonvsNTFR} for a $1$-dimensional Newtonian model as in  \eqref{eq:NewtonModel} driven by
\begin{align}\label{eq:NewtonTrajectoriesF}
f(x,x') = - x - \frac{\tanh(5(x'-x))}{(1+\|x'-x\|)^2}.
\end{align}
\end{example}

\begin{figure}[tb]
	\centering
	\includegraphics{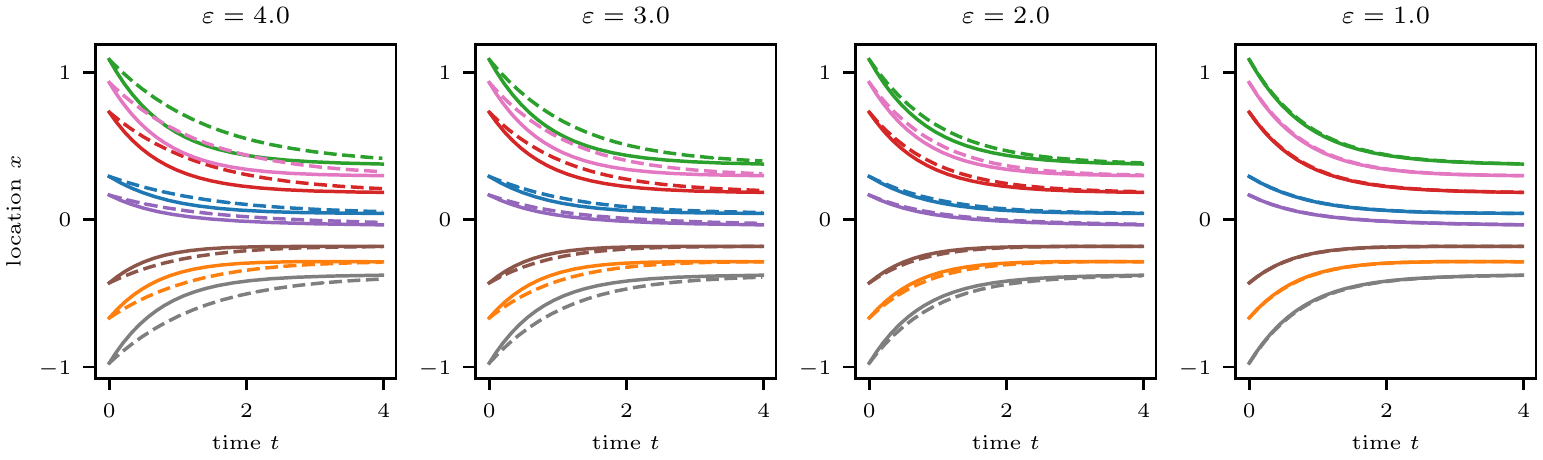}
	\caption{Approximation of a Newtonian model by an undisclosed fast-reaction entropic game model. Solid lines: original model, dashed lines: approximation. The Newtonian model is driven by  \eqref{eq:NewtonTrajectoriesF}, the approximation procedure is described in Example \ref{ex:FastReactionVsNewton}. The approximation becomes more accurate as $\veps$ decreases.}
	\label{fig:newtonvsNTFR}
\end{figure}

\begin{example}[Undisclosed fast-reaction models are strictly richer than Newtonian models]
\label{ex:FastReactionVsNewtonII}
In the previous example, each agent $j$ tried to persuade agent $i$ to move with velocity $f(x_i,x_j)$. Deviations from this velocity were penalized in the payoff function with a quadratic function. By minimizing the sum of these functions, the compromise of the linear average of all $f(x_i,x_j)$ then yields the best payoff. By picking different penalties we may obtain other interesting behaviour that cannot be described by Newtonian systems.

As above, let $U$ be a (sufficiently large) compact subset of $\R^d$, $e(x,u) \assign u$. Let $g : \R^d \to \R$ be a Lipschitz non-negative `bump function' with compact support, i.e.,~$g$ is maximal at $0$ with $g(0)=1$, and $g(u)=0$ for $\|u\|\geq \delta$ for some small $\delta >0$. Then we set
$$J(x,u,x') \assign (C-\|x-x'\|) \cdot g\left(u-\frac{x'-x}{\|x'-x\|} \right)-C \cdot g(u),$$
where $C>0$ is a suitable, sufficiently large constant. Then approximately, the function $u \mapsto \frac{1}{N}\sum_j J(x_i,u,x_j)$ is maximal at $u=\frac{x_j-x_i}{\|x_j-x_i\|}$ where $j$ is the agent that is closest to $i$ (if it is closer than $C$ and not closer than $\delta$, in particular the last term in $J$ is added to avoid that $u=0$ is the maximizer). Therefore, approximately, we have created a model, where agents are attracted with unit speed by their nearest neighbours, but not by other agents.
Such a non-linear trade-off between the influences of the other agents is not obtainable by a Newtonian model.
\end{example}

\begin{remark}[Fast reaction limit in the fully general case]
	\label{rem:TelepathicFastReaction}
	For the fully general setting (i.e., $J$ being also a function of $u'$) the study of the fast-reaction regime is a much harder problem as the stationary state of mixed strategies (for fixed spatial locations) often depends on the initial mixed strategies $(\sigma_i)_{i=1}^N$, thus the fast-reaction limit is a `genuine' quasi-static evolution problem. 
	We refer to \cite{MR3705699,SS2017,SS20172} for derivation and well-posedness results of quasi-static evolutions of critical points of nonconvex energies on the Euclidean space by means of vanishing-viscosity methods. A similar analysis is in the course of development for nonconvex energies on Hilbert spaces \cite{pers_comm}. Data-driven evolutions of critical points have been considered in \cite{2639}.
\end{remark}

\paragraph{Well-posedness and mean-field limit of the undisclosed fast-reaction system.}
Similar to the full model \eqref{eq:EntropicGameModel} we are also interested in a mean-field limit of the undisclosed fast-reaction limit as the number of agents tends to infinity, $N \to \infty$.
A formal limiting procedure leads to the equation
\begin{equation}
	\label{eq:ContEqR}
	\partial_t \mu(t) + \ddiv \left( v^J(\mu(t))\cdot \mu(t) \right) = 0
\end{equation}
where for $x \in \R^d$ and $\nu \in \prob(\R^d)$ we set
\begin{align}
	\label{eq:MeanFieldVel}
	v^J(\nu)(x) & \assign \int_U e(x,u)\sigma^J(\nu)(x, u)\,d\eta(u) \\
	\label{eq:MeanFieldSigma}
	\sigma^J(\nu)(x,\cdot) & \assign \frac{\exp\left(\tfrac{1}{\veps}\int_{\R^d} J(x,\cdot,x')\,d\nu(x')\right)}{
	\int_U \exp\left(\tfrac{1}{\veps}\int_{\R^d} J(x,v,x')\,d\nu(x')\right)\,\diff \eta(v)}.
\end{align}
Given $\mathbf{x} = (x_1,\dots,x_N) \in [\R^d]^N$ and setting $\mu^N = \frac1N \sum_{j=1}^N \delta_{x_j}$, \eqref{eq:QuasiStaticGameModel}, \eqref{eq:MeanFieldVel} and \eqref{eq:MeanFieldSigma} are related through
\begin{align*}
\sigma_i^J(\mathbf{x}) = \sigma_i^J(x_1,\dots,x_N) = \sigma^J(\mu^N)(x_i,\cdot) \quad \text{and} \quad v_i^J(\mathbf{x}) = v_i^J(x_1,\dots,x_N) = v^J(\mu^N)(x_i).
\end{align*}
The key ingredient for the study of \eqref{eq:QuasiStaticGameModel} and its limiting behaviour is to establish Lipschitz continuity of the fast-reaction equilibrium strategies and velocities with respect to payoff function and particle locations.
\begin{lemma}\label{lemma:UniformSigma}
	Let $J, J' \in X$ and consider $M > 0$ such that $M \geq \|J\|_{\infty} + \Lip(J)$ and $M \geq \|J'\|_{\infty} + \Lip(J')$. Let $\mu, \mu' \in \prob(\R^d)$ and $x, x' \in \R^d$. Then:
	\begin{enumerate}
		\item There exists $C = C(M, \veps)$ such that, for every $u \in U$,
		\begin{align}\label{eq:SigmaLip}
		\left| \sigma^{J}(\mu)(x,u) - \sigma^{J'}(\mu')(x',u) \right| &\leq C\left( W_1(\mu,\mu') + \|J - J'\|_\infty + \|x-x'\|\right)
		\end{align}
		and $1/C < \sigma^{J}(\mu)(x,u) < C$.
		\item There exists $C = C(M, e, \veps)$ such that
		\begin{align}\label{eq:VelocityLip}
		\left\| v^{J}(\mu)(x) - v^{J'}(\mu')(x') \right\| &\leq C\left( W_1(\mu,\mu') + \|J - J'\|_\infty + \|x-x'\|\right).
		\end{align}
	\end{enumerate}
\end{lemma}

\begin{proof}	
	Let us define two continuous functions $g, g' \colon U \to \R$ as
	\[
	\begin{aligned}
	g(u) = \int_{\R^d} J(x, u, y)\,d\mu(y) \quad \text{and} \quad g'(u) = \int_{\R^d} J'(x', u, y)\,d\mu'(y)
	\end{aligned}
	\]
	For every $u \in U$, using $M \geq \|J\|_{\infty}$ and $M \geq \|J'\|_{\infty}$, we immediately obtain the global bounds $-M \leq g(x) \leq M$ and $-M \leq g'(x) \leq M$. Using the triangle inequality and the dual definition for $W_1$ in \eqref{eq:W1dual}, we also estimate
	\begin{equation}\label{eq:Lipfu}
	|g(u)-g'(u)| \leq C\left( W_1(\mu,\mu') + \|J - J'\|_\infty + \|x-x'\|\right)
	\end{equation}
	for $C = C(M)$. Recall now that
	\[
	\sigma^{J}(\mu)(x,u) = \frac{\exp( g(u)/\veps )}{\int_U \exp\left( g(v)/\veps \right)\,d\eta(v)} \quad \text{and} \quad \sigma^{J'}(\mu')(x',u) = \frac{\exp( g'(u)/\veps )}{\int_U \exp\left( g'(v)/\veps \right)\,d\eta(v)}.
	\]
	Thus, \eqref{eq:SigmaLip} follows combining global boundedness of $g, g'$ with \eqref{eq:Lipfu}, the uniform bounds $1/C < \sigma^{J}(\mu)(x,u) < C$ are obtained in the same way, while \eqref{eq:VelocityLip} follows from \eqref{eq:SigmaLip} combined with $e \in \Lip_b(\R^d \times U; \R^d)$.
\end{proof}

For the discrete fast-reaction system, the following Lemma adapts the estimate \eqref{eq:VelocityLip} as one in terms of discrete particle locations and their velocities.
\begin{lemma}[Map to undisclosed fast-reaction agent velocity is Lipschitz continuous]\label{lemma:LipschitzVelocityMap}
	For any $N > 0$, consider the map
	\[
	\mathbf{v}^J \colon [\R^d]^N \to [\R^d]^N, \quad (x_1,\dots,x_N) \mapsto (v_1^J(x_1,\dots,x_N), \dots, v_N^J(x_1,\dots,x_N))
	\]
	associated to the fast-reaction ODE \eqref{eq:QuasiStaticGameModel}. Then, $\mathbf{v}^J$ is Lipschitz continuous under the distance induced by $\|\cdot\|_N$, with Lipschitz constant $L = L(J, e, \veps)$.
\end{lemma}
\begin{proof}
	Fix any $\mathbf{x},\mathbf{x}' \in [\R^d]^N$ and define $\mu = \frac{1}{N} \sum_{j=1}^N \delta_{x_i}$ and $\mu' = \frac{1}{N} \sum_{j=1}^N \delta_{x_i'}$. Then, by Lemma \ref{lemma:UniformSigma}, with $J' = J$, $x = x_i$ and $x' = x_i'$, we have
	\[
	\left\| v^{J}(\mu)(x_i) - v^{J}(\mu')(x_i') \right\| \leq C\left( W_1(\mu,\mu') + \|x_i-x_i'\|\right)
	\]
	for every $i = 1,\dots, N$, with $C = C(J,e,\veps)$. Using the definition of $W_1$ in \eqref{eq:W1primal}, we observe
	\[
	W_1(\mu,\mu') \leq \frac{1}{N}\sum_{j=1}^N \|x_i - x_i'\| = \|x - x'\|_N
	\]
	so that
	\begin{align}
	&\|\mathbf{v}^J(\mathbf{x})-\mathbf{v}^J(\mathbf{x}')\|_N =\frac{1}{N}\sum_{i=1}^N \| v_i^J(x_1,\dots,x_N) - v_i^J(x_1',\dots,x_N') \| \nonumber \\
	&\quad = \frac{1}{N}\sum_{i=1}^N\left\| v^{J}(\mu)(x_i) - v^{J}(\mu')(x_i') \right\|\leq \frac{C}{N} \sum_{i=1}^N \left( \|x-x'\|_N + \|x_i-x_i'\| \right) = 2C \|\mathbf{x}-\mathbf{x}'\|_N \nonumber 
	\end{align}
	which is the sought-after estimate.
\end{proof}

After clarifying what we mean by a solution of \eqref{eq:ContEqR} we summarize the mean-field result in Theorem \ref{thm:wellPosedFastReaction}, whose proof then builds on Lipschitz continuity of the velocity field in the discrete system and follows by fairly standard arguments (see, e.g., \cite{Carrillo2011, Carrillo2014}).
\begin{definition}
	We say a curve $\mu \in C([0,T]; (\prob(\R^d), W_1) )$ solves \eqref{eq:ContEqR} if $\mu(t)$ has uniformly compact support for $t \in [0,T]$ and
\begin{align*}
	\frac{d}{dt} \int_{\R^d} \phi(x)\diff\mu(t)(x)
	= \int_{\R^d} \nabla\phi(x) \cdot v^J(\mu(t))(x)\,\diff\mu(t)(x) \quad \text{for every $\phi \in C^\infty_c(\R^d)$}.
\end{align*}
\end{definition}

\begin{theorem}[Well-posedness and mean-field limit of undisclosed fast-reaction model]
	\label{thm:wellPosedFastReaction}
	Let $J \in X$, $\veps > 0$, $0 < \bar{R} < +\infty$, and $0 < T < +\infty$. Define $R = \bar{R} + T \cdot \|e\|_\infty$. Then:
	\begin{enumerate}
		\item Given $\bar{\mathbf{x}}^N = (\bar{x}_1^N,\dots, \bar{x}_N^N) \in [\dball{\bar{R}}]^N$ there exists a unique trajectory $\mathbf{x}^N=(x_1^N,\ldots,x_N^N) \colon\allowbreak [0,T] \to [\dball{R}]^N$ of class $C^1$ solving \eqref{eq:QuasiStaticGameModelX} with $\mathbf{x}^N(0) = \bar{\mathbf{x}}^N$.
		In particular, $\mu^N(t) \assign \tfrac{1}{N} \sum_{i=1}^N \delta_{x_i^N(t)}$ provides a solution of \eqref{eq:ContEqR} for initial condition $\bar{\mu}^N \assign \tfrac{1}{N} \sum_{i=1}^N \delta_{\bar{x}_i^N}$.
		\label{item:QuasiStaticDiscrete}
		\item Given $\bar{\mu} \in \prob(\dball{\bar{R}})$ there exists a unique $\mu \in C([0,T]; (\prob(\dball{R}), W_1) )$ satisfying in the weak sense the continuity equation \eqref{eq:ContEqR} with initial condition $\mu(0)=\bar{\mu}$.
		\label{item:QuasiStaticContinuous}
		\item For initial conditions $\bar{\mu}_1, \bar{\mu}_2 \in \prob(\dball{\bar{R}})$ and the respective solutions $\mu_1$ and $\mu_2$ of \eqref{eq:ContEqR} one has the stability estimate
		\begin{align}
			W_1(\mu_1(t),\mu_2(t)) \leq \exp(C\,(t-s)) \cdot W_1(\mu_1(s),\mu_2(s))
		\end{align}
		for every $0 \leq s \leq t \leq T$, with $C = C(J, e, \veps, \bar{R}, T)$.
		\label{item:QuasiStaticStability}
	\end{enumerate}
\end{theorem}

\begin{proof}[Proof outline]
	The proof is rather standard and builds on the Lipschitz continuity of the finite agents system \eqref{eq:QuasiStaticGameModel}. We highlight the main steps without in-depth details and refer to Proposition 2.1 and Theorem 2.4 of \cite{BFHM16}, and references therein, for further information.
	
	\emph{Part 1: finite agents setting.} For given $\bar{\mathbf{x}}^N = (\bar{x}_{1}^N,\dots, \bar{x}_{N}^N) \in [\dball{\bar{R}}]^N$, using the Lipschitz continuity provided in Lemma \ref{lemma:LipschitzVelocityMap}, there exists a unique curve $\mathbf{x}^N \colon [0,T] \to [\R^d]^N$ of class $C^1$ solving \eqref{eq:QuasiStaticGameModel} with $\mathbf{x}^N(0) = \bar{\mathbf{x}}^N$.
	A direct estimate provides $\|\partial x_i^N(t)\| \leq \|e\|_\infty$, so that $\|x_i^N(t)\| \leq \bar{R} + T \cdot \|e\|_\infty$ for every $t \in [0,T]$, $i = 1,\dots,N$. Furthermore, setting $\mu^N(t) \assign \tfrac{1}{N} \sum_{i=1}^N \delta_{x^N_i(t)}$ one can explicitly verify that $\mu^N$ solves \eqref{eq:ContEqR} with initial condition $\bar{\mu}^N \assign	\tfrac{1}{N} \sum_{i=1}^N \delta_{\bar{x}^N_i}$. This establishes point \ref{item:QuasiStaticDiscrete}.
	
	\emph{Part 2: mean-field solution as limit of finite agents solutions.}
	Fix $\bar{\mu} \in \prob(\dball{\bar{R}})$. For $N > 0$, let $(\bar{\mu}^N)_N \subset \prob(\dball{\bar{R}})$ be a sequence of empirical measures such that $W_1(\bar{\mu}^N,\bar{\mu}) \to 0$ as $N \to \infty$ and let $(\mu^N)_N$ be the respective sequence of solutions of \eqref{eq:ContEqR} with initial conditions $\bar{\mu}^N$. This sequence of curves $(\mu^N)_{N \in \N} \subset C([0,T]; (\prob(\dball{R}), W_1))$ is equicontinuous and equibounded. Therefore, an application of the Ascoli--Arzelà theorem provides a cluster point $\mu \in C([0,T]; (\prob(\dball{R}), W_1))$ such that, up to a subsequence, we have
	\[
	\lim_{N\to\infty} W_1(\mu^N(t), \mu(t)) = 0 \quad \text{ uniformly for } t \in [0,T].
	\]
	Invoking Lemma \ref{lemma:UniformSigma} for $x=x'$ and $J=J'$ this implies
	$\| v^J(\mu^N(t))(x) - v^J(\mu(t))(x)\| \to 0$ uniformly in $t \in [0,T]$, $x \in \dball{R}$ as $N \to \infty$. Consequently, the cluster point $\mu$ is a solution to \eqref{eq:ContEqR} with $\bar{\mu}$ as initial condition. This establishes the existence part of point \ref{item:QuasiStaticContinuous}. 
	
	\emph{Part 3: stability estimates.}
	For fixed $N$ a stability estimate of the form $\|\mathbf{x}^N_1(t)-\mathbf{x}	^N_2(t)\|_N \leq \exp(C\,(t-s)) \cdot \|\mathbf{x}^N_1(s)-\mathbf{x}	^N_2(s)\|_N$ for solutions to the discrete system \eqref{eq:QuasiStaticGameModel} follows quickly from Grönwall's lemma. Since the $\|\cdot\|_N$-norm between two point clouds provides an upper bound for the $W_1$ distance between the respective empirical measures this would provide point \ref{item:QuasiStaticStability} for empirical measures.
	The extension to arbitrary measures can be done as in the proof of Theorem 2.4 in \cite{BFHM16}. This then provides uniqueness of the solution $\mu$ to \eqref{eq:ContEqR} for initial condition $\bar{\mu} \in \prob(\dball{\bar{R}})$, completing point \ref{item:QuasiStaticContinuous}.
\end{proof}

Finally, we establish convergence of \eqref{eq:NonTelepathicGameModel} to \eqref{eq:QuasiStaticGameModel} (and their respective mean-field versions) as $\lambda \to \infty$. The proof is given in Section \ref{sec:NonTelepathicFastReactionProofs} of the Appendix, a simple numerical example is shown in Figure \ref{fig:forwardlambdatoinfinity}.
\begin{theorem}[Convergence to fast-reaction limit in undisclosed setting as $\lambda \to \infty$]\hfill
	\label{thm:LambdaQuasiStaticConvergence}
	\begin{enumerate}
	\item \textbf{Discrete setting:} For initial positions $\mathbf{x}(0)=(x_1(0),\ldots,x_N(0)) \in [\dball{\bar{R}}]^N$ and initial mixed strategies $\bm{\sigma}(0)=(\sigma_1(0),\ldots,\sigma_N(0)) \in \sigmaSet^N$ let $\mathbf{x}(t) = (x_1(t),\ldots,x_N(t))$ and $\bm{\sigma}(t)=(\sigma_1(t),\ldots,\sigma_N(t))$ be the solution to the undisclosed model \eqref{eq:NonTelepathicGameModel}.
	For the same initial positions let $\mathbf{x}^{\ast \ast}(t)=(x^{\ast \ast}_1(t),\ldots,x^{\ast \ast}_N(t))$ be the solution to the undisclosed fast-reaction model \eqref{eq:QuasiStaticGameModel} and let $\bm{\sigma}^{\ast \ast}(t)=(\sigma_1^{\ast \ast}(t),\ldots,\sigma_N^{\ast \ast}(t))$ with $\sigma_i^{\ast \ast}(t)=\sigma^J_{i}(\mathbf{x}^{\ast \ast}(t))$ be the corresponding fast-reaction strategies.
	
	Then, there exists $C = C(a,b,\bar{R},J,e)$ (independent of $N$ and $i$) such that for all $t \in [0,\infty)$,
	\begin{align}
		\label{eq:QuasiStaticConvergenceX}
		\|\mathbf{x}(t)-\mathbf{x}^{\ast \ast}(t)\|_N \leq \frac{C}{\sqrt{\lambda}} \cdot \exp(t \cdot C)
	\end{align}
	and
	\begin{align}
		\label{eq:QuasiStaticConvergenceSigma}
		\|\sigma_i(t)-\sigma_i^{\ast\ast}(t)\|_{L^2_\eta(U)} \leq \frac{C}{\sqrt{\lambda}} \cdot \exp(t \cdot C) + 
			\left[C \left[\tfrac{1}{\lambda}+ \exp\left(-\tfrac{\lambda\,t}{C}\right)\right]\right]^{1/2}.
	\end{align}
	So on any compact time horizon $[0,T]$ we find that $\|\mathbf{x}(t) - \mathbf{x}^{\ast \ast}(t)\|_N \to 0$ uniformly in time as $\lambda \to \infty$.
	$\bm{\sigma}$ converges to $\bm{\sigma}^{\ast \ast}$ uniformly in time on the interval $[\tau,T]$ for any $\tau>0$. Near $t=0$ we cannot expect uniform convergence of $\bm{\sigma}$, since by initialization it can start at a finite distance from $\bm{\sigma}^{\ast \ast}(0)$ and thus it takes a brief moment to relax to the vicinity of the fast-reaction state.
	\item \textbf{Mean-field setting:}
	For an initial configuration $\bar{\Sigma} \in \prob(\fullSpace)$ let $\Sigma$ be the solution to the entropic mean-field model \eqref{eq:ContEqY} for a undisclosed $J$ with $\lambda \in (0,\infty)$.
	Let $\proj : \fullSpace \to \R^d$, $(x, \sigma) \mapsto x$ be the projection from $\fullSpace$ to the spatial component.
	Set $\bar{\mu} \assign \proj_\sharp \bar{\Sigma}$ and let $\mu$ be the solution to the fast-reaction mean-field model \eqref{eq:ContEqR} with initial condition $\mu(0)=\bar{\mu}$. Then for the same $C$ as in \eqref{eq:QuasiStaticConvergenceX} one has
	\begin{align}
		\label{eq:QuasiStaticConvergenceMeanField}
		W_1(\mu(t),\proj_\sharp \Sigma(t)) \leq \frac{C}{\sqrt{\lambda}} \cdot \exp(t \cdot C).
	\end{align}
	\end{enumerate}
\end{theorem}

\begin{figure}
	\centering
	\includegraphics{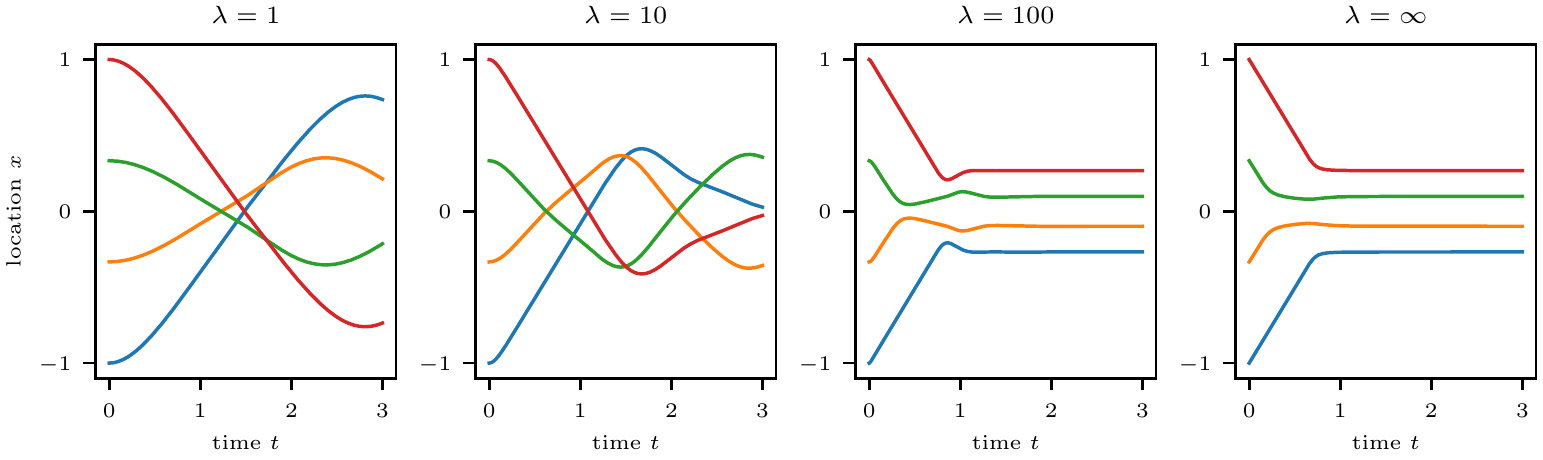}
	\caption{Convergence of the undisclosed model to the fast-reaction limit as $\lambda \to \infty$. The payoff function encourages agents to move to the origin, but penalizes small pairwise distances (see Example \ref{ex:NTFRd1} for a detailed description). With $\lambda$ small, agents cannot adjust their strategies fast enough. They overshoot the origin and completely fail to avoid each other. The situation improves as $\lambda$ increases. For $\lambda=100$ the model closely resembles the fast-reaction limit, in accordance with \eqref{eq:QuasiStaticConvergenceX}.}
	\label{fig:forwardlambdatoinfinity}
\end{figure}

\begin{remark}
The statement about the mean-field equations can be expanded further: From the fast-reaction solution $\mu$ on $\prob(\R^d)$ one can construct a `lifted trajectory' $\hat{\Sigma}$ on $\prob(\fullSpace)$, intuitively by attaching to each mass particle in $\mu$ at $x$ its corresponding fast-reaction mixed strategy $\sigma^J(\mu)(x,\cdot)$, \eqref{eq:MeanFieldSigma}.
Using the bounds \eqref{eq:QuasiStaticConvergenceX} and \eqref{eq:QuasiStaticConvergenceSigma}, the continuity properties of $\sigma^J(\mu)(x,\cdot)$ (see Lemma \ref{lemma:UniformSigma}) and arguing as in the proof of \eqref{eq:QuasiStaticConvergenceMeanField} one can then establish a $W_1$-bound between $\Sigma$ and $\hat{\Sigma}$.
\end{remark}

\section{Inference for entropic evolutionary games}
\label{sec:Inference}
After having introduced a new class of models for interacting agents, we now turn to the question of how the payoff function $J$, that parametrizes a model, can be inferred from data.
Some remarks are in order before we proceed.

\subsection{Discussion}
\label{sec:InferenceDiscussion}
Our motivation is as follows: we observe a set of rational agents, e.g.,~pedestrians in a confined space, and want to learn about their interaction rules.
As argued above (e.g.~in Examples \ref{ex:FastReactionVsNewton} and \ref{ex:FastReactionVsNewtonII}), undisclosed fast-reaction entropic games can be used to parametrize a rich class of interaction behaviours, in particular subsuming Newtonian models.
We may therefore hope that, when feeding our observations into an inference functional for $J$, that by analysing and interpreting the resulting payoff function we can learn something about the interaction rules of the agents.
We focus on the inference of $J$ and assume that $U$, $e$ and $\veps$ are known. We will demonstrate during the numerical examples (Section \ref{sec:Numerics}) that it is usually possible to make plausible choices for $U$, $e$ and that qualitatively equivalent choices yield qualitatively equivalent results for the inferred $J$. One may also assume that $\veps=1$, since re-scaling $\veps$ can be compensated by re-scaling $J$ in the same way, see \eqref{eq:QuasiStaticGameModel}.

In Section \ref{sec:InferenceDifferential} we discuss a differential inference functional, proposed in analogy to \eqref{eq:NewtonEnergy}, for the full entropic game model, i.e.~without the undisclosed assumption and not in the fast-reaction limit.
For this we assume that we are able to observe the full dynamics, i.e.~physical locations and velocities of the agents, as well as their mixed strategies and temporal derivatives, $(x_i,\partial_t x_i,\sigma_i,\partial_t \sigma_i)_{i=1}^N$ and propose a functional that seeks to infer $J$ by comparing its predictions with the observed $\partial_t \sigma_i$.

In many cases it may not be possible to observe the mixed strategies $\sigma_i$, let alone their temporal derivatives $\partial_t \sigma_i$, since these may correspond to internal states of the agents that we, as external observers, cannot perceive.
In Section \ref{sec:InferenceNTFR} we turn to the undisclosed fast-reaction setting, where physical velocities and mixed strategies are functions of current physical locations, and show that there one can still perform inference with limited observations.
In Section \ref{sec:InferenceNTFRVel} we propose an inference functional, again in analogy to \eqref{eq:NewtonEnergy}, that only requires the observation of the physical locations $x_i$ and velocities $\partial_t x_i$. Although it is non-convex, we observe in our numerical examples that it provides plausible payoff functions $J$, indicating that non-convexity does not seem to be a practical issue.

In Section \ref{sec:InferenceNTFRSigma} we propose a convex inference functional that requires knowledge of the mixed strategies $\sigma_i$, but not of their temporal derivatives.
This may be practical in two cases:
In Section \ref{sec:NumericsSetup} we propose a scheme to provide such data from observed physical velocities, intuitively by `inverting' the approximation scheme of Example \ref{ex:FastReactionVsNewton}. Also, mixed strategies could represent a stochastic movement of agents and the `smooth' physical velocity $\partial_t x_i$ is merely the result of an averaging over a smaller time scale. Should we indeed be able to observe the fluctuating agent dynamics at small time scales, this could be used to approximate the mixed strategies. This was one of our motivations for providing Example \ref{ex:NTFRNoise}.

For both functionals there are natural candidates for the mean-field limit.
We analyse these two functionals in more detail in Section \ref{sec:InferenceNTFRTheory}, providing an estimate of the approximation quality (Theorem \ref{thm:BoundOnTrajectories}), as well as existence of minimizers and consistency in the mean-field limit (Theorem \ref{thm:Main}).

Now, we formalize our notion of admissible observations for inference and a notion of consistency of observations in the mean-field limit.
\begin{assumption}[Admissible observations]\label{ass:Observations}\hfill
\begin{enumerate}
	\item \emph{Discrete observation:} For a fixed number of agents $N$, a time horizon $T \in (0,\infty)$, and some radius $R \in (0,\infty)$ we observe the agents' physical paths $\bm{x}^N := (x^N_1,\ldots,x^N_N) \in C^1([0,T],B^d(R))^N$ with velocities $\bm{v}^N := (v^N_1,\ldots,v^N_N)$, $v^N_i = \partial_t x^N_i$.

Optionally, in addition we may also observe the agents' mixed strategies $\bm{\sigma}^N := (\sigma^N_1,\ldots,\allowbreak\sigma^N_N)\allowbreak \in C([0,T],\sigmaSet)^N$ for some bounds $0<a\leq b<\infty$ in the definition of $\sigmaSet$, see \eqref{eq:SigmaSet}. The mixed strategies are consistent with the observed velocities, i.e.
\begin{equation}
\label{eq:ObservedStrategiesConsistent}
v_i^N(t) = \int_U e(x_i^N(t),u)\sigma_i^N(t)(u)\,\diff \eta(u).
\end{equation}
\emph{Note:} The assumption that the agent's velocity is exactly consistent with their mixed strategy may seem unrealistic, as the velocity may be subject to other external influences and noise. However, we will subsequently often assume that the mixed strategies are not observed directly, but are only inferred from the velocities. In this case, satisfying assumption \eqref{eq:ObservedStrategiesConsistent} is quite natural. In other cases, assumption \eqref{eq:ObservedStrategiesConsistent} will be used in Theorem \ref{thm:BoundOnTrajectories} to bound the error on trajectories based on the error on strategies. If \eqref{eq:ObservedStrategiesConsistent} only holds approximately, an additional corresponding error term will appear there.
\label{item:ObservationsDiscrete}
\item \emph{Consistent mean-field behaviour:} For fixed $T$,$R$, for an increasing sequence of $N$ we make observations as specified in part \ref{item:ObservationsDiscrete}, and there is a limit observation $t \mapsto \mu^\infty(t) \in \prob(B^d(R))$ with a velocity field $t \mapsto v^\infty(t)$ such that
\begin{align}
	\label{eq:ObservationsMuUniform}
	W_1(\mu^N(t),\mu^\infty(t)) \rightarrow 0 \quad \tn{uniformly in $t \in [0,T]$ where}\quad \mu^N(t) \assign \frac1N \sum_{i=1}^N \delta_{x^N_i(t)},
\end{align}
\begin{align}
	\label{eq:ObservationsVWeak}
	\int_0^T \frac{1}{N} \sum_{i=1}^N \la \phi(t,x^N_i(t)), v^N_i(t) \ra \,\diff t
	\rightarrow \int_0^T \int_{\R^d} \la \phi(t,x) , v^\infty(t)(x) \ra \,\diff \mu^\infty(t)(x)\,\diff t
\end{align}
for all $\phi \in [C([0,T] \times \R^d)]^d$, and
\begin{align}
	\label{eq:ObservationsVConvex}
	\int_0^T \frac{1}{N} \sum_{i=1}^N \|v^N_i(t)\|^2 \,\diff t
	\rightarrow \int_0^T \int_{\R^d} \|v^\infty(t)(x)\|^2 \,\diff \mu^\infty(t)(x)\,\diff t<\infty.
\end{align}
\eqref{eq:ObservationsMuUniform} implies that physical locations are consistent, \eqref{eq:ObservationsVWeak} implies that observed velocities converge in a weak sense and \eqref{eq:ObservationsVConvex} implies that they are consistent. Intuitively, for \eqref{eq:ObservationsVConvex} to hold, mass particles that converge to the same limit location $x$ as $N \to \infty$ need to converge to the same velocity, otherwise Jensen's strict inequality contradicts \eqref{eq:ObservationsVConvex}.

If we also observe mixed strategies in part \ref{item:ObservationsDiscrete}, then the bounds $0<a \leq b<\infty$ are uniform in $N$ and there also is a mixed strategy mean-field $(t,x) \mapsto \sigma^\infty(t)(x) \in \sigmaSet$ such that
\begin{multline}
	\label{eq:ObservationsSigmaWeak}
	\int_0^T \frac{1}{N} \sum_{i=1}^N \int_U \phi(t,x^N_i(t),u) \cdot \sigma^N_i(t)(u)\, \diff \eta(u) \,\diff t \\
	\rightarrow \int_0^T \int_{\R^d} \int_U \phi(t,x,u) \cdot \sigma^\infty(t)(x)(u) \diff \eta(u)\,\diff \mu^\infty(t)(x)\,\diff t
\end{multline}
for all $\phi \in [C([0,T] \times \R^d \times U)]^d$ and
\begin{multline}
	\label{eq:ObservationsSigmaConvex}
	\int_0^T \frac{1}{N} \sum_{i=1}^N \int_U \sigma^N_i(t)(u)\log(\sigma^N_i(t)(u))\, \diff \eta(u) \,\diff t \\
	\rightarrow \int_0^T \int_{\R^d} \int_U \sigma^\infty(t)(x)(u)\log(\sigma^\infty(t)(x)(u)) \diff \eta(u)\,\diff \mu^\infty(t)(x)\,\diff t.
\end{multline}
These are direct equivalents of \eqref{eq:ObservationsVWeak} and \eqref{eq:ObservationsVConvex}.
\label{item:ObservationsLimit}
\end{enumerate}
\end{assumption}

In particular, observations are admissible when they were generated by an entropic game model with some ground-truth payoff $J$.
\begin{lemma}\label{lemma:ExplicitForward}
Let $J \in X$, $T > 0$, $\bar{R}>0$, $N \in \N$.
	\begin{enumerate}
	\item For initial locations $\bar{\bm{x}}^N=(\bar{x}^N_1,\ldots,\bar{x}^N_N) \in [B^d(\bar{R})]^N$ the induced solution to \eqref{eq:QuasiStaticGameModel} with corresponding velocities and mixed strategies provides an admissible discrete observation in the sense of Assumption \ref{ass:Observations}, part \ref{item:ObservationsDiscrete}, for $R=\bar{R}+\|e\|_\infty \cdot T$.

	\item Let $\bar{\mu} \in \prob(B^d(\bar{R}))$ and consider a sequence of empirical measures $(\bar{\mu}^N)_N$ in $\prob(B^d(\bar{R}))$ of the form
	\[
	\bar{\mu}^N = \frac1N \sum_{i=1}^N \delta_{\bar{x}_{i}^N}, \quad \bar{x}_{i}^N \in B^d(\bar{R}), \quad \text{such that } \lim_{N\to \infty}W_1(\bar{\mu}^N, \bar{\mu}) = 0.
	\]
	Then the solutions to \eqref{eq:QuasiStaticGameModel} with initial positions $(\bar{x}_1^N, \dots, \bar{x}_N^N)$ are a suitable sequence of discrete observations in the sense of Assumption \ref{ass:Observations}, part \ref{item:ObservationsLimit}, and the solution to \eqref{eq:ContEqR} provides a corresponding limit observation.
	\end{enumerate}
\end{lemma}
\noindent The proof is quite straightforward and provided in Section \ref{sec:InferenceProofs}.

In the following we assume that observations are admissible in the sense of Assumption \ref{ass:Observations}.

\subsection{Differential inference functional}
\label{sec:InferenceDifferential}
In this section we discuss an inference functional for payoff function $J$, in close analogy to \eqref{eq:NewtonEnergy}.
For now, we assume that in addition to $(\bm{x}^N, \bm{v}^N, \bm{\sigma}^N)$ we can even observe $\partial_t \sigma_i^N$ for all $i \in \{1,\ldots,N\}$. Our differential inference functional is therefore aimed at recovering $J$, by comparing the observed $\partial_t \sigma_i^N$ with the predictions by the model \eqref{eq:EntropicGameModelSigma}.
In \eqref{eq:NewtonEnergy} the discrepancy between observed $v_i^N$ and predicted velocities is penalized by the squared Euclidean distance.
As metric to compare the observed $\partial_t \sigma_i^N$ and the prediction by \eqref{eq:EntropicGameModelSigma} we choose the (weak) Riemannian tensor of the Hellinger--Kakutani and Fisher--Rao metrics. 
That is, we set for a base point $\sigma \in \sigmaSet$, and two tangent directions $\delta \mu, \delta \nu \in L^p_{\eta}(U)$ attached to it,
\begin{align}
	d_\sigma(\delta \mu,\delta \nu)^2 \assign \frac{1}{4} \int_U \frac{\left(\delta \mu- \delta \nu\right)^2}{\sigma}\,\diff \eta.
\end{align}
A potential inference functional for $J$ could then be:
\begin{align}
	\label{eq:GamesDifferentialEnergy}
	& \energy^{N}_{\dot{\sigma}}({J}) \assign \nonumber\\
	& \frac{1}{T} \int_0^T \left[
		\frac{1}{N} \sum_{i=1}^N d_{\sigma_i^N(t)}\left(
			\partial_t \sigma_i^N(t), \frac{1}{N} \sum_{j=1}^N f^{{J}}\big(x_i^N(t),\sigma_i^N(t),x_j^N(t),\sigma_j^N(t)\big)
			+f^\veps\big(\sigma_i^N(t)\big)
			 \right)^2
			\right] \diff t.
\end{align}
Minimization should be restricted to a sufficiently regular and compact class of ${J}$. Since $d_\sigma(\cdot,\cdot)^2$ is quadratic in its arguments and $f^{{J}}$ is linear in ${J}$, the mismatch term is again quadratic in ${J}$ and thus, restricted to a suitable class, this is a convex optimization problem.

In Section \ref{sec:NumericsDifferential} we provide a simple example where we simulate trajectories $(\bm{x}^N,\bm{\sigma}^N)$ according to the dynamics \eqref{eq:EntropicGameModel}, based on a given $J$, and then obtain a corresponding minimizer $\hat{J}$ of \eqref{eq:GamesDifferentialEnergy}. We then show that trajectories simulated with $J$ and $\hat{J}$ are close even when the initial configurations are not drawn from training data used to infer $\hat{J}$. Indeed, upon selecting a suitable functional framework, one can in principle extend the analysis in \cite{BFHM16} (and the analysis presented here below) to prove existence of minimizers, limiting behaviour as $N \to \infty$ and ability of the inferred model to generalize.

However, application to real data seems challenging, as one usually only observes $(\bm{x}^N, \bm{v}^N)$ but not $\bm{\sigma}^N$, let alone variations $\partial_t \sigma_i^N$.
To address this, we propose in the next section two new inference functionals for the undisclosed fast-reaction setting.
		
\subsection{Undisclosed fast-reaction inference functionals}
\label{sec:InferenceNTFR}

\subsubsection{Penalty on velocities}
\label{sec:InferenceNTFRVel}
The inference situation is simplified in the undisclosed fast-reaction setting since now, for fixed payoff function $J$, mixed strategies and physical velocities are a direct function of physical locations. We can therefore attempt to minimize the discrepancy between observed physical velocities and those predicted by $J$.

For a discrete admissible observation $(\bm{x}^N,\bm{v}^N)$, as in Assumption \ref{ass:Observations}, we define the inference functional on hypothetical payoff functions ${J} \in X$
\begin{subequations}
	\label{eq:energyxdot}
	\begin{align}
	\energy^{N}_{v}({J}) & \assign
		\frac{1}{T} \int_0^T \left[
			\frac{1}{N} \sum_{i=1}^N \left\|v_i^N(t)-v_i^{{J}}(x_1^N(t),\ldots,x_N^N(t)) \right\|^2 
		\right] \diff t. \label{eq:energyJNxdot}
	\intertext{The natural candidate for the limit functional is}
	\energy_v({J}) & \assign
		\frac{1}{T} \int_0^T \int_{\R^d} \left\| v^\infty(t)(x) - v^{{J}}(\mu^\infty(t))(x) \right\|^2 \diff \mu^\infty(t)(x) \diff t
		\label{eq:energyJxdot}
	\end{align}
\end{subequations}
where the couple $(\mu^\infty, v^\infty)$ is the corresponding limit of $(\bm{x}^N,\bm{v}^N)_N$ as introduced in Assumption \ref{ass:Observations}, point \ref{item:ObservationsLimit}.

This functional is intuitive and requires only the observations of $(\bm{x}^N,\bm{v}^N)$ and no observations about the mixed strategies. However, it is not convex. This could potentially lead to poor local minima, but we did not observe such problems in our numerical examples.

The motivation for studying mean-field inference functionals is twofold. First, it establishes asymptotic consistency of inference in the limit of many agents. Second, the mean-field equation yields an approximation for the expected behaviour of the finite-agent system under many repetitions with random initial conditions.
While we do not prove this approximation property the existence of a consistent mean-field inference functional is still an encouraging indicator that inference over the collection of many finite-agent samples will yield a reasonable result.

\subsubsection{Penalty on mixed strategies}
\label{sec:InferenceNTFRSigma}
In the case where information about the mixed strategies of the agents is available (see Section \ref{sec:InferenceDiscussion} for a discussion), the discrepancy between observed mixed strategies and those predicted by $J$ can be used for inference. We choose to measure this discrepancy with the Kullback--Leibler (KL) divergence.

Similar to above, for a discrete admissible observation $(\mathbf{x}^N,\mathbf{v}^N,\bm{\sigma}^N)$, we define the inference functional for hypothetical payoff functions ${J} \in X$
\begin{subequations}
	\label{eq:energysigma}
	\begin{align}
	\energy^{N}_\sigma({J}) & \assign 
		\frac{1}{T} \int_0^T \left[
			\frac{1}{N} \sum_{i=1}^N \KL\big(\sigma_i^N(t)|\sigma_i^{J}(x_1^N(t),\ldots,x_N^N(t))\big)
		\right] \diff t.
		\label{eq:energyJNsigma}
	\intertext{The natural candidate for the limit functional is}
	\energy_\sigma({J}) & \assign
		\frac{1}{T} \int_0^T \int_{\R^d}
		\KL\big(\sigma^\infty(t)(x,\cdot)|\sigma^{{J}}(\mu^\infty(t))(x,\cdot)\big)
		\diff\mu^\infty(t)(x)\, \diff t
		\nonumber \\
	& = \frac1T \int_0^T \int_{\R^d} \int_U
		\log\left( \frac{\sigma^\infty(t)(x,u)}{\sigma^{{J}}(\mu^\infty(t))(x,u)} \right)
		\sigma^\infty(t)(x,u)
		\diff \eta(u) \diff \mu^\infty(t)(x)\,\diff t.
		\label{eq:energyJsigma}
	\end{align}
\end{subequations}
where $(\mu^\infty, \sigma^\infty)$ are as introduced in Assumption \ref{ass:Observations}, point \ref{item:ObservationsLimit}.
Due to the particular structure of stationary mixed strategies, \eqref{eq:QuasiStaticGameModelSigma}, and the KL divergence, functionals \eqref{eq:energysigma} are convex on $X$ (Proposition \ref{prop:energysigmaConvex}), which is a potential advantage over \eqref{eq:energyxdot}.

\subsubsection{Analysis of the inference functionals}
\label{sec:InferenceNTFRTheory}

We now provide some theoretical results about the inference functionals of Sections \ref{sec:InferenceNTFRVel} and \ref{sec:InferenceNTFRSigma}.
The first result establishes that the obtained minimal inference functional value provides an upper bound on the accuracy of the trajectories that are simulated with the inferred $\hat{J}$ (the proof builds upon the proof of \cite[Proposition 1.1]{BFHM16} and is given in Section \ref{sec:InferenceProofs}).
\begin{theorem}\label{thm:BoundOnTrajectories}
	Let $\hat{J} \in X$ and $(\mathbf{x}^N,\mathbf{v}^N,\bm{\sigma}^N)$ be an admissible discrete observation for $N$ agents (cf. Assumption \ref{ass:Observations}, point \ref{item:ObservationsDiscrete}). Let $\hat{\mathbf{x}}^N$ be the solution of \eqref{eq:QuasiStaticGameModel} induced by $\hat{J}$ for the initial condition $\hat{\bm{x}}^N(0) = \bm{x}^N(0) = (\bar{x}_{1}^N,\dots,\bar{x}_{N}^N) \in [\R^d]^N$. Then,
	\[
	\|\mathbf{x}^N(t)-\hat{\mathbf{x}}^N(t)\|_N \leq C \sqrt{\energy^{N}_v(\hat{J})}
	\quad \text{and} \quad
	\|\mathbf{x}^N(t)-\hat{\mathbf{x}}^N(t)\|_N \leq C \sqrt{\energy^{N}_\sigma(\hat{J})}
	\]
	for all $t \in [0,T]$, with $C = C(T, \hat{J}, e, \veps)$.
	Analogously, let $(\mu^\infty, \sigma^\infty)$ as introduced in Assumption \ref{ass:Observations}, point \ref{item:ObservationsLimit} and let $\hat{\mu}$ be the solution of \eqref{eq:ContEqR} induced by $\hat{J}$ for the initial condition $\bar{\mu} = \mu^\infty(0)$. Then,
	\[
	W_1(\mu^\infty(t),\hat{\mu}(t)) \leq C \sqrt{\energy_v(\hat{J})}
	\quad \text{and} \quad
	W_1(\mu^\infty(t),\hat{\mu}(t)) \leq C \sqrt{\energy_\sigma(\hat{J})}
	\]
	for all $t \in [0,T]$ and the same constant $C$ as above.
\end{theorem}

Next, we address the existence of minimizing payoff functions, both in theory and numerical approximation.
At the theoretical level, we need to ensure compactness of minimizing sequences, which we obtain here by restriction to a suitable compact space.
At the numerical level, we are interested in finite-dimensional approximations of this space that asymptotically are dense as the discretization is refined.

\begin{remark}[Compactness and finite-dimensional approximation]
\label{rem:CompactX}
In order to obtain compactness we restrict the variational problems to suitable subspaces of $X$. In particular, for $R, M > 0$, let us define
\[
X_{R,M} = \{ J \in \Lip_b(\dball{R} \times U \times \dball{R}) \mid \|J\|_{\infty} + \Lip(J) \leq M \}.
\]
The parameter $R$ bounds the learning domain in space: inference where we have no data available is simply meaningless. The parameter $M$ will ensure compactness with respect to uniform convergence.

For each $R,M > 0$ we consider a family of closed convex subsets $(X_{R,M}^N)_{N \in \mathbb{N}} \subset X_{R,M}$ satisfying the uniform approximation property: for every $J \in X_{R,M}$ there exists a sequence $(J^N)_{N \in \mathbb{N}} $ uniformly converging to $J$ with $J^N \in X_{R,M}^N$ for each $N \in \mathbb{N}$. These approximating spaces can be selected to be finite dimensional.
\end{remark}

An alternative way to obtain compactness would be via the introduction of additional regularizing functionals. This is discussed in Section \ref{sec:NumericsSetup}. Conceptually, both approaches serve the same purpose and we find that the former is more convenient for analysis while the latter is more attractive numerically.

The following key lemma ensures convergence of the error functionals under uniform convergence of the payoff functions and/or $W_1$ convergence of measures.

\begin{lemma}
	\label{lemma:UniformConvergence}
	Let $(\mathbf{x}^N,\mathbf{v}^N,\bm{\sigma}^N)_N$ be a family of admissible observations and $\energy^{N}_v$, $\energy_v$, $\energy^{N}_\sigma$ and $\energy_\sigma$ be the corresponding discrete and continuous inference functionals \eqref{eq:energyxdot} and \eqref{eq:energysigma}.	For $M>0$, let $({J}^N)_N$ be a sequence of payoff functions in $X_{R,M}$, converging uniformly to some ${J} \in X_{R,M}$ as $N \to \infty$.
	Then,
	\begin{align}
	\lim_{N\to\infty} \energy^{N}_v({J}^N) = \energy_v({J}) \quad \text{ and } \quad \lim_{N\to\infty} \energy^{N}_\sigma({J}^N) = \energy_\sigma({J}).
	\end{align}
\end{lemma}
\begin{proof}	
	Thanks to Assumption \ref{ass:Observations}, point \ref{item:ObservationsLimit}, we have $W_1(\mu^N(t),\mu^\infty(t)) \to 0$ uniformly in $t \in [0,T]$ as $N \to \infty$. Lemma \ref{lemma:UniformSigma} provides
	\begin{align*}
	\left\| v^{{J}^N}(\mu^N(t))(x) - v^{{J}}(\mu^\infty(t))(x) \right\| &
	\leq C \cdot \left( W_1(\mu^N(t),\mu^\infty(t)) + \|\hat{J}^N-\hat{J}\| \right)
	\end{align*}
	uniformly in $x \in \dball{R}, t \in [0,T]$ and $N$.
	Hence, $v^{{J}^N}(\mu^N(t))(x)$ converges uniformly to $v^J(\mu^\infty(t))(x)$ in $C([0,T]\times \dball{R}; \R^d)$. Since $\|v^J(\mu^\infty(t))(x)\|\allowbreak \leq \|e\|_\infty$ for every $x \in \dball{R}$ and $t \in [0,T]$, we also have $\|v^{J^N}(\mu^N(t))(x) \|^2 \to \|v^J(\mu^\infty(t))(x)\|^2$ uniformly in $C([0,T]\times \dball{R})$. Recalling that $v_i^{{J}^N}(x_1^N(t),\ldots,x_N^N(t))(x) = v^{{J}^N}(\mu^N(t))(x_i^N(t))$ by definition and using the convergence of velocities provided by Assumption \ref{ass:Observations}, point \ref{item:ObservationsLimit}, the result follows by passing to the limit in the definition:
	\begin{align*}
	\lim_{N \to \infty} \energy^{N}_v({J}^N) &= \lim_{N \to \infty} \frac{1}{T} \int_0^T \left[
	\frac{1}{N} \sum_{i=1}^N \left\|v_i^N(t)-v_i^{{J}^N}(x_1^N(t),\ldots,x_N^N(t)) \right\|^2 
	\right] \diff t \\
	&= \frac{1}{T} \int_0^T \int_{\R^d} \left\|v^\infty(\mu^\infty(t))(x) - v^{{J}}(\mu^\infty(t))(x)\right\|^2 \diff \mu^\infty(t)(x) \diff t = \energy_v({J}).
	\end{align*}
	An analogous derivation applies to the $\sigma$-energies \eqref{eq:energysigma} where we use the Lipschitz estimate \eqref{eq:SigmaLip} and the $\sigma$-part of Assumption \ref{ass:Observations}, point \ref{item:ObservationsLimit}.
\end{proof}

We are now ready to state the main result concerning the inference functionals and convergence of minimizers for $\energy_v^{N}$ and $\energy_\sigma^{N}$.

\begin{theorem}\label{thm:Main}
	Let $(\mathbf{x}^N,\mathbf{v}^N,\bm{\sigma}^N)_N$ be a family of admissible observations and fix $M > 0$. Then:
	\begin{enumerate}
	\item The functionals \eqref{eq:energyJNxdot} and \eqref{eq:energyJNsigma} have minimizers over $X_{R,M}^N$ and \eqref{eq:energyJxdot} and \eqref{eq:energyJsigma} have minimizers over $X_{R,M}$.
	\item For $N > 0$, let
	\[
		{J}^N_v \in \argmin_{{J} \in X_{R,M}^N} \energy^{N}_v({J}) \quad \text{and} \quad {J}^N_\sigma \in \argmin_{{J} \in X_{R,M}^N} \energy^{N}_\sigma({J}).
	\]
	The sequence $({J}^N_v)_N$, resp. $({J}^N_\sigma)_N$, has a subsequence converging uniformly to some
	continuous function ${J}_v \in X_{R,M}$, resp. ${J}_\sigma \in X_{R,M}$, such that
	\begin{align}
		\label{eq:MainLimitMinimizers}
		\lim_{N \to \infty} \energy^{N}_v({J}^N_v) = \energy_v({J}_v) \quad \text{and} \quad \lim_{N \to \infty} \energy^{N}_\sigma({J}^N_\sigma) = \energy_\sigma({J}_\sigma).
	\end{align}
	In particular, ${J}_v$ (resp. ${J}_\sigma$) is a minimizer of $\energy_v^J$ (resp. $\energy_\sigma^J$) over $X_{R,M}$.
	\end{enumerate}
\end{theorem}

\begin{proof}\hfill \\
	\emph{Part 1: existence of minimizers.}
	All functionals \eqref{eq:energyxdot} and \eqref{eq:energysigma} are continuous in ${J}$ with respect to uniform convergence which follows from Lemma \ref{lemma:UniformConvergence} setting the initial measure $\bar{\mu}$ to be the empirical one (i.e.~just looking at converging ${J}$ for fixed trajectories).
	The sets $X_{R,M}$ and $X_{R,M}^N$ are compact with respect to uniform convergence. Therefore, minimizers exist.
	
	\noindent
	\emph{Part 2: convergence of minimizers.}
	The proof is inspired by \cite[Section 4.2]{BFHM16}.
	Since $X_{R,M}^N \subset X_{R,M}$ and the latter is compact, the sequence of minimizers $({J}_v^N)_N$ with ${J}^N_v \in X_{R,M}^N$ must have a cluster point ${J}_v$ in $X_{R,M}$. For simplicity, denote the converging subsequence again by $({J}^N_v)_N$. Let $\tilde{J} \in X_{M,R}$. By the uniform approximation property of $X_{R,M}^N$ let $(\tilde{J}^N)_N$ be a sequence converging uniformly to $\tilde{J}$ and $\tilde{J}^N \in X_{R,M}^N$. By Lemma \ref{lemma:UniformConvergence} and optimality of each ${J}^N_v$,
	\begin{align*}
	\energy_v({J}_v) = \lim_{N\to\infty} \energy^{N}_v({J}^N_v) \leq \lim_{N\to\infty} \energy^{N}_v(\tilde{J}^N) = \energy_v(\tilde{J}).
	\end{align*}
	Therefore, ${J}_v$ minimizes $\energy_v$ and the optimal values converge. The same argument covers the $\sigma$-functionals \eqref{eq:energysigma} and the corresponding sequence of minimizers $({J}_\sigma^N)_N$.
\end{proof}

\medskip
\noindent
We conclude this section with a couple of comments and remarks on the interpretation of results in Theorem \ref{thm:Main}.

\begin{remark}[Realizability of data]
In general, observed data may or may not be realizable by the model. This distinguishes two regimes.
\begin{itemize}
	\item Assume there is a \emph{true} regular $J$ that generated the observed data (as for example in Lemma \ref{lemma:ExplicitForward}). If we pick $M$ large enough in Theorem \ref{thm:Main}, so that $J \in X_{R,M}$, one has $\energy_v({J}_v) \leq \energy_v({J}) = 0$ and analogously $\energy_v({J}_\sigma) \leq \energy_\sigma({J}) = 0$. Hence, minimizers of $\energy_v$ and $\energy_\sigma$ over $X_{R,M}$ reproduce exactly the trajectories of the system thanks to Theorem \ref{thm:BoundOnTrajectories}.
	\item Assume there is no \emph{true} regular $J$ that generated the observed data. This means that, no matter how large $M$ is taken in Theorem \ref{thm:Main}, the minimal limit energies $\energy_\sigma({J}_\sigma)$ and $\energy_v({J}_v)$ may not be equal to $0$. However, the bound on the trajectories generated by the inferred model, Theorem \ref{thm:BoundOnTrajectories}, still holds and so the remaining residual values $\energy_\sigma({J}_\sigma)$ and $\energy_v({J}_v)$ are then an indication for how well our parametrized family of interaction models was able to capture the structure in the observed data.
\end{itemize}
\end{remark}

\begin{remark}[Non-uniqueness of minimizers]\label{rem:NonUniqueness}
	We point out that the parametrization of an entropic game model by a payoff function $J$ is not unique.
	Replacing $J$ by $\hat{J}(x,u,x',u') \assign J(x,u,x',u') + g(x,x',u')$ yields the same agent trajectories for arbitrary $g : \Omega \times \Omega \times U$. This can be seen from the fact that $f^{J+g}=f^J$ for $f^J$ and $f^{J+g}$ as given in \eqref{eq:fJ}.
	Intuitively, given some $x$, $x'$ and $u'$, which strategy $u$ is most attractive does not change, if we add a constant benefit for all potential choices.
	The analogous observation holds for the undisclosed case.
	This implies that minimizers of \eqref{eq:GamesDifferentialEnergy}, \eqref{eq:energyxdot} and \eqref{eq:energysigma} are not unique.
	This has to be taken into account when we try to interpret an inferred ${J}$.
	In our numerical examples with the undisclosed fast-reaction model we usually made additional structural assumptions on $J$ (see Section \ref{sec:NumericsSetup}) and as a consequence did not observe issues with non-unique minimizers. The phenomenon is encountered in a proof-of-concept example on the differential inference functional, Section \ref{sec:NumericsDifferential}.
\end{remark}

\section{Numerical experiments}
\label{sec:Numerics}
In this section we present some numerical examples for the entropic game models of Section \ref{sec:Model} and their inference with the functionals of Section \ref{sec:Inference} to provide a basic proof-of-concept and some intuition about their behaviour.

\subsection{Numerical setup}
\label{sec:NumericsSetup}
\paragraph{Preliminaries and discretization.}
We assume for simplicity that the strategy space $U$ is finite, i.e.~$U=\{u_1,\ldots,u_K\}$ for some $K>0$ and fix the reference measure $\eta = \sum_{k=1}^K \frac{1}{K}\delta_{u_k}$ as the normalized uniform measure over $U$.

For a system of $N$ agents starting at positions $x_1(0), \dots, x_N(0)$, we observe $S$ snapshots of the evolution at times $t_s = s\cdot \frac{T}{S}$ for $s = 1,\dots,S$, so that we are given locations $\{x_i(t_s)\}_{s,i}$, velocities $\{\partial_t x_i(t_s)\}_{s,i}$ and (sometimes) mixed strategies $\{\sigma_i(t_s)\}_{s,i}$ where indices $s$ and $i$ run from $1$ to $S$ and $1$ to $N$ respectively. Here $\sigma_{i}(t_s) = (\sigma_{i,1}(t_s),\dots,\sigma_{i,K}(t_s))$ is a discrete probability density with respect to $\eta$.

Let $\Omega = \prod_{i=1}^d [\ell_i, r_i]$, $\ell_i,r_i \in \R$, be the smallest hypercube containing all observed locations. We discretize $\Omega$ by a regular Cartesian grid and describe the unknown payoff $J$ by its values $J(\cdot,u_k,\cdot)$ at grid points for $u_k \in U$ (or $J(\cdot,u_k,\cdot,u_{k'})$ in the fully general setting). Between grid points, $J$ is extended by coordinate-wise linear interpolation, i.e., we consider piecewise $d$-linear finite elements over hypercubes.

Within this setting, the inference functional \eqref{eq:energysigma} reduces to the discrete error functional
\begin{align}
\label{eq:energysigmaDiscrete}
\energy_\sigma^N(J) = \frac{1}{SNK} \sum_{s=0}^S \sum_{i=1}^N \sum_{k=1}^K
\log\left( \frac{\sigma_{i,k}(t_s)}{\sigma^{J}_{i,k}(t_s)} \right)
\sigma_{i,k}(t_s)
\end{align}
where, for $s \in \{0,\ldots,S\}$, $i \in \{1,\dots,N\}$ and $k \in \{1,\dots,K\}$, the optimal mixed strategy at time $t_s$ is given by \eqref{eq:QuasiStaticGameModelSigma} as
\[
\sigma^{J}_{i,k}(t_s) = \sigma^{J}_{i}(x_1(t_s),\dots,x_N(t_s))(u_k) = \frac{\exp\left(\tfrac{1}{\veps N}\sum_{j=1}^N J(x_i(t_s),u_k,x_j(t_s))\right)}{\frac{1}{K}\sum_{k'=1}^K \exp\left(\tfrac{1}{\veps N}\sum_{j=1}^N J(x_i(t_s),u_{k'},x_j(t_s))\right)}
\]
Similarly, the inference functional \eqref{eq:energyxdot} reduces to the discrete functional
\begin{align}
\label{eq:energyxdotDiscrete}
\energy_v^N(J) = \frac{1}{SN} \sum_{s=0}^S \sum_{i=1}^N \left\| v_i^J(t_s) - \partial_t x_i(t_s)\right\|^2
\end{align}
where
\[
v_i^J(t_s) = v^{J}_{i}(x_1(t_s),\dots,x_N(t_s)) = \frac{1}{K}\sum_{k=1}^K e(x_i(t_s), u_k) \sigma_{i,k}^J(t_s).
\]

\paragraph{Reducing dimensionality.}
When $d \geq 2$, describing a general function $J : \Omega \times U \times \Omega \to \R$ where $\Omega \subset \R^d$ (or even $\Omega \times U \times \Omega \times U \to \R$ in the fully general setting) requires many degrees of freedom (one value for each node on the grid) and inferring them from data in a reliable way would require a large set of observations.
After inference, interpreting such a general function to understand the interactions between agents can be a daunting task.
For these two reasons we may wish to incorporate some prior knowledge on the structure of $J$ via a suitable ansatz. A prototypical example is
\begin{align}
\label{eq:JSplitting}
J(x,u,x')=J_1(x,u) + J_2(x'-x,u)
\end{align}
where the first term models how a single agent prefers to move by itself depending on its absolute location, and the second term models a translation invariant pairwise interaction with other agents. This reduces the dimensionality from one function over $\R^d \times U \times \R^d$ to two functions over $\R^d \times U$. Further simplifications can be made, for instance, by assuming that the interaction is also rotation invariant. Such parametrizations are very common in Newtonian models, see for instance \cite{BFHM16}, where radially symmetric interaction functions are assumed.
Unless noted otherwise, we will in the following use the ansatz \eqref{eq:JSplitting} in our experiments.

\paragraph{Regularization and dropping the Lipschitz constraints.}
Even after the reduction \eqref{eq:JSplitting} inferring the coefficients of $J$ is a high-dimensional non-linear inverse problem. We may observe no (or very few) agents near some grid points and thus have no reliable data to set the corresponding coefficients of $J$ directly.
A common approach to avoiding ill-posedness is to add regularization to the minimization problem.
Intuitively, this will diffuse information about observations over the grid to some extent.
For regularization in $\Omega$ we use the squared $L^2$-norm of the gradient of $J$, which boils down to weighted sums of the finite differences between grid points. For the strategy space $U$ one can in principle do the same, although this was not necessary in our examples.
When $J$ is split as in \eqref{eq:JSplitting}, a regularizer can be applied separately to each term.
The regularized version of \eqref{eq:energysigmaDiscrete} then becomes
\begin{align}
\label{eq:energysigmaDiscreteReg}
\energy_\sigma^N(J) + \lambda_1 \cdot \mathcal{R}_1(J_1) + \lambda_2 \cdot \mathcal{R}_2(J_2).
\end{align}
Of course, over-regularization will lead to loss of contrast and high-frequency features of $J$.
The effect of under-regularization is illustrated in Example \ref{ex:NumericsUnderreg}.

While in Theorem \ref{thm:Main} we consider increasingly finer discretizations $X_{R,M}^N$, $N \to \infty$, of the payoff function space $X_{R,M}$, in our numerical examples a fixed resolution is entirely sufficient (approximately 30 grid points along any spatial dimension). On a fixed grid the Lipschitz and $L^2$-gradient semi-norms are equivalent.
Consequently, for simplicity, in our numerical experiments we may drop the explicit constraint on the Lipschitz constant of $J$ and only use the regularization term to impose regularity.
This has the advantage that our finite-dimensional discrete problem is then unconstrained and we may approach it with quasi-Newton methods such as L-BFGS, for which we use the Julia package \textup{Optim} \cite{Optim}.

\paragraph{Inferring $\sigma_i(t_s)$ from observations.}
In many situations we are not able to observe the mixed strategies of agents, but merely their locations and velocities.
We can then decide to use the non-convex velocity based error functional \eqref{eq:energyxdotDiscrete}.
Alternatively, in case we want to optimize the convex inference functional $\energy_\sigma^N$ \eqref{eq:energysigmaDiscrete} instead, we can try to infer some meaningful $\sigma_i(t_s)$ from the observed velocities.
In both cases we need to make suitable choices for $U$ and $e$.
Unless stated otherwise, we pick $U$ as a discrete subset of $\R^d$ such that observed velocities lie in the (interior of the) convex hull of $U$ and pick $e(x,u) \assign u$. This means that every observed velocity can be reproduced by some mixed strategy.
As a heuristic method for inferring a viable mixed strategy $\sigma$ from a velocity $v$ we propose to solve the following minimization problem:
\begin{align}
\label{eq:ReconstructSigma}
\arg \min\left\{
\int_U \left[ \veps \cdot \log(\sigma(u)) + \|u-v\|^2 \right]\,\sigma(u)\,\diff \eta(u)
\middle|
\sigma \in \sigmaSet, \,
\int_U u \cdot \sigma(u)\,\diff \eta(u)=v
\right\}
\end{align}
The constraint enforces that $\sigma$ reproduces the observed velocity $v$. The term $\|u-v\|^2\,\sigma(u)$ encourages $\sigma$ to be concentrated in the vicinity of $v$ in $U$, whereas the entropic term $\veps \cdot \log(\sigma(u))\,\sigma(u)$ keeps $\sigma$ from collapsing onto a single point.
It is easy to see that minimizers of \eqref{eq:ReconstructSigma} are of the form $\sigma(u)=A \cdot \exp(-\|u-\tilde{v}\|^2/\veps)$ where $A$ is a normalization factor and $\tilde{v}$ is a velocity close to $v$, potentially slightly shifted to make sure that the constraint $\int_U u \cdot \sigma(u)\,\diff \eta(u)=v$ is satisfied.
\eqref{eq:ReconstructSigma} is readily solved by standard methods.
We consider this approach a heuristic `inverse' of the approximation of Newtonian models by game models, as outlined in Example \ref{ex:FastReactionVsNewton}.

\paragraph{Observations from forward models: subsampling and multiple realizations.}
In all the examples here below, observations are obtained by means of simulating a forward-model with explicit Euler-stepping with time-step size $\Delta t=0.02$. Since the numerical complexity of our inference functionals grows linearly with the number of observations, we usually choose to sub-sample the trajectories in time, keeping only every $\delta$-th step, since intuitively the information contained in subsequent steps is largely redundant as agents are confronted with a very similar configuration. Unless noted otherwise, we set $\delta=2$.

The quality of the reconstructed model is highly related to the data quality: more observations in a given region imply a better reconstruction of the payoff function in that specific region. Generally, if we consider as inputs only data coming from a single realization of the system we have highly correlated observations, which tend to concentrate only on a small subset of the reconstruction domain. Vaster regions of the domain are then explored taking $N \to \infty$. However, another option to enhance the quality of the reconstructed model is to combine data from multiple realizations with different initial conditions, i.e., we fix the number of agents $N$ and simulate the system for multiple different initial conditions. This provides a good coverage of the domain. Extending the functionals in Section \ref{sec:Inference} to multiple observations is straight-forward.

\subsection{Learning from undisclosed fast-reaction game models}
\label{sec:NumericsNTFR}
We now begin with relatively simple numerical examples: an explicit game model is fixed (including $U$, $e$, $J$), simulations according to this model are performed and subsequently the payoff function is inferred from the observations. By possibly re-scaling $J$ we may w.l.o.g.~set $\veps=1$ for all numerical experiments, see \eqref{eq:QuasiStaticGameModel}.
In this section we focus on the undisclosed fast-reaction setting.
A proof of concept of inference for the full entropic game models via the differential functional of Section \ref{sec:InferenceDifferential} is discussed in Section \ref{sec:NumericsDifferential}.

\begin{figure}
	\begin{subfigure}{.49\textwidth}
		\centering
		\includegraphics{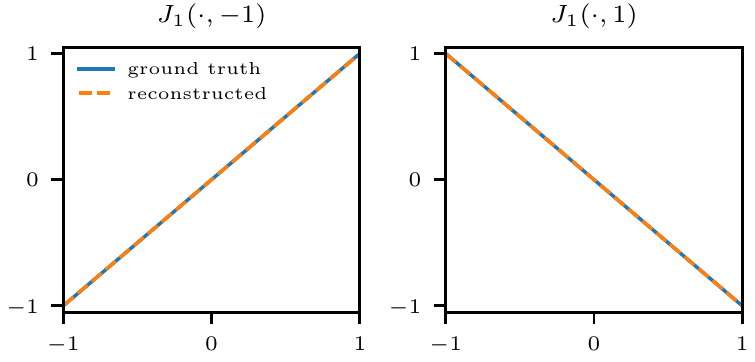}
		\caption{}
		\label{fig:03ntfr1dexactdataJself}
	\end{subfigure}
	\begin{subfigure}{.49\textwidth}
		\centering
		\includegraphics{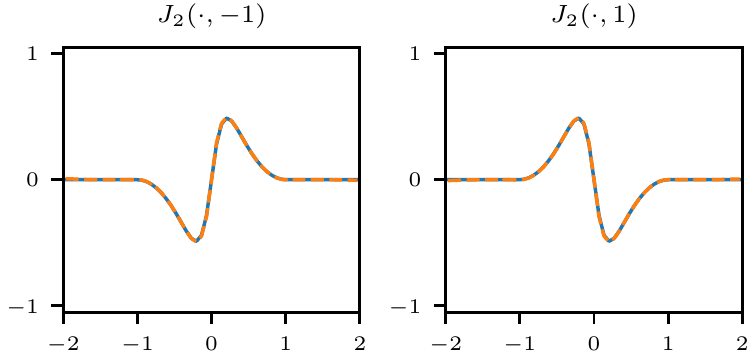}
		\caption{}
		\label{fig:03ntfr1dexactdataJdist}
	\end{subfigure}\\
	\begin{subfigure}{.49\textwidth}
		\centering
		\includegraphics{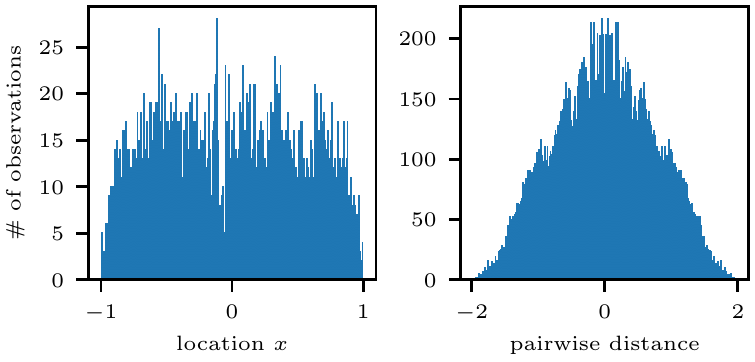}  
		\caption{}
		\label{fig:03ntfr1dexactdatahistogram}
	\end{subfigure}
	\begin{subfigure}{.49\textwidth}
		\centering
		\includegraphics{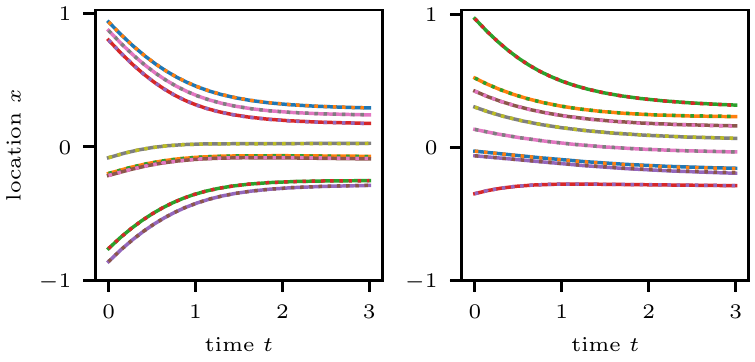}
		\caption{}
		\label{fig:03ntfr1dexactdatatruevslearn}
	\end{subfigure}
	\caption{Inference from $1$d observations of an undisclosed fast-reaction system. %
	(\subref{fig:03ntfr1dexactdataJself},\subref{fig:03ntfr1dexactdataJdist}): reconstruction of the payoff function. %
	(\subref{fig:03ntfr1dexactdatahistogram}): distribution of the data in the training set. %
	(\subref{fig:03ntfr1dexactdatatruevslearn}): comparison between true trajectories (solid lines) and trajectories generated by the inferred model (dashed lines) for two new realizations (not part of the training data).}
	\label{fig:03ntfr1d}
\end{figure}

\begin{example}[Game model, $d=1$]
	\label{ex:NTFRd1}
	Let $d=1$, $U=\{-1,+1\}$ and $e(x,u)=u$. We set
	\begin{align}\label{eq:Jdefinition}
		J(x,u,x') & \assign J_1(x,u) + J_2(x'-x,u), \qquad
		J_1(x,u) \assign -u \cdot x, \nonumber \\
		J_2(\Delta x,u) & \assign -u \cdot \tanh(5\Delta x) \cdot \left(\max\{1-|\Delta x|^2,0\}\right)^2.
	\end{align}
	$J_1$ encourages agents to move towards the origin $x=0$, e.g.~for $x>0$ one has $J_1(x,-1)>J_1(x,+1)$, thus preferring strategy (and velocity) $u=-1$ over $u=+1$.
	$J_2$ implements repulsion between nearby agents: when $\Delta x \in (0,1)$, i.e.~the `other agent' is to the right of `me' and relatively close, then $J_2(\Delta x,-1)>J_2(\Delta x,+1)$, and I am encouraged to move left.
	
	We set $N=8$, simulated 100 instances over a time span of $0.2$, collecting 5 data points per instance, i.e.,~a total of 500 observed configurations (consisting of locations and mixed strategies). In each instance, initial locations are sampled uniformly from $[-1,1]$. Observed relative locations between two agents are thus distributed over $[-2,2]$.
	Describing the discrete $J$ required 178 coefficients.
	Since observed mixed strategies are available in this example, we use the energy \eqref{eq:energysigmaDiscrete} for inference augmented with a regularizer as discussed in Section \ref{sec:NumericsSetup}, with $\lambda_{1}=\lambda_2=10^{-6}$, see \eqref{eq:energysigmaDiscreteReg}.
	The results are illustrated in Figure \ref{fig:03ntfr1d}.
	We find that, as intended, the model describes agents moving towards the origin while avoiding getting too close to each other.
	The functions $J_1$ and $J_2$ are accurately recovered from the data. For newly generated initial conditions, i.e.,~not taken from the training data, the trajectories simulated with the inferred $J$ are in excellent agreement with the underlying true model, demonstrating that the inferred model is able to generalize.
\end{example}

\begin{figure}[tb]
	\centering
	\includegraphics{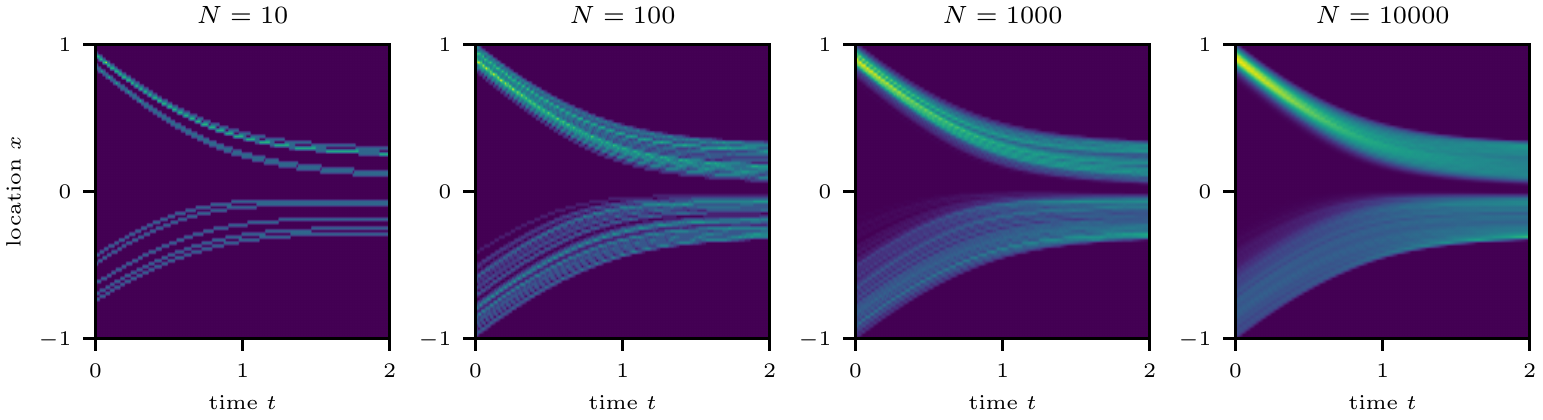}  
	\caption{Numerical approximation of the mean field limit. The forward model of Example \ref{ex:NTFRd1} is simulated for an increasing number of agents. Each panel shows a discrete histogram of the particle distribution over time. As $N$ increases the histograms approach a consistent limit.}
	\label{fig:meanfieldlimit}
\end{figure}

\begin{example}[Towards the mean-field regime]
	We can also explore the mean-field regime numerically. Figure \ref{fig:meanfieldlimit} shows trajectories for the same model as in the previous example for an increasing number of particles where the initial locations are sampled from some underlying distribution. As anticipated the behaviour approaches a consistent limit.

	Inference can also be performed for larger numbers of particles.
	We applied the same approach as in the previous example with $N=100$ agents. Since we are dealing with a large number of agents, already a small number of configurations carries enough information for the inference. Thus, we simulate 10 instances over a time span of $0.02$ with time-step size $\Delta t = 0.002$ , collecting 2 data points per instance, i.e.,~a total of 20 observed configurations (for locations and mixed strategies). In each instance, initial locations are sampled uniformly from $[-1,1]$. The inferred payoff function is essentially identical to the one obtained in Figure \ref{fig:03ntfr1d}.
\end{example}

\begin{figure}
	\begin{subfigure}{.49\textwidth}
		\centering
		\includegraphics{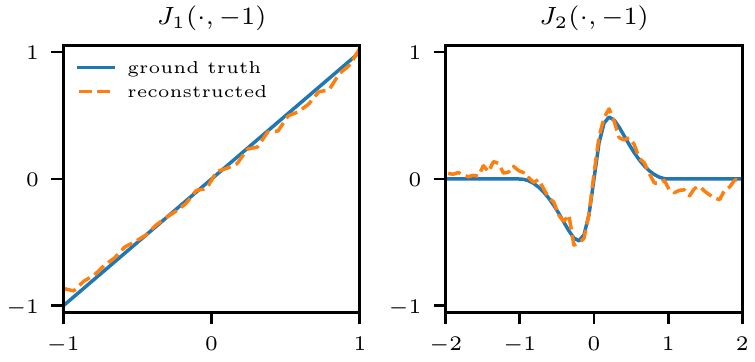}  
		\caption{}
		\label{fig:03ntfr1dexactdataresamplingJselfdist}
	\end{subfigure}
	\begin{subfigure}{.49\textwidth}
		\centering
		\includegraphics{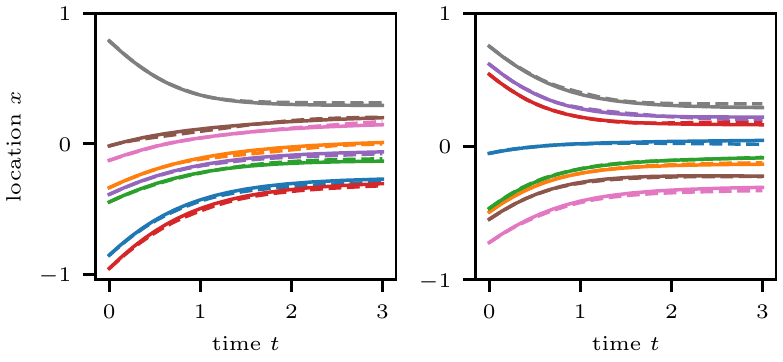}  
		\caption{}
		\label{fig:03ntfr1dexactdataresamplingtruevslearn}
	\end{subfigure}\\
	\caption{Inference from noisy $1$d observations of an undisclosed fast-reaction system. %
	(\subref{fig:03ntfr1dexactdataresamplingJselfdist}): reconstruction of the payoff function for $u=-1$. %
	(\subref{fig:03ntfr1dexactdataresamplingtruevslearn}): comparison between true trajectories (solid lines) and trajectories generated by the inferred model (dashed lines) for multiple new realizations (not part of the training data).}
	\label{fig:03ntfr1dexactdataresampling}
\end{figure}

\begin{example}[Game model, $d=1$, noisy data]
\label{ex:NTFRNoise}
We repeat the previous example with some added noise. We consider the same set of observations, but the observed mixed strategies are corrupted by re-sampling: For a given `true' simulated $\sigma$, we draw 20 times from $U$, according to $\sigma$ and now use the resulting empirical distribution for inference. These new corrupted $\sigma$ provide corrupted velocities predictions: in our example, the overall standard deviation between exact and corrupted velocities is $\approx 0.2$ (with the range of admissible velocities being $[-1,1]$).
This could model the situation where observation of the mixed strategies is imperfect or by observing very noisy (stochastic) agent movement and using empirical distributions of agent velocities over short time intervals as substitute for mixed strategies.
The results are shown in Figure \ref{fig:03ntfr1dexactdataresampling}, for parameters $\lambda_{1}=\lambda_2=10^{-5}$ to enforce more regularity due to noisy inputs.
The inferred payoff functions are similar to the previous example with just a few spurious fluctuations.
The trajectories simulated with the inferred $J$ are close to the (unperturbed) true trajectories, the error is consistently smaller than the one caused by the re-sampling noise, indicating that inference also works well in noisy settings.
\end{example}

\begin{figure}[tbh]
	\begin{subfigure}{\textwidth}
		\centering
		\includegraphics{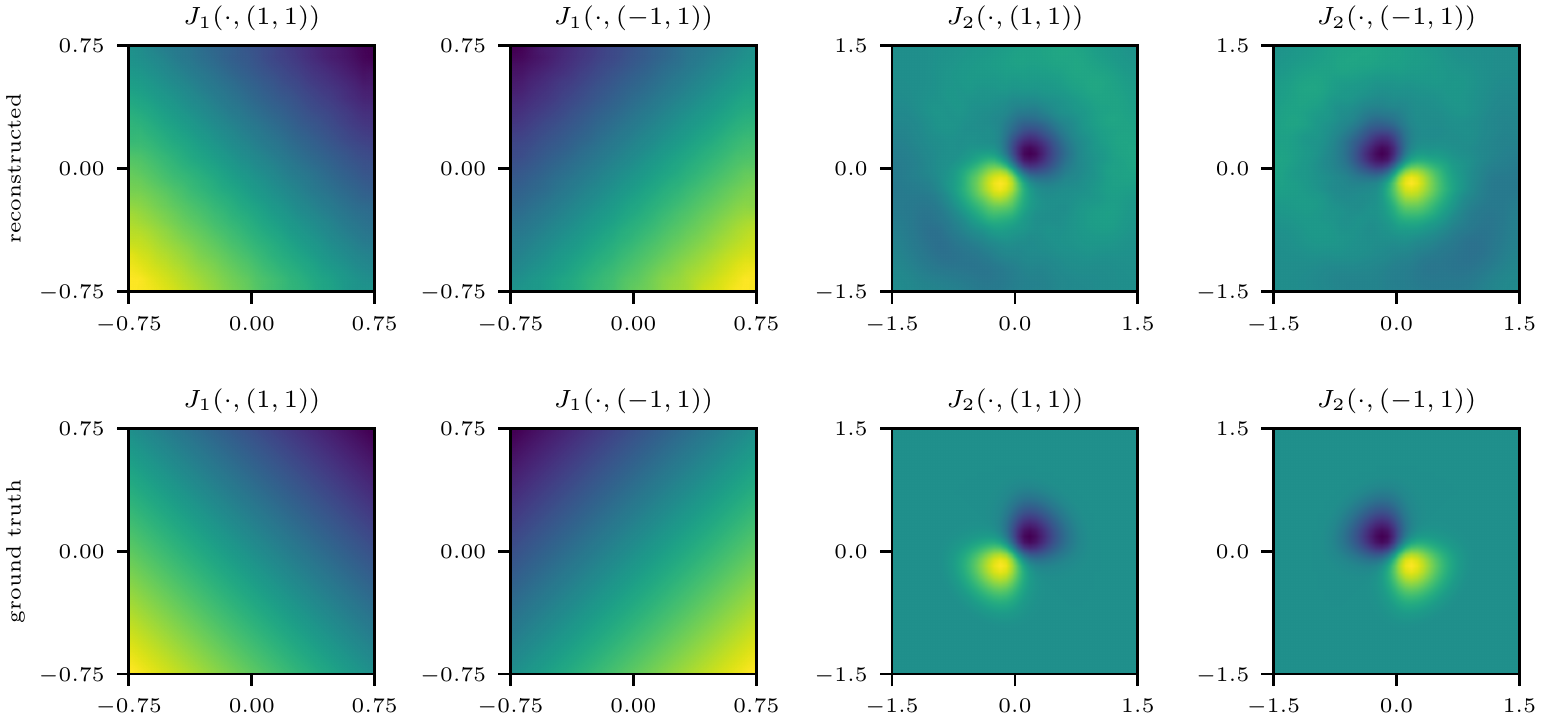}  
		\caption{}
		\label{fig:03ntfr2dexactdatajrec}
	\end{subfigure}\\
	\begin{subfigure}{0.49\textwidth}
		\centering
		\includegraphics{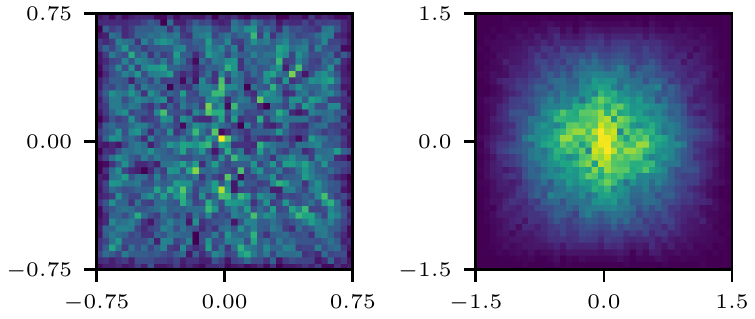}  
		\caption{}
		\label{fig:03ntfr2dexactdatahistogram}
	\end{subfigure}
	\begin{subfigure}{0.49\textwidth}
		\centering
		\includegraphics{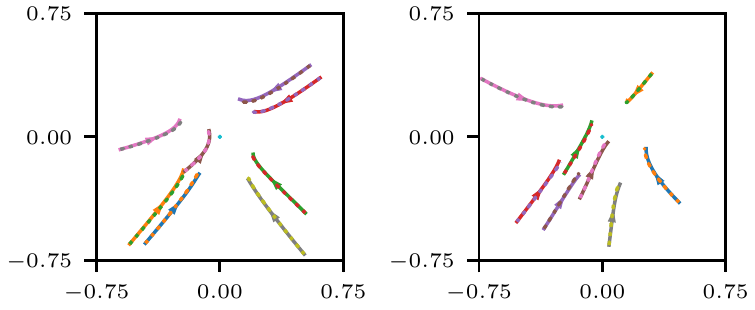}  
		\caption{}
		\label{fig:03ntfr2dexactdatatruevslearn}
	\end{subfigure}
	\caption{Inference from $2$d observations of an undisclosed fast-reaction system. %
	(\subref{fig:03ntfr2dexactdatajrec}): inferred payoff function and ground truth for pure strategies $(1,1)$ and $(-1,1)$. %
	(\subref{fig:03ntfr2dexactdatahistogram}): distribution of the data in the training set (agents locations and relative locations), dark blue : low, yellow: high. %
	(\subref{fig:03ntfr2dexactdatatruevslearn}): comparison between exact trajectories (solid lines) and trajectories generated by the inferred model (dashed lines) for two new realizations (not part of the training data).}
	\label{fig:03ntfr2d}
\end{figure}

\begin{example}[Game model, $d=2$]
	\label{ex:NTFRd2}
	Let $d=2$, $U = \{(1,1),(-1,1),(-1,-1),(1,-1)\} \subset \R^2$ and $e(x,u)=u$. We use the same $J$ defined in \eqref{eq:Jdefinition}, thus combining a driving force towards the origin $x=(0, 0)$ and repulsion between nearby agents (we assume $\tanh$ to be applied componentwise to vectors).
	
	Data collection follows the same approach as in Example \ref{ex:NTFRd1}: for $N=8$ agents we simulate 100 instances over a time span of $0.2$, collecting 5 data points per instance, for a total of 500 observed configurations. In each instance, initial locations are sampled uniformly from $[-0.75,0.75]^2$. Observed relative locations between two agents are thus distributed over $[-1.5,1.5]^2$, as displayed in Figure \ref{fig:03ntfr2dexactdatahistogram}.
	Describing the discrete $J$ required 10656 coefficients, resulting from a $30\times 30$ spatial grid for $J_1$ and a $42\times 42$ spatial grid for $J_2$. We use again the energy \eqref{eq:energysigmaDiscrete} for inference with a regularizer on derivatives, with $\lambda_{1}=\lambda_2=10^{-5}$.
	The results are illustrated in Figure \ref{fig:03ntfr2d}. The recovered functions $J_1$ and $J_2$ reproduce the same structural features of the exact ones and the trajectories simulated with the inferred $J$ follow closely the underlying true model (also in this case newly generated initial conditions are considered).
\end{example}

\begin{figure}[tbh]
	\begin{subfigure}{\textwidth}
		\centering
		\includegraphics{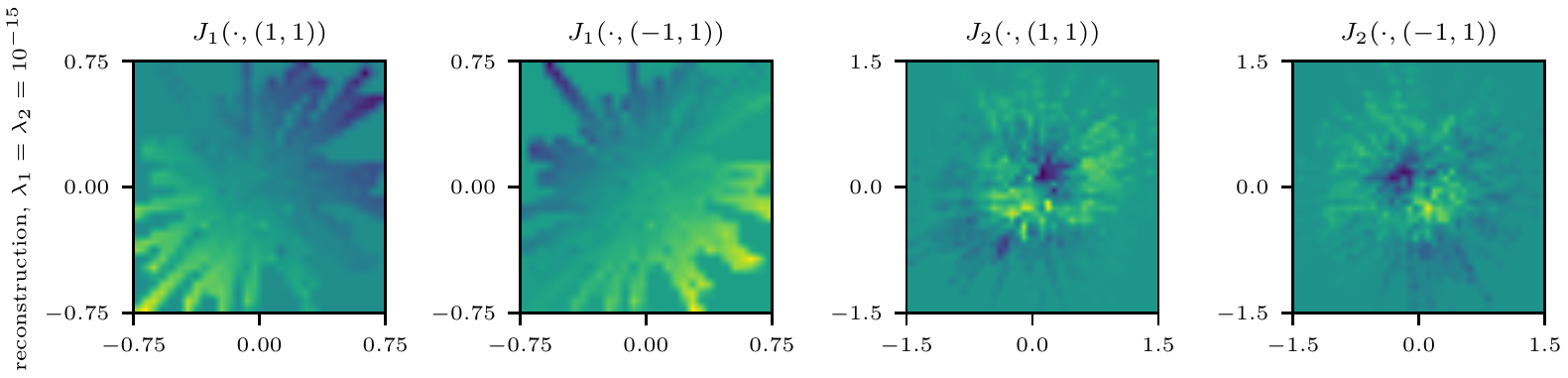}  
		\caption{}
		\label{fig:03ntfr2dexactdatalessregularizedjrecweakreg}
	\end{subfigure}\\
	
	\begin{subfigure}{\textwidth}
		\centering
		\includegraphics{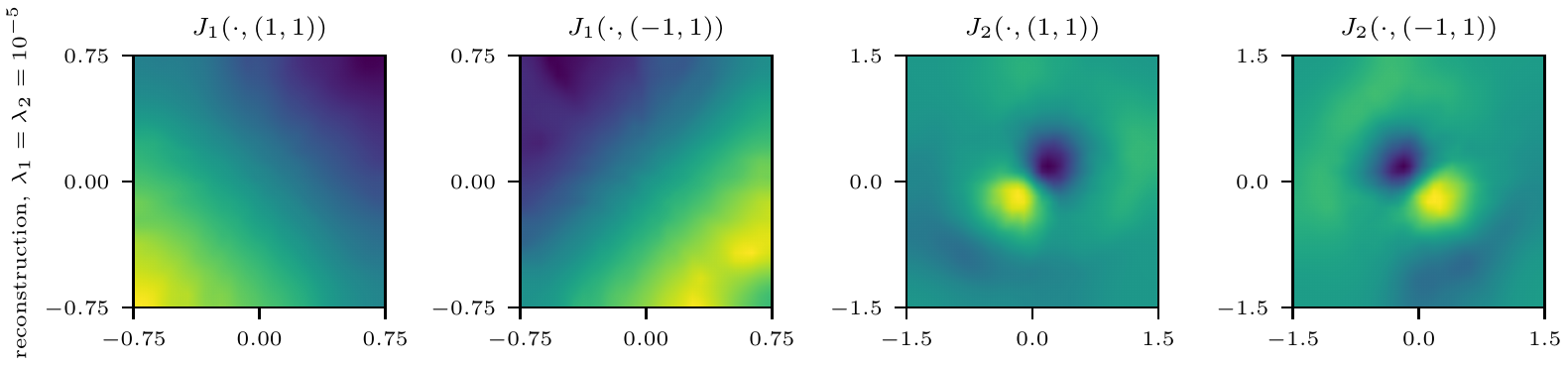}  
		\caption{}
		\label{fig:03ntfr2dexactdatalessregularizedjrecstrongreg}
	\end{subfigure}\\
	
	\begin{subfigure}{\textwidth}
		\centering
		\includegraphics{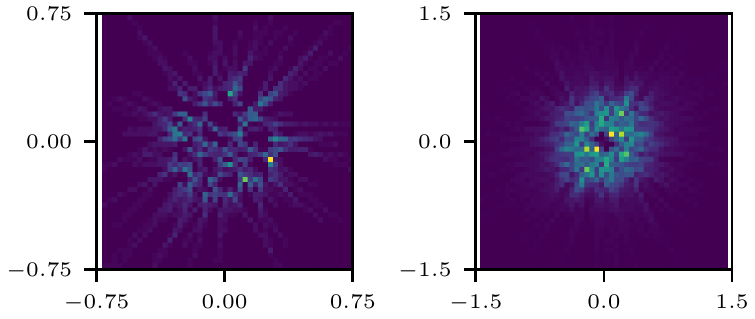}  
		\caption{}
		\label{fig:03ntfr2dexactdatahistogramlessregularized}
	\end{subfigure}
	\caption{Influence of the regularization parameter. %
	Analogous to Figure \ref{fig:03ntfr2d} but with fewer observations, as displayed in (\subref{fig:03ntfr2dexactdatahistogramlessregularized}).
	(\subref{fig:03ntfr2dexactdatalessregularizedjrecweakreg}): reconstruction for low regularization parameters, $\lambda_1=\lambda_2=10^{-15}$.
	(\subref{fig:03ntfr2dexactdatalessregularizedjrecstrongreg}): reconstruction for higher regularization parameters, $\lambda_1=\lambda_2=10^{-5}$. %
	See Example \ref{ex:NumericsUnderreg} for more details.}
	\label{fig:03ntfr2dexactdatalessregularized}
\end{figure}

\begin{example}[Under-regularization]
	\label{ex:NumericsUnderreg}
	As is common in inverse problems, if the regularization parameters $\lambda_1, \lambda_2$ in \eqref{eq:energysigmaDiscreteReg} are chosen too high, one obtains an over-regularized result, where the large penalty for spatial variation leads to a loss of contrast and high-frequency features of $J$.
	Conversely we may experience under-regularization when the parameters are small and the observed data does not cover the learning domain sufficiently, e.g.~when we only observe few realizations of the system.
	Such a case is illustrated in Figure \ref{fig:03ntfr2dexactdatahistogramlessregularized} where the observations are concentrated on a few one-dimensional trajectories.
	For weak regularization, the reconstructed $J$ then tends to be concentrated on these trajectories as well, see Figure \ref{fig:03ntfr2dexactdatalessregularizedjrecweakreg}.
	And while the inferred $J$ may predict the agent behaviour with high accuracy on these trajectories, it will generalize very poorly to different data.
	With stronger regularization, $J$ is extended beyond the trajectories in a meaningful way, see Figure \ref{fig:03ntfr2dexactdatalessregularizedjrecstrongreg}.
\end{example}

\begin{figure}[tbh]
	\begin{subfigure}[b]{.25\textwidth}
		\centering
		\includegraphics{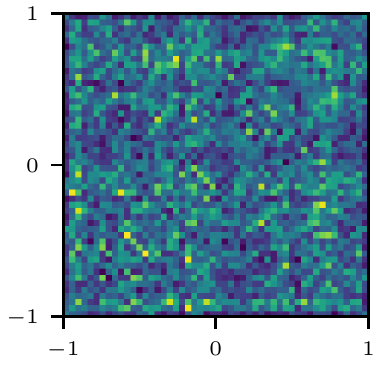}  
		\caption{}
		\label{fig:05fullmodelexamplehistogram}
	\end{subfigure}
	\begin{subfigure}[b]{.5\textwidth}
		\centering
		\includegraphics{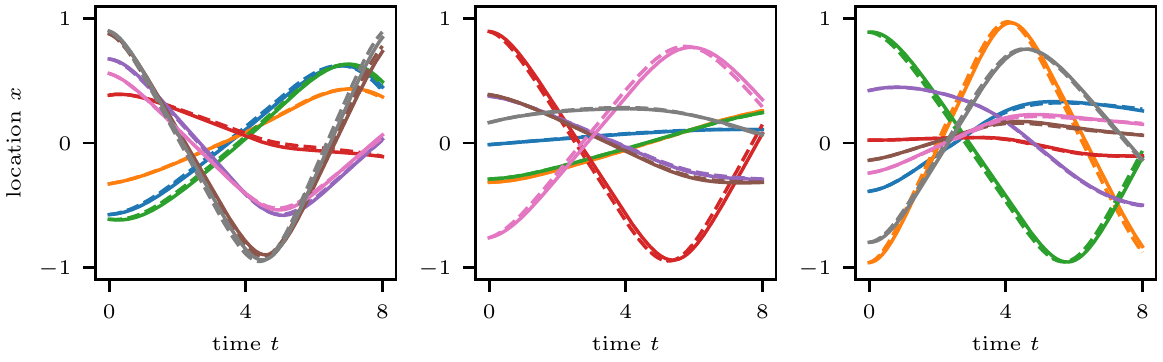}  
		\caption{}
		\label{fig:05fullmodelexampletruevslearn}
	\end{subfigure}\\
	\caption{Inference from $1$d observations generated with \eqref{eq:EntropicGameModel}. %
	(\subref{fig:05fullmodelexamplehistogram}): distribution of observed pairs $(x,x')$, $x \neq x'$. %
	(\subref{fig:05fullmodelexampletruevslearn}): comparison between exact trajectories (solid lines) and inferred model (dashed lines) on new realizations (not part of the training data).}
	\label{fig:fullexample}
\end{figure}

\subsection{An example for differential inference}
\label{sec:NumericsDifferential}
The `differential inference' functional $\energy_{\dot{\sigma}}^{J,N}$ introduced in Section \ref{sec:InferenceDifferential} was constructed in close analogy to \eqref{eq:NewtonEnergy}. We have argued that it may be challenging to apply to real data, as it would require the observation of mixed strategies $\{\sigma_{i}(t_s)\}_{s,i}$ and their temporal derivatives $\{\partial_t \sigma_{i}(t_s)\}_{s,i}$.
As a proof of concept, we now demonstrate that if these observations are indeed available, the numerical set-up described up to now can be easily adapted to $\energy_{\dot{\sigma}}^{J,N}$.
After discretization, $\energy_{\dot{\sigma}}^{J,N}$ reduces to a finite-dimensional quadratic objective.
For a quadratic regularizer, such as the squared $L^2$-norm of the gradient of $J$, minimization then corresponds to solving a linear system.
Alternatively, we can implement the Lipschitz constraint on $J$ which amounts to adding linear constraints to the quadratic problem.
We solve the resulting finite dimensional optimization problem using the operator splitting approach provided by the \textup{OSQP} solver \cite{OSQP}.

As an example, we simulate the entropic regularized model in \eqref{eq:EntropicGameModel} for $U = \{-1,1\}$ and $J \colon \R \times \{-1,1\} \times \R \times \{-1,1\}$ defined by
\begin{align*}
J(x,u,x',u') &= -\frac{1}{2}(u+u')((u+1)x^5+(u-1)(x+x')^3)
\end{align*}
Our choice of $J$ is not motivated by providing a meaningful model for anything, but simply to illustrate that inference works in principle.
We collect data using the same sub-sampling approach presented above: $8$ agents, $100$ realizations for a time span of 0.2, initial conditions uniformly sampled from $[-1,1]$, only every other observation is kept for a total of 500 input configurations. Observed positions $(x,x')$ with $x \neq x'$ are displayed in Figure \ref{fig:05fullmodelexamplehistogram}. No particular structure is assumed for $J$.

As discussed in Remark \ref{rem:NonUniqueness}, the optimal inferred $J$ is not unique.
As a remedy, we normalize the reconstructed $J^r$ such that $J^r(x,1,x',-1) = J^r(x,-1,x',1) = 0$. Comparisons for trajectories are reported in Figure \ref{fig:05fullmodelexampletruevslearn} for new randomly selected initial conditions with $T = 8$: trajectories with the reconstructed $J^r$ follow closely the trajectories generated with the true $J$.

\subsection{Learning for Newtonian models}
\label{sec:NumericsNewton}
In Example \ref{ex:FastReactionVsNewton} we have shown that entropic game models with fast reaction can also capture Newtonian interaction models. In this section we perform inference with an entropic game model ansatz on data generated by a Newtonian model. As an example, we pick
\begin{align}\label{eq:NewtonTrajectories}
\partial_t x_i(t) = \frac{1}{N} \sum_{j=1}^N f(x_i(t),x_j(t))
\qquad
\tn{with} \qquad
f(x,x') = - x - \frac{\tanh(5(x'-x))}{(1+\|x'-x\|)^2}
\end{align}
where the function $\tanh$ is assumed to be applied componentwise. The first term in $f$ encourages agents to move towards the origin, the second term implements a pairwise repulsion at close range.
Trajectory data is collected for a two dimensional system: for $N=8$ agents we simulate 100 instances of \eqref{eq:NewtonTrajectories} over a time span of $0.2$, collecting 5 data points per instance to total 500 observed configurations. In each instance, initial locations are sampled uniformly from $[-0.75,0.75]^2$. Observed relative locations between two agents are thus distributed over $[-1.5,1.5]^2$, as displayed in Figure \ref{fig:04ntfr2dNewtondatahistogram}.

We use the parametrization \eqref{eq:JSplitting} for the payoff function. Unlike in the the previous examples, the original model does not involve a strategy set $U$ or a map $e$. Thus, prior to inference, we have to make informed choices for these.
As in earlier examples, we choose $U \subset \R^2$ and $e(x,u)=u$ (i.e.~pure strategies correspond directly to velocities) and in particular we pick $U = \{u_1,\dots,u_K\}$ such that the convex hull of $U$ in $\R^2$ contains (almost) all observed velocities, meaning that they can be reproduced by suitable mixed strategies. So for (almost) every $i$ and $s$ we can write $\partial_t x_i(t_s) = \sum_k u_k\sigma_{k}$ for some mixed strategy $\sigma$ over $U$.

\begin{figure}[tbh]
	\begin{subfigure}{\textwidth}
		\centering
		\includegraphics{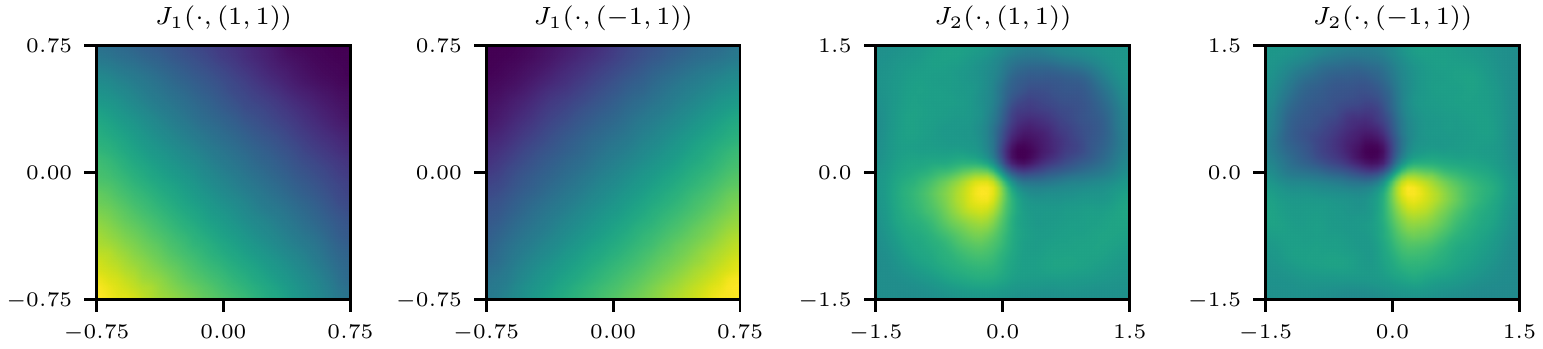}  
		\caption{}
		\label{fig:04ntfr2dNewtondatajrec}
	\end{subfigure}\\
	\begin{subfigure}{0.49\textwidth}
		\centering
		\includegraphics{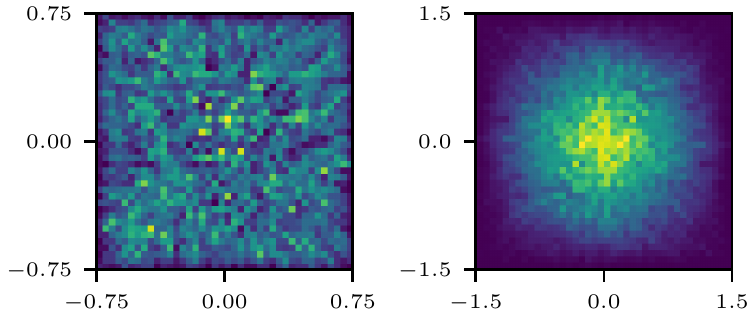}  
		\caption{}
		\label{fig:04ntfr2dNewtondatahistogram}
	\end{subfigure}
	\begin{subfigure}{0.49\textwidth}
		\centering
		\includegraphics{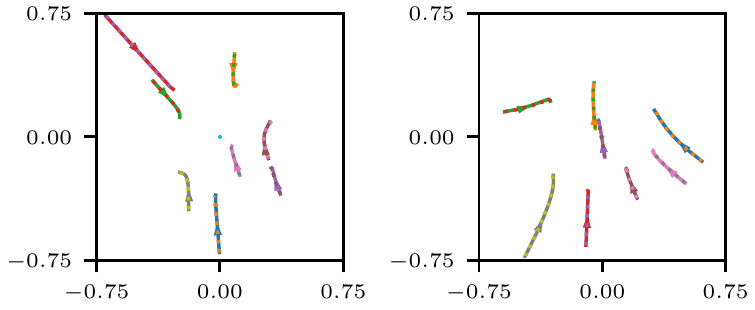}  
		\caption{}
		\label{fig:04ntfr2dNewtondatatruevslearn}
	\end{subfigure}
	\caption{Inference from $2$d observations of a Newtonian system. Top rows: inferred payoff function for pure strategies $(1,1)$ and $(-1,1)$, bottom line: distribution of the data in the training set (agents locations and pairwise distances) and comparison between exact Newtonian trajectories (solid lines) and trajectories generated by the inferred game model (dashed lines).}
	\label{fig:Newton2d}
\end{figure}

\begin{example}
	\label{ex:Newtond2}
	We pick $U = \{(1,1),(-1,1),(-1,-1),(1,-1)\} \subset \R^2$, optimize the velocity based energy \eqref{eq:energyxdotDiscrete} because no mixed strategies are provided by observations and consider a regularizer on derivatives as described in Section \ref{sec:NumericsSetup}, with $\lambda_{1}=\lambda_2=10^{-5}$. Using the same setup as Example \ref{ex:NTFRd2} we describe the discrete $J$ using 10656 coefficients ($30\times 30$ grid for $J_1$ and $42\times 42$ grid for $J_2$). The results are illustrated in Figure \ref{fig:Newton2d}. The recovered functions $J_1$ and $J_2$ have the same structure as the ones in Example \ref{ex:NTFRd2}: the qualitative behaviour of the two systems is indeed the same. In Figure \ref{fig:04ntfr2dNewtondatatruevslearn} we simulate a couple of dynamics with newly generated initial conditions and observe again how Newtonian trajectories and game trajectories computed through the reconstructed $J$ are essentially the same. 
\end{example}

\begin{figure}
	\centering
	\includegraphics{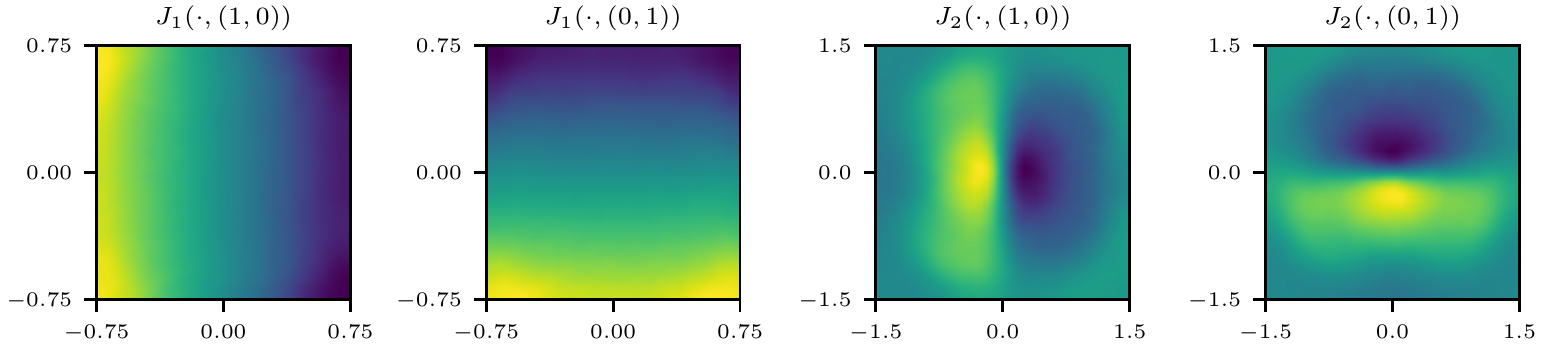}  
	\caption{Inferred payoff function via $\energy_v^N$-based energy for pure strategies $(1,0)$ and $(0,1)$ (same Newtonian input data as in Figure \ref{fig:04ntfr2dNewtondatatruevslearn}).}
	\label{fig:04ntfr2dNewtondatadifferentUjrec}
\end{figure}

\begin{example}
	\label{ex:Newtond2otherU}
	In most cases there will not be a unique preferred choice for $U$ and various options could be reasonable.
	We now demonstrate robustness of the inference results between different choices that are essentially equivalent.
	Thus, we repeat Example \ref{ex:Newtond2} changing $U$ into $U = \{(1,0),(0,1),(-1,0),(0,-1)\}$. The results are illustrated in Figure \ref{fig:04ntfr2dNewtondatadifferentUjrec}. Essentially, a rotation of $U$ caused a corresponding rotation in the reconstructed functions, without changing the underlying behaviour that is being described.
\end{example}

\begin{figure}
	\centering
	\includegraphics{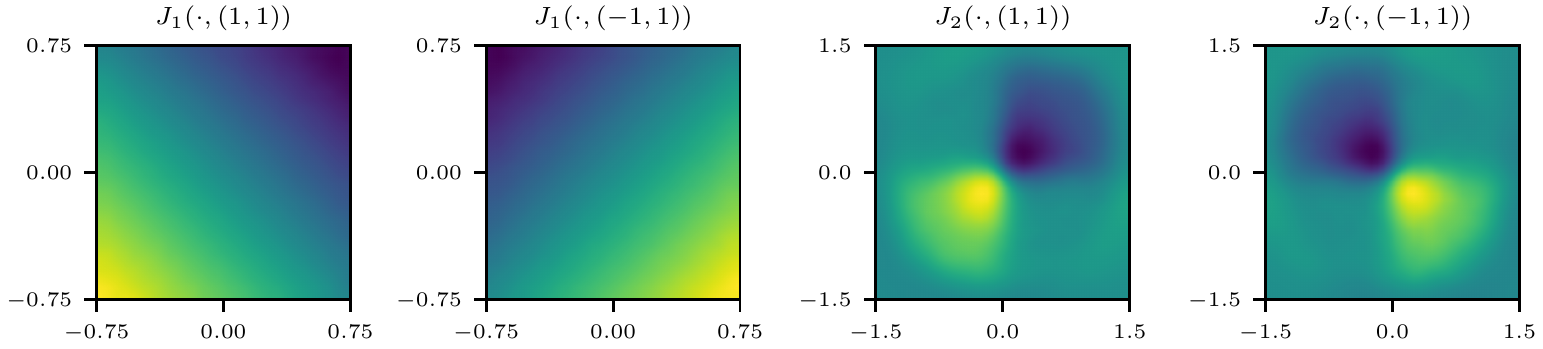}
	\caption{Inferred payoff function via $\energy_\sigma^N$-based energy for pure strategies $(1,1)$ and $(-1,1)$ (same Newtonian input data as in Figure \ref{fig:04ntfr2dNewtondatatruevslearn}). To compare with $\energy_v^N$-based reconstruction in Figure \ref{fig:04ntfr2dNewtondatajrec}.}
	\label{fig:04ntfr2dNewtondatafictituoussigmajrec}
\end{figure}

\begin{example}
	\label{ex:Newtond2sigma}
	In the previous two examples we used the velocity-based functional $\energy_v^N$ \eqref{eq:energyxdotDiscrete} for inference.
	As discussed in Section \ref{sec:NumericsSetup} we can also optimize the error functional $\energy_\sigma^N$ \eqref{eq:energysigmaDiscrete} upon providing some suitable mixed strategies via \eqref{eq:ReconstructSigma}.
	We repeat Example \ref{ex:Newtond2} with this approach. The results are illustrated in Figure \ref{fig:04ntfr2dNewtondatafictituoussigmajrec}: the reconstructed functions closely resemble the previous reconstruction provided in Figure \ref{fig:04ntfr2dNewtondatajrec}.
\end{example}

\subsection{Pedestrian dynamics}
\label{sec:NumericsPedestrian}

The examples presented in Sections \ref{sec:NumericsNTFR} and \ref{sec:NumericsDifferential} were based on observations generated with the underlying entropic game model with known $U$ and $e$. The examples in Section \ref{sec:NumericsNewton} were based on observations generated by Newtonian models, which by virtue of Example \ref{ex:FastReactionVsNewton} are known to be covered by the entropic game models, if suitable $U$ and $e$ are chosen.
Here we consider a more challenging scenario, trying to learn interaction rules between pedestrians in a model adapted from \cite{bailo2018pedestrian, Degond2013}.
This model attempts to describe `rational' agents (pedestrians) that avoid collisions with other agents not merely via a repulsion term involving the other agents' current locations, but by explicit anticipation of their trajectories.
Related datasets and models in the computer vision community can be found, for instance, in \cite{5459260,PedestrianDataset2019,GameTheoryTransportModel2021}. It will be interesting future work to see if our model can be trained on such data.
Note that this differs from the setting where pedestrians in high-density and high-stress situations are modeled as a non-sentient fluid that is solely driven by external gradients and internal pressure, see for instance \cite{BellomoCrowdReview11,CaMaWo2016} for reviews on the literature of micro- and macroscopic crowd motion models.

\begin{figure}[tbh]
	\centering
	\includegraphics[width=0.5\linewidth]{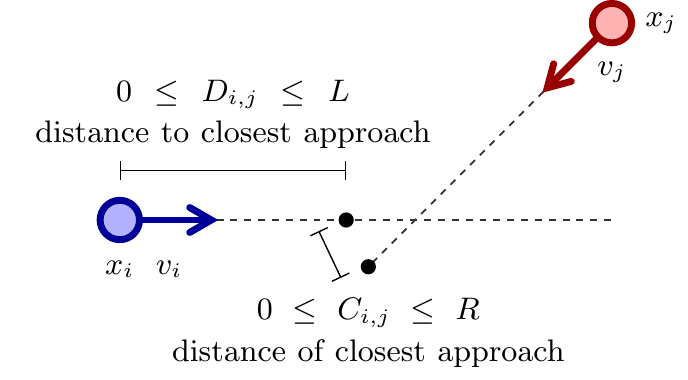}
	\caption{Adapted from \cite{bailo2018pedestrian}}
	\label{fig:heuristics}
\end{figure}

\paragraph{Forward model.}
We now describe the model for pedestrian movement that we use to generate observations. It is slightly adapted from \cite{bailo2018pedestrian, Degond2013}, to which we refer for more details and motivation of the quantities we re-introduce here. We consider $N$ pedestrians in $\R^2$, each of which with desired target velocity $\bar{v}_i \in \R^2$, $i=1,\ldots,N$. The dynamic is described by the coupled system of second order equations
\begin{align}\label{eq:forwardPedestrians}
\left\{
\begin{aligned}
\partial_t x_i(t) &= v_i(t), \\
\partial_t v_i(t) &= -\nabla_{v_i} \Phi_i(x_1(t),\ldots,x_N(t),v_1(t),\ldots,v_N(t)), \\
\|v_i(t)\| & = 1
\end{aligned}
\right.
\end{align}
where the third line means that we confine the gradient movement of the velocities in the second line to the unit sphere. This reduces the dimension of the strategy space and thus simplifies inference (see below).
The potential $\Phi_i$ is designed to steer the velocity towards the target velocity $\bar{v}_i$ while at the same time trying to avoid as much as possible close encounters with other pedestrians. The potential is computed from pairwise interaction terms that involve anticipated (near-)collision events. A key assumption is that agents are able to evaluate at each time the current locations and velocities of the other agents in their cone of vision, i.e.,~they fully know the relevant status of the system. Thus, for locations $x_1, \dots, x_N$ and velocities $v_1, \dots, v_N$, each agent $i$ assumes the others' velocities to be fixed and can evaluate ``optimality'' of a new velocity $v$ through the following quantities:
\begin{align}
D_{i,j}(v) &= -\frac{(x_j-x_i)\cdot(v_j-v)}{\|v_j-v\|^2}\|v\| \\
C_{i,j}(v) &= \left( \|x_j-x_i\|^2 - \frac{((x_j-x_i)\cdot(v_j-v))^2}{\|v_j-v\|^2}\right)^{1/2} \\
S_{i,j}(v) &=
\left\{
\begin{aligned}
1 \quad&\text{ if } \frac{(x_j-x_i)\cdot v}{\|x_j-x_i\|\|v\|} > \cos\left(\frac{7}{12}\pi\right) \text{ and } D_{i,j}(v) > 0 \\
0 \quad&\text{ else }
\end{aligned}
\right.
\end{align}
These are illustrated in Figure \ref{fig:heuristics}.
$D_{i,j}$ is the travelling distance of $i$ to the closest encounter with $j$, based on velocities $v$ and $v_j$. It is negative if the the closest encounter would have happened in the past.
$C_{i,j}$ is their distance realized at that closest encounter and $S_{i,j}$ counts whether agent $j$ is visible to $i$ and getting closer.
Note that $D_{i,j}$, $C_{i,j}$ and $S_j$ are functions of the other agent's location and velocity. From the perspective of agent $i$ these are considered parameters and thus we drop them in the explicit notation of arguments.
For a parameter $a > 0$ define the quadratic penalty
\begin{align*}
\phi_a(x) =
\begin{cases}
\left(\frac{x}{a}-1\right)^2 & \tn{if } x\leq a, \\
0 & \tn{else,}
\end{cases}
\end{align*}
and then set the $i$-th potential to be
\begin{align}\label{eq:potentialPhi}
\Phi_i(v) = k_1 \|v-\bar{v}_i\|^2 + \frac{k_2}{\sum_{j \neq i} S_{i,j}(v)}\sum_{j \neq i} S_{i,j}(v) \cdot \phi_R(C_{i,j}(v)) \cdot \phi_L(D_{i,j}(v))
\end{align}
for some parameters $k_1, k_2, L, R > 0$. Agent $i$ will then evolve its velocity based on the gradient of $\Phi_i$, which balances the desire to stay close to the target velocity $\bar{v}_i$ and the avoidance of close encounters with other agents.
The second term penalizes anticipated minimal distances $C_{i,j}$ below $R$, that happen on a distance horizon of $L$, with a stronger penalty if the anticipated encounter is closer.

\paragraph{State space and strategy space.}
System \eqref{eq:forwardPedestrians} is simulated with potentials $\Phi_i$ defined as in \eqref{eq:potentialPhi}.
By the third line of \eqref{eq:forwardPedestrians} we can write $v_i(t) = (\cos(\theta_i(t)), \sin(\theta_i(t)))$ at each $t$ for a suitable time depended angle $\theta_i$. Observations consist of locations $\{x_i(t_s)\}_{s,i}$ and velocities $\{v_i(t_s)\}_{s,i}$ for $s \in \{0,\dots,S\}$ and $i \in \{1,\dots,N\}$. In particular, observations on velocities can be recast into observations of the angles $\{\theta_i(t_s)\}_{s,i}$, $\theta_i(t_s) \in S^1$, so that $v_i(t_s) = (\cos(\theta_i(t_s)),\sin(\theta_i(t_s)))$.
Likewise, the desired velocity $\bar{v}_i$ can be encoded by an orientation $\bar{\theta}_i \in S^1$.

We want to model the system as an undisclosed fast-reaction entropic game. Consequently, we must choose a suitable `location' space, strategy space and map $e$ for the agents.
In the undisclosed setting, the choice of strategy of an agent does not depend on the choices of strategies of the other agents.
However, for pedestrians, the movement choices of one agent clearly depend on the current velocities of other agents. This is explicit in the above model, and of course also true in reality, where we can learn about other persons' current movement by their orientation (people walk forward).
Consequently, the orientations $\theta_i$ cannot be interpreted as strategies, and should be considered as part of the agents' locations.
Instead, the second line of \eqref{eq:forwardPedestrians} suggests that strategies describe the agents' choice to change their orientations.
Finally, different agents may have different desired orientations $\bar{\theta}_i$, i.e.,~we may have different `species' of agents. As discussed in Remark \ref{rem:multipleSpecies} this can be formally incorporated into our model by enriching the physical space.

Consequently, as physical space we use $\R^2 \times S^1 \times S^1$ where the first component describes the physical location of the agent, the second component their current orientation and the third component their desired orientation.

As strategy space we pick $U \subset \{-C,+C\}$ for some $C>0$, where a pure strategy $u \in U$ corresponds to the desire to change the orientation with rate $u$.
$C$ should be picked sufficiently large to capture the maximally observed changes of orientation of the agents.

In summary, in this particular instance of \eqref{eq:QuasiStaticGameModel} we choose
\begin{align*}
e((x,\theta,\bar{\theta}),u) =
\begin{pmatrix}
\cos(\theta),
\sin(\theta),
u,
0
\end{pmatrix},
\end{align*}
and consequently obtain the full system
\begin{align}
\label{eq:PedestriansExplicitJSystem}
\left\{
\begin{aligned}
\partial_t x_i(t) &= (\cos(\theta_i(t)),\sin(\theta_i(t)))  \\
\partial_t \theta_i(t) 	&= \vartheta_i^J\big((x_1(t), \theta_1(t),\bar{\theta}_1(t)),\dots,(x_N(t), \theta_N(t), \bar{\theta}_N(t))\big)\\
&\assign \frac1K \sum_{k=1}^K u_k \cdot \sigma_{i}^J\big((x_1(t), \theta_1(t),\bar{\theta}_1(t)),\dots,(x_N(t), \theta_N(t), \bar{\theta}_N(t))\big)(u_k) \\
\partial_t \bar{\theta}_i(t) &= 0
\end{aligned}
\right.
\end{align}

\paragraph{Ansatz for payoff function.}
For the payoff function, in analogy to \eqref{eq:JSplitting}, we make the following ansatz:
\begin{align}
\label{eq:JstructurePedestrians}
J((x,\theta, \bar{\theta}),u,(x',\theta',\bar{\theta}')) = J_1(\theta, u, \bar{\theta}) + J_2(R_{-\theta} \cdot (x'-x), \theta'-\theta, u)
\end{align}
where $R_{-\theta}$ denotes the rotation matrix by angle $-\theta$.
$J_1$ is intended to recover the drive towards the desired orientation $\bar{\theta}_i$ and $J_2$ the pairwise interaction between agents.
We have added the natural assumption that $J_2$ is invariant under translations and rotations. For each agent we may choose a coordinate system such that they are currently in the origin ($x=0$), heading right ($\theta=0$), whereas the other agent is located at $y=R_{-\theta} (x'-x)$ with orientation $\Delta \theta=\theta'-\theta$.
Further dimensionality reduction can be obtained by setting $J_2$ to zero, when the other agent is not within the agent's field of vision. Alternatively, we can attempt to infer the field of vision from the observations during learning.

\paragraph{Discrete inference functional.}
We approximate the missing velocities $\{\partial_t \theta_i(t_s)\}_{s,i}$ by finite differences,
\[
\partial_t \theta_i(t_s) \approx \frac{\theta_i(t_s)-\theta_i(t_{s-1})}{t_s-t_{s-1}} \quad \text{ for all }s=1,\dots,S, \text{ and } i = 1,\dots,N,
\]
and the corresponding velocity based inference functional \eqref{eq:energyxdot} is given by
\begin{align}
\label{eq:energyxdotDiscretePedestrians}
\energy_v^N(J) = \frac{1}{SN} \sum_{s=1}^S \sum_{i=1}^N \left| \vartheta_i^J(t_s) - \partial_t \theta_i(t_s)\right|^2
\end{align}
where $\vartheta_i^J$ is introduced in \eqref{eq:PedestriansExplicitJSystem}.

\begin{figure}[tbh]
	\begin{subfigure}{0.49\textwidth}
		\centering
		\includegraphics{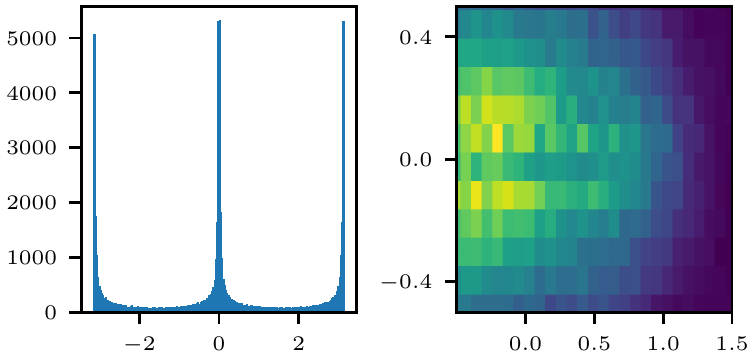}  
		\caption{Distribution of observed data. Left, distribution of observed directions of agents (some minimal observations are available everywhere). Right, distribution of relative configurations between agents (assuming reference agent sitting at $(0,0)$).}
		\label{fig:06pedestrians2histogram}
	\end{subfigure}
	\begin{subfigure}{0.49\textwidth}
		\centering
		\includegraphics{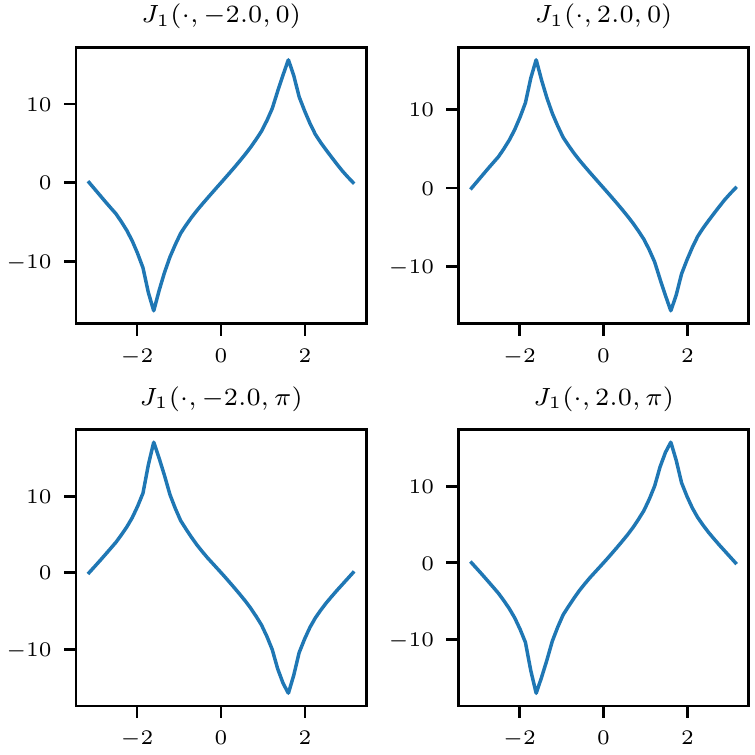}  
		\caption{}
		\label{fig:06pedestrians2Jself}
	\end{subfigure}\\
	\begin{subfigure}{0.49\textwidth}
		\centering
		\includegraphics{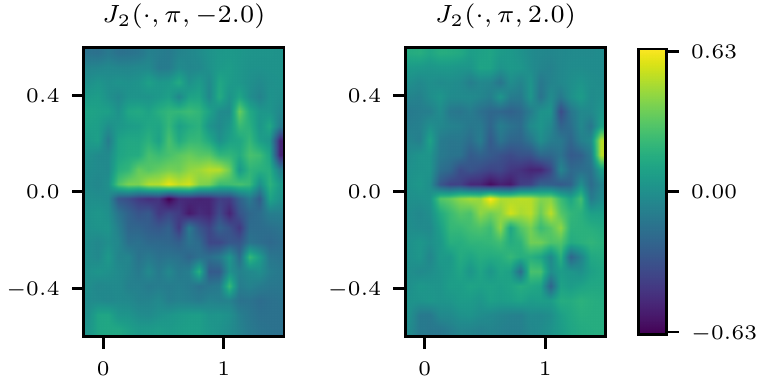}  
		\caption{}
		\label{fig:06pedestrians2Jdist}
	\end{subfigure}
	\begin{subfigure}{0.49\textwidth}
		\centering
		\includegraphics{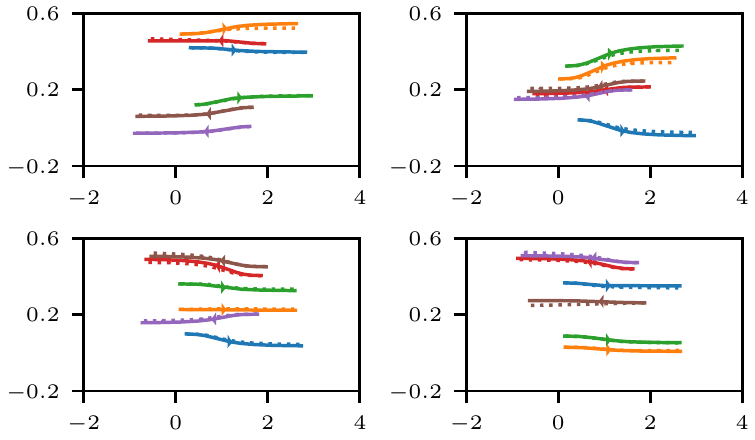}  
		\caption{}
		\label{fig:06pedestrians2truevslearn}
	\end{subfigure}
	\caption{Inference from $2$d observations of a pedestrian model. (\subref{fig:06pedestrians2histogram}): distribution of the data in the training set. %
	(\subref{fig:06pedestrians2Jself}, \subref{fig:06pedestrians2Jdist}): reconstruction of the payoff function (section for $J_2$ at $\Delta \theta = \pi$). %
	(\subref{fig:06pedestrians2truevslearn}): comparison between exact trajectories (solid lines) and trajectories generated by the inferred model (dashed lines) for four new realizations (not part of the training data).}
	\label{fig:pedestrians2}
\end{figure}

\begin{example}
Analogously to previous examples, we simulate multiple realization of \eqref{eq:forwardPedestrians}: for each instance we consider 6 agents, 3 starting at random locations inside $[0, 0.5] \times [0, 0.5]$ with desired direction $\bar{\theta} = 0$, while the other 3 start somewhere within $[1.0, 1.5] \times [0, 0.5]$ with target direction $\bar{\theta} = \pi$. We simulate $400$ instances of \eqref{eq:forwardPedestrians}, assuming for $300$ of them a slight perturbation of the target direction as initial direction, and sampling randomly in $(-\pi,\pi)$ for the initial directions of the remaining $100$ runs (to observe the behaviour at different directions). Each run we perform $500$ steps of an explicit Euler-stepping with $\Delta t=0.005$. Observations are then sub-sampled every $20$-th step.  The final distribution of directions within the training data is reported in Figure \ref{fig:06pedestrians2histogram}.

In the reconstruction we consider $U = \{-2,2\}$. Energy \eqref{eq:energyxdotDiscretePedestrians} is then optimized coupled with regularizers in the same fashion as in \eqref{eq:energysigmaDiscreteReg}, using a $30$ nodes regular grid over $[-\pi,\pi]$ for $J_1$ and a $20 \times 20 \times 20$ grid over $[-0.15, 1.5] \times [-0.6,0.6] \times [-\pi,\pi]$ for $J_2$ (we assume a priori that agents outside my cone of vision go unseen). Results are reported in Figure \ref{fig:pedestrians2}.
The self-interaction function $J_1$, Figure \ref{fig:06pedestrians2Jself}, correctly dictates the steering of the direction towards the desired direction. Observe how the self-driving force to go back to the desired direction is maximal around $\pm \pi/2$, as expected when looking at potential \eqref{eq:potentialPhi}. The interaction function $J_2$, Figure \ref{fig:06pedestrians2Jdist}, encodes the expected deviation for avoiding a close interaction: when approaching an agent to my left I favour deviations to the right and vice versa. Eventually, we compare in Figure \ref{fig:06pedestrians2truevslearn} simulated trajectories obtained with the ground truth model and the inferred one, starting at locations outside the training set. The two models reproduce the same behaviour.
\end{example}

\section{Conclusion and outlook}
\label{sec:Conclusion}
In this article we proposed several modifications to the model of spatially inhomogeneous evolutionary games as introduced by \cite{AmFoMoSa2018}.
We added entropic regularization to obtain more robust trajectories, considered the undisclosed setting and then derived the quasi-static fast-reaction limit.
Relying on results from \cite{AmFoMoSa2018} we established well-posedness and a mean-field limit for all considered model variants.
The resulting new family of multi-agent interaction models vastly generalizes first and second-order Newtonian models.

We then posed the inverse problem of inferring the agent payoff function from observed interactions. We gave several variational formulations that seek to minimize the discrepancy between the observed particle velocities and those predicted by the payoff function. The mismatch functionals provide quantitative control on the ability of the reconstructed payoff function to replicate the observed trajectories and the functionals are (under suitable assumptions) consistent in the limit of increasing observed data.

Finally, we demonstrated that the inference scheme is not merely limited to abstract analysis but can be implemented numerically. Our computational examples included the inference from models for which a true payoff function was available (and known), examples where this was not the case (e.g.~data generated by a Newtonian model or a simple `rational' model for pedestrian movement), and examples with simulated noise. These showed that our scheme can infer meaningful payoff functions from a broad variety of input data.

A logical next step would be the numerical application to real data. Here one will need to be particularly careful about choosing the family of models to match (i.e.~the complexity of the space of pure strategies and the generality of the form of the payoff function). On one hand, a relatively simple class of models will be numerically robust and inexpensive but involve considerable modelling bias due to a potential lack of flexibility. On the other hand, a rather general class of models provides more descriptive power but will require many observations for meaningful inference and be numerically more challenging. On top of that, real data will be corrupted by noise and thus a more detailed analysis of learning from noisy data will become crucial (in the spirit of Example \ref{ex:NTFRNoise}).

\section*{Acknowledgement}
M.F. acknowledges the support of the DFG Project "Identification of Energies from Observation of Evolutions" and the DFG SPP 1962 "Non-smooth and Complementarity-based Distributed Parameter Systems: Simulation and Hierarchical Optimization". 
M.B. and M.F. acknowledge the support of the DFG Project "Shearlet approximation of brittle fracture evolutions" within the 
SFB Transregio 109 "Discretisation in Geometry and Dynamics".
M.B. and B.S. acknowledge support by the Emmy Noether Programme of the DFG.

\appendix

\section{Mathematical analysis}
\label{sec:Analysis}

\subsection{Well-posedness for entropic games}
\label{sec:AnalysisEntropic}
In this Section we prove the key technical ingredients involved in the proof of Theorem \ref{thm:wellPosedEntropic} in Section \ref{sec:ModelEntropic}, which establishes the well-posedness of the modified entropic dynamics given by \eqref{eq:EntropicGameModel}. 
The two subsequent Lemmas establish that $f^{J,\veps}=f^J + f^\veps$, \eqref{eq:BanachFParts}, is Lipschitz continuous in the chosen setting for strategies in $\sigmaSet$.

\begin{lemma}[$f^\veps$ is $L^p_\eta$-Lipschitz continuous on $\sigmaSet$]
	\label{lem:fEpsLip}
	For $\sigma_1, \sigma_2 \in \sigmaSet$ and $f^\veps$ defined in \eqref{eq:fH}, we have
	\begin{align}
	\|f^\veps(\sigma_1)-f^\veps(\sigma_2)\|_{L^p_\eta(U)} \leq L \cdot \|\sigma_1-\sigma_2\|_{L^p_\eta(U)}
	\end{align}
	with $L = L(a,b)$.
\end{lemma}
\begin{proof}
	Define $G(x) \assign x\log x$, let $L_{a,b}$ be its Lipschitz constant and $M_{a,b}$ its maximal absolute value over the interval $[a,b]$.
	For any $\sigma_1,\sigma_2 \in \sigmaSet$, $L^p_\eta$-Lipschitz continuity can be proved via direct computations:
	\[
	\begin{aligned}
	&\|f^\veps(\sigma_2)-f^\veps(\sigma_1)\|_{L^p_\eta(U)} \\
	= {} & \left\| -G(\sigma_2(\cdot)) + G(\sigma_1(\cdot))
	+ \sigma_2(\cdot) \int_U G(\sigma_2(v))\,d\eta(v)
	- \sigma_1(\cdot) \int_U G(\sigma_1(v))\,d\eta(v)
	\right\|_{L^p_\eta(U)} \\
	\leq {} & \left\| G(\sigma_2(\cdot)) - G(\sigma_1(\cdot)) \right\|_{L^p_\eta(U)}
	+ \left\| \sigma_1(\cdot) \int_U |G(\sigma_2(v)) - G(\sigma_1(v))|\,d\eta(v) \right\|_{L^p_\eta(U)} \\
	& + \left\| |\sigma_2(\cdot)-\sigma_1(\cdot)| \cdot \int_U G(\sigma_2(v))\,d\eta(v) \right\|_{L^p_\eta(U)} \\
	\leq {} &  (L_{a,b} +  b \cdot L_{a,b} + M_{a,b} ) \cdot \|\sigma_2-\sigma_1\|_{L^p_\eta(U)}
	\end{aligned}
	\]
	where we used $\sigma_1(\cdot) \leq b$ and Jensen's inequality in the second term.
\end{proof}

\begin{lemma}[$f^J$ is Lipschitz continuous]
	\label{lem:fJLip}
	For $J \in \fullX$, let $f^J$ be defined as in \eqref{eq:fJ}. Then, $f^J$ is Lipschitz continuous in all four arguments, i.e., given $x_i, x_i' \in \R^d$ and $\sigma_i, \sigma_i' \in \sigmaSet$, $i = 1,2$, we have
	\[
	\begin{aligned}
	\|f^J(x_2,\sigma_2,x_2',\sigma_2')-f^J(x_1,\sigma_1,x_1',\sigma_1')\|_{L^p_\eta(U)} \leq L \cdot \bigg[ &\|x_2-x_1\| + \|\sigma_2-\sigma_1\|_{L^p_\eta(U)} \\
	+ &\|x_2'-x_1'\| + \|\sigma_2'-\sigma_1'\|_{L^p_\eta(U)} \bigg]
	\end{aligned}
	\]
	with $L = L(J,b)$.
\end{lemma}

\begin{proof}
	Define the payoff for the single interaction
	$j \colon \R^d \times \R^d \times \sigmaSet \times U \to \R$ as
	\[
	j(x,x',\sigma',u) \assign \int_U J(x,u,x',u')\sigma'(u') \,d\eta(u')
	\]
	and observe that, for any $x,x' \in \R^d$, $u \in U$ and $\sigma_1',\sigma_2' \in \sigmaSet$, we have
	\begin{equation}\label{eq:jlip}
	\begin{aligned}
	|j(x,x',\sigma'_2,u)-j(x,x',\sigma'_1,u)| &= \left| \int_U J(x,u,x',u')(\sigma'_2(u')-\sigma'_1(u'))\,d\eta(u') \right|\\
	& \leq \|J\|_\infty \cdot \|\sigma'_2-\sigma'_1\|_{L^p_\eta(U)}.
	\end{aligned}
	\end{equation}
	The function $f^J$ introduced in \eqref{eq:fJ} can then be written as
	\[
	\begin{aligned}
	&f^J(x,\sigma,x',\sigma') = \left( j(x,x',\sigma',\cdot) - \int_U j(x,x',\sigma',v)\sigma(v)\,d\eta(v) \right) \sigma.
	\end{aligned}
	\]
	The Lipschitz continuity in $x$ and $x'$ is clear: it follows by Lipschitz continuity of $j$ in $x,x'$, descending directly from Lipschitz continuity of $J$. It remains to investigate the dependence on $\sigma, \sigma'$. For $\sigma_1,\sigma_2,\sigma' \in \sigmaSet$, direct estimates provide
	\[
	\begin{aligned}
	& \|f^J(x,\sigma_2,x',\sigma') - f^J(x,\sigma_1,x',\sigma')\|_{L^p_\eta(U)}
	= \Bigg\| j(x,x',\sigma',\cdot)\,(\sigma_2(\cdot)-\sigma_1(\cdot)) \\
	& \qquad - \sigma_2(\cdot )\int_U j(x,x',\sigma',v)\sigma_2(v)\,d\eta(v)
	+ \sigma_1(\cdot) \int_U j(x,x',\sigma',v)\sigma_1(v)\,d\eta(v) \Bigg\|_{L^p_\eta(U)} \\
	& \leq \left\| |j(x,x',\sigma',\cdot)|\cdot |\sigma_2(\cdot)-\sigma_1(\cdot)| \right\|_{L^p_\eta(U)}
	+  \left\|  \sigma_1(\cdot)\cdot\int_U |j(x,x',\sigma',v)|\cdot|\sigma_2(v)-\sigma_1(v)|\,d\eta(v)
	\right\|_{L^p_\eta(U)} \\
	& \qquad + \left\|  |\sigma_2(\cdot)-\sigma_1(\cdot)| \cdot
	\int_U |j(x,x',\sigma',v)|\cdot \sigma_2(v) \,d\eta(v) \right\|_{L^p_\eta(U)} \\
	&\leq (\|J\|_\infty + b\cdot \|J\|_\infty + \|J\|_\infty) \cdot \|\sigma_2-\sigma_1\|_{L^p_\eta(U)}
	\end{aligned}
	\]
	where we used $j(\cdot,\cdot,\cdot,\cdot) \leq \|J\|_\infty$, $\sigma(\cdot) \leq b$ and Jensen's inequality.
	Similarly, for $\sigma,\sigma_1',\sigma_2' \in \sigmaSet$, taking into account \eqref{eq:jlip}, we have
	\[
	\begin{aligned}
	&\|f^J(x,\sigma,x',\sigma'_2) - f^J(x,\sigma,x',\sigma'_1)\|_{L^p_\eta(U)} \\
	= {} & \Bigg\| \Bigg[
	\left[ j(x,x',\sigma'_2,\cdot) - j(x,x',\sigma'_1,\cdot)\right]
	- \int_U\left[ j(x,x',\sigma'_2,v) - j(x,x',\sigma'_1,v)\right] \sigma(v)\,d\eta(v)
	\Bigg]\,\sigma(\cdot) \Bigg\|_{L^p_\eta(U)} \\
	\leq {} & (b \cdot \|J\|_\infty + b^2 \cdot \|J\|_\infty) \cdot \|\sigma'_2-\sigma'_1\|_{L^p_\eta(U)}
	\end{aligned}
	\]
	which concludes the proof.
\end{proof}

\begin{remark}[Entropic $f^\veps$ is not BL-Lipschitz continuous]
	In \cite{AmFoMoSa2018}, for the study of the original un-regularized system \eqref{eq:GameModel}, the authors work in the space $Y = \R^d \times F(U)$, where
	\[
	F(U) \assign \overline{\textup{span}(\prob(U))}^{\|\cdot\|_{BL}}
	\]
	with the bounded Lipschitz norm defined as
	\begin{equation}\label{eq:BLnorm}
	\|\sigma\|_{BL} \assign \sup \left\{ \int_U \varphi(u) \diff\sigma(u) \mid \varphi \in C(U), \|\varphi\|_\infty + \Lip(\varphi) \leq 1  \right\}.
	\end{equation}
	Large parts of the analysis in \cite{AmFoMoSa2018} are based on this bounded Lipschitz norm. However, the function $f^\veps$ is not Lipschitz continuous with respect to this norm. For $0 < a \ll 1 \ll b$, let $U=[0,1]$ and $\eta$ the uniform Lebesgue measure over $U$. For $n \geq 2$ we split $[0,1]$ into $2^n$ intervals and define the nodes $x_i^n = i/2^n$, for $0 \leq i \leq 2^n$. Define $s^n \colon (0,1) \to (0,2)$ as
	\[
	s^n(x) = \begin{cases}
	2-a & \tn{if } x \in (x_i^n, x_{i+1}^n), i\text{ even} \\
	a   & \tn{if } x \in (x_i^n, x_{i+1}^n), i\text{ odd}.
	\end{cases} 
	\]
	This function alternates between $2-a$ and $a$ over each sub-interval. Set now
	\[
	\sigma_1^n(x) = \begin{cases}
	1 & \tn{if } x \in (0, 1/2) \\
	s^n(x)   & \tn{if } x \in (1/2,1)
	\end{cases} 
	\quad \text{and} \quad
	\sigma_2^n(x) = \begin{cases}
	s^n(x) & \tn{if } x \in (0, 1/2) \\
	1   & \tn{if } x \in (1/2,1)
	\end{cases} 
	\]
	On the one hand, we find that $\|\sigma_2^n - \sigma_1^n\|_{BL} \leq W_1(\sigma_2^n, \sigma_1^n) \to 0$ as $n \to \infty$. On the other hand, observe that
	\[
	\int_0^1 \sigma_1^n(x) \log(\sigma_1^n(x))\,dx = \int_0^1 \sigma_2^n(x) \log(\sigma_2^n(x))\,dx = \frac{1}{4}(a\log(a)+(2-a)\log(2-a)) =: M
	\]
	so that, for $\varphi(x) = x/2$ (which is admissible in \eqref{eq:BLnorm}), one has
	\begin{align*}
	&\tfrac{1}{\veps}\|f^\veps(\sigma_2^n) - f^\veps(\sigma_1^n)\|_{BL}
	\geq -\int_0^{1/2} \tfrac{x}{2} \cdot s^n(x)\log(s^n(x))\,\diff x
	+ \int_{1/2}^1 \tfrac{x}{2} \cdot s^n(x)\log(s^n(x))\,\diff x \\
	& + M \underbrace{\int_0^1 \tfrac{x}{2} \cdot(\sigma_2^n(x)-\sigma_1^n(x))\,\diff x}_{\to 0 \text{ as } n \to \infty}
	= \frac{1}{4}\int_0^{1/2} s^n(x)\log(s^n(x))\,\diff x + o(1) = \frac{M}{4} + o(1).
	\end{align*}
	Hence, $f^\veps$ is not Lipschitz continuous with respect to the bounded Lipschitz norm.
\end{remark}

\medskip
We are left with the proof of the compatibility Lemma \ref{lem:Compat}, which we briefly restate for the reader's convenience.
\begin{lemma}[Compatibility condition]
	\label{lem:CompatII}
	For $J \in \fullX$ and $\varepsilon > 0$, let $f^J$ and $f^\varepsilon$ be defined as in \eqref{eq:fJ} and \eqref{eq:fH}. Then, there exist $a,b$ with $0 < a < 1 < b < \infty$ such that for any $(x,\sigma),(x',\sigma') \in \R^d \times \sigmaSet$ there exists some $\theta>0$ such that
	\begin{align}
	\label{eq:CompatII}
	\sigma + \theta \lambda \left[f^J\big(x,\sigma,x',\sigma'\big) +f^\veps\big(\sigma\big) \right] \in \sigmaSet.
	\end{align}
\end{lemma}
\begin{proof}
	Fix any $(x,\sigma), (x',\sigma') \in (\R^d \times \sigmaSet)^2$. Taking into account that $\|\sigma\|_{L^1_\eta(U)} = 1$, one obtains
	\[
	\int_U \left[ \sigma + \theta\lambda (f^J(x,\sigma,x',\sigma') + f^\veps(\sigma) ) \right] \,\diff\eta = 1
	\]
	for any $\theta>0$. We are left with proving that, for suitable $0 < a \ll 1 \ll b$, there exists $\theta > 0$ such that $a \leq \sigma + \theta\lambda (f^J(x,\sigma,x',\sigma') + f^\veps(\sigma) ) \leq b$. On the one hand, thanks to global boundedness of $J$, one has
	\[
	-2\|J\|_\infty \leq f^J(x,u,x',\sigma') \leq 2\|J\|_\infty
	\]
	On the other hand, writing any $\sigma \in \sigmaSet$ as $\sigma(u) = \xi(u) a + (1-\xi(u))b$, with $0 \leq \xi \leq 1$, and using convexity of $x \mapsto x\log x$, one has
	\[
	\int_U \sigma(v)\log(\sigma(v))\,d\eta(v) \leq \int_U \xi(u)a\log a\,d\eta(u) + \int_U (1-\xi(u))b\log b \,d\eta(u).
	\]
	The integral of $\xi$ can be explicitly computed (recall $\|\sigma\|_{L^1_\eta(U)} = 1$), so that
	\begin{equation}\label{eq:mainIntLog}
	0 \leq \int_U \sigma(v)\log(\sigma(v))\,d\eta(v) \leq  \frac{a(b-1)}{b-a}\log(a) + \frac{b(1-a)}{b-a}\log(b) =: K_{a,b}
	\end{equation}
	where non-negativity follows by Jensen's inequality. The proof boils then down to find suitable $a,b$ such that for any given $\sigma \in \sigmaSet$ there exists $\theta>0$ such that
	\begin{align*}
	a \leq\, &\sigma(u) (1 - \theta\lambda\left[ 2\|J\|_\infty + \varepsilon\log(\sigma(u)) \right] ) \tag{L}\label{eq:L} \\
	&\sigma(u) \left(1 + \theta\lambda \left[ 2\|J\|_\infty + \varepsilon(K_{a,b} -\log(\sigma(u)) \right] \right) \,\leq b\tag{U}\label{eq:U}
	\end{align*}
	for every $u \in U$. Select $a,b$ in the following way:
	\begin{equation}\label{eq:AandB}
	0 < a < \exp\left(-\frac{2\|J\|_\infty}{\varepsilon}\right) = L^\ell_\varepsilon < L^u_\varepsilon = \exp\left(\frac{2\|J\|_\infty}{\varepsilon} + K_{a,b}\right) < b
	\end{equation}
	The existence of $b$ satisfying this requirement follows by a direct one-dimensional asymptotic analysis.
	We distinguish three regions:
	\begin{itemize}
		\item $a \leq \sigma(u) < L_\varepsilon^\ell$: the lower bound \eqref{eq:L} is satisfied for any $\theta>0$ because
		$
		\left[2\|J\|_\infty + \varepsilon\log(\sigma(u))\right]\allowbreak < 0
		$
		in the given region, while the upper bound \eqref{eq:U} follows taking
		\[
		\theta \leq \frac{b-L_\varepsilon^\ell}{\lambda L_\varepsilon^\ell(2\|J\|_\infty + \varepsilon\log (b/a))}
		\]
		\item $L_\varepsilon^\ell \leq \sigma(u) \leq L_\varepsilon^u$: a direct estimate on \eqref{eq:L} and \eqref{eq:U} leads to
		\[
		\theta \leq \min\left( \frac{L_\varepsilon^\ell-a}{\lambda L_\varepsilon^\ell(2\|J\|_\infty + \varepsilon\log b)},  \frac{b-L_\varepsilon^u}{\lambda L_\varepsilon^u(2\|J\|_\infty + \varepsilon\log (b/a))} \right)
		\]
		\item $L_\varepsilon^u < \sigma(u) \leq b$: the lower bound \eqref{eq:L} is satisfied for
		\[
		\theta \leq \frac{L_\varepsilon^u-a}{\lambda L_\varepsilon^u(2\|J\|_\infty + \varepsilon\log b)}
		\]
		while upper bound \eqref{eq:U} is satisfied for any $\theta > 0$ because $\left[ 2\|J\|_\infty + \varepsilon(K_{a,b} -\log(\sigma(u)) \right] < 0$ in the given region.
	\end{itemize}
	The sought-after $\theta$ can then be selected to be the smallest between the previous upper bounds.
\end{proof}

\subsection{Convergence to fast-reaction limit in undisclosed setting}
\label{sec:NonTelepathicFastReactionProofs}

In this Section we prove Theorem \ref{thm:LambdaQuasiStaticConvergence}, which is concerned with the convergence to the undisclosed fast-reaction limit as $\lambda \to \infty$.

\begin{proof}[Proof of Theorem \ref{thm:LambdaQuasiStaticConvergence}]\hfill\\
	\emph{Part 1: $\sigma_i(t)$ close to $\sigma^J_i(x(t))$.}
	For $i \in \{1,\ldots,N\}$ the dynamics of $\sigma_i$ are described by
	\begin{align}
	\label{eq:sigmaDotExpl}
	\partial_t \sigma_i(t)(u) & = \lambda \cdot \Big(-g'_{i,u}(t,\sigma_i(t)(u))+\underbrace{\int_U g'_{i,u}(\sigma_i(t)(u))\,\sigma_i(t)(u)\,\diff \eta(u)}_{\assign m_i(t)} \Big) \cdot \sigma_i(t)(u) \\
	\tn{where} \quad g_{i,u}(t,s) & \assign -\frac{1}{N}\sum_{j=1}^N J(x_i(t),u,x_j(t))\,s
	+ \veps\,\left[s\,\log(s)-s+1\right] \\
	\intertext{and $g_{i,u}'$ denotes the first derivative with respect to $s$. This is formally the gradient flow (scaled by $\lambda$) of the (explicitly time-dependent) energy}
	G_i(t,\sigma) & \assign \int_U g_{i,u}(t,\sigma(u))\,\diff \eta(u)
	\end{align}
	with respect to the (spherical) Hellinger--Kakutani distance over $\sigmaSet$. (This relation only serves as intuition for our proof and is never used in a mathematical argument.)
	
	Let $a, b$ such that $0 < a < b < \infty$. Then there exists a $c \in (0,\infty)$ such that $g_{i,u}''(t,s) \geq c$ for $t \in [0,\infty)$, $s \in [a,b]$, $u \in U$ and $i \in \{1,\ldots,N\}$ (where analogous to above $g_{i,i}''(t,s)$ refers to the second derivative with respect to $s$).
	This implies that $G_i(t,\cdot)$ is strongly convex on $\sigmaSet$. 
	
	Consider now the fast-reaction mixed strategies $\bm{\sigma}^\ast=(\sigma^\ast_i)_{i=1}^N$ associated with locations $\mathbf{x}=(x_i)_{i=1}^N$, \eqref{eq:QuasiStaticGameModelSigma}, (i.e.~$\sigma_i^\ast(t) \assign \sigma^J_i(\mathbf{x}(t))$) where we introduce a normalization constant $\mc{A}_i$:
	\begin{align}
	\sigma_i^\ast(t)(u) & = \mc{A}_i(t) \cdot
	\exp\left(\frac{1}{N\,\veps} \sum_{j=1}^N J(x_i(t),u,x_j(t)) \right), \\
	\mc{A}_i(t)^{-1} & = \int_U \exp\left(\frac{1}{N\,\veps} \sum_{j=1}^N J(x_i(t),u,x_j(t)) \right) \diff \eta(u).
	\end{align}		
	For these we find
	\begin{align}
	g_{i,u}'(t,\sigma_i^\ast(t)(u)) & = -\frac{1}{N} \sum_{j=1}^N J(x_i(t),u,x_j(t)) + \veps\,\log(\sigma_i^\ast(t)(u))
	= \veps\,\log(\mc{A}_i(t)).
	\end{align}
	Using convexity of the functions $g_{i,u}(t,\cdot)$ we find for arbitrary $\sigma \in \sigmaSet$:
	\begin{align}
	G_i(t,\sigma) & \geq
	\int_U \big[ g_{i,u}(t,\sigma_i^\ast(t)(u)) + \underbrace{g_{i,u}'(t,\sigma_i^\ast(t)(u))}_{=\veps\,\log(\mc{A}_i(t))} \cdot 
	(\sigma(u)-\sigma_i^\ast(t)(u)) \big] \,\diff \eta(u)
	= G_i(t,\sigma_i^\ast(t))
	\label{eq:QuasiStaticConvergenceG}
	\end{align}
	where the second term integrates to zero, since $\veps\,\log(\mc{A}_i(t))$ is constant with respect to $u \in U$ and both $\sigma$ and $\sigma_i^\ast(t)$ are normalized. Therefore, $\sigma_i^\ast(t)$ is a minimizer of $G_i(t,\cdot)$ over $\sigmaSet$, and since $G_i(t,\cdot)$ is strongly convex, it is the unique minimizer.
	
	Using the bound $c \leq g_{i,u}''(t,s)$ for $s \in [a,b]$ we find for $s^\ast \in [a,b]$
	\begin{align}
	g_{i,u}'(t,s^\ast) \cdot (s-s^\ast) + \tfrac{c}{2} (s-s^\ast)^2
	\leq g_{i,u}(t,s)-g_{i,u}(t,s^\ast)
	\end{align}
	and applying this bound to $G_i(t,\cdot)$, using $\sigma_i^\ast(t) \in \sigmaSet$ due to \eqref{eq:AandB}, analogous to \eqref{eq:QuasiStaticConvergenceG} one obtains for any $\sigma \in \sigmaSet$ that
	\begin{align}
	\label{eq:QuasiStaticConvergenceGInterval}
	\frac{c}{2} \|\sigma-\sigma_i^\ast(t)\|^2_{L^2_\eta(U)}
	\leq
	G_i(t,\sigma) - G_i(t,\sigma_i^\ast(t)).
	\end{align}
	Strong convexity of $G_i(t,\cdot)$ provides another useful estimate (where $m_i$ is defined in \eqref{eq:sigmaDotExpl}):
	\begin{multline*}
	\|g_{i,u}'(t,\sigma_i(t))-m_i(t)\|_{L^2_\eta(U)} \cdot \|\sigma_i(t)-\sigma_i^\ast(t)\|_{L^2_\eta(U)} \\
	\geq 
	\int_U \big[g'_{i,u}(t,\sigma_i(t)(u))-m_i(t)\big] \cdot \big[\sigma_i(t)(u)-\sigma_i^\ast(t)(u)\big] \,\diff \eta(u) \\
	= \int_U g'_{i,u}(t,\sigma_i(t)(u)) \cdot \big[\sigma_i(t)(u)-\sigma_i^\ast(t)(u)\big] \,\diff \eta(u)
	\geq G_i(t,\sigma_i(t))-G_i(t,\sigma_i^\ast(t))
	\end{multline*}
	where we used the Cauchy-Schwarz inequality.
	With \eqref{eq:QuasiStaticConvergenceGInterval} this implies
	\begin{align}
	\label{eq:QuasiStaticConvergenceGGradient}
	\|g_{i,u}'(t,\sigma_i(t))-m_i(t)\|_{L^2_\eta(U)}^2
	\geq \frac{c}{2} [G_i(t,\sigma_i(t))-G_i(t,\sigma_i^\ast(t))].
	\end{align}
	
	We need two additional estimates to control the explicit time-dependency of $G_i$. Note that $\|\partial_t x_i(t)\| \leq \|e\|_{\infty}$ and thus $|\partial_t g_{i,u}(t,s)| \leq 2 \cdot \Lip(J) \cdot \|e\|_\infty$. ($J$ may not be differentiable, but since $t \mapsto x_i(t)$ and $J$ are Lipschitz, so is $g_{i,u}$ and thus this derivative exists for almost all $t$.)
	This implies $|\partial_t G_i(t,\sigma)| \leq 2 \cdot \Lip(J) \cdot \|e\|_\infty$. Next, an explicit computation yields
	\begin{align*}
	G_i(t,\sigma_i^\ast(t)) & = \veps \cdot \log \mc{A}_i(t)
	= -\veps \log \left(\int_U \exp\left(\frac{1}{N\,\veps} \sum_{j=1}^N J(x_i(t),u,x_j(t)) \right) \diff \eta(u) \right)
	\end{align*}
	from which we deduce that $|\tfrac{\diff}{\diff t} G_i(t,\sigma_i^\ast(t))| \leq 2 \cdot \Lip(J) \cdot \|e\|_\infty$.
	
	Summarizing now, we find for almost every $t \in [0,\infty)$:
	\begin{align}
	\frac{\diff}{\diff t} \left[ G_i(t,\sigma_i(t)) - G_i(t,\sigma_i^\ast(t)) \right] &
	= \int_U g_{i,u}'(t,\sigma_i(t)(u)) \cdot \partial_t \sigma_i(t)(u)\,\diff \eta(u) \nonumber \\
	& \qquad + (\partial_t G_i)(t,\sigma_i(t)) - \tfrac{\diff}{\diff t} G_i(t,\sigma_i^\ast(t)) \nonumber \\
	& = \int_U [g_{i,u}'(t,\sigma_i(t)(u))-m_i(t)] \cdot \partial_t \sigma_i(t)(u)\,\diff \eta(u) \nonumber \\
	& \qquad + (\partial_t G_i)(t,\sigma_i(t)) - \tfrac{\diff}{\diff t} G_i(t,\sigma_i^\ast(t)) \nonumber \\
	& = -\lambda \int_U [g_{i,u}'(t,\sigma_i(t)(u))-m_i(t)]^2 \cdot \underbrace{\sigma_i(t)(u)}_{\geq a}
	\,\diff \eta(u) \nonumber \\
	& \qquad + (\partial_t G_i)(t,\sigma_i(t)) - \tfrac{\diff}{\diff t} G_i(t,\sigma_i^\ast(t)) \nonumber \\
	& \leq -\frac{\lambda\,a\,c}{2} \left[G_i(t,\sigma_i(t))-G_i(t,\sigma_i^\ast(t))\right]
	+4 \cdot \Lip(J) \cdot \|e\|_\infty.
	\end{align}
	Applying Gr\"onwall's lemma to $h(t) \assign G_i(t,\sigma_i(t)) - G_i(t,\sigma_i^\ast(t)) - \tfrac{8\,\Lip(J)\,\|e\|_\infty}{\lambda\,a\,c}$ yields
	\begin{multline}
	G_i(t,\sigma_i(t)) - G_i(t,\sigma_i^\ast(t)) \leq \tfrac{8\,\Lip(J)\,\|e\|_\infty}{\lambda\,a\,c} \\
	+ \left[ \left(G_i(0,\sigma_i(0)) - G_i(0,\sigma_i^\ast(0))\right)- \tfrac{8\,\Lip(J)\,\|e\|_\infty}{\lambda\,a\,c}\right] \cdot \exp\left(-\tfrac{\lambda\,a\,c}{2}t\right).
	\end{multline}
	Using $\sigma_i(0) \in \sigmaSet$ and $x_i(0) \in \dball{R}$ for all $i \in \{1,\ldots,N\}$ there exists some $C = C(a,b,R,\allowbreak\Lip(J),\allowbreak\|e\|_\infty) < \infty$ (not depending on $N$ or $i$) such that
	\begin{align}
	G_i(t,\sigma_i(t)) - G_i(t,\sigma_i^\ast(t)) \leq C \left[\tfrac{1}{\lambda}+ \exp\left(-\tfrac{\lambda\,t}{C}\right)\right].
	\end{align}
	By virtue of \eqref{eq:QuasiStaticConvergenceGInterval} an analogous bound with a different $C$ (with the same dependency structure) holds for $\|\sigma_i(t)-\sigma_i^\ast(t)\|_{L^2_\eta(U)}$:
	\begin{align}
	\label{eq:QuasiStaticSigmaBound}
	\|\sigma_i(t)-\sigma_i^\ast(t)\|_{L^2_\eta(U)} \leq C \left[\tfrac{1}{\lambda}+ \exp\left(-\tfrac{\lambda\,t}{C}\right)\right]
	\end{align}
	
	\emph{Part 2: $x_i$ and $\sigma_i$ close to $x^{\ast \ast}_i$ and $\sigma^{\ast \ast}_i$.}
	Recall now that the fast-reaction dynamics \eqref{eq:QuasiStaticGameModel} for $\mathbf{x}^{\ast \ast}$ are given by
	\begin{align*}
	\partial_t x_i^{\ast \ast}(t) & = \int_U e(x_i^{\ast \ast}(t),u)\,\sigma_i^{\ast \ast}(t)(u)\,\diff \eta(u) =
	v^J_i(x_1^{\ast \ast}(t),\ldots,x_N^{\ast \ast}(t)) \\
	\intertext{where $\sigma^{\ast \ast}_i(t)=\sigma^J_i(\mathbf{x}^{\ast\ast}(t))$, \eqref{eq:QuasiStaticGameModelSigma}.
		The dynamics for $\mathbf{x}(t)$ follow}
	\partial_t x_i(t) & = \int_U e(x_i(t),u)\,\sigma_i(t)(u)\,\diff \eta(u) =
	v^J_i(x_1(t),\ldots,x_N(t)) \nonumber \\
	& \qquad + \int_U e(x_i(t),u)\,[\sigma_i(t)(u)-\sigma_i^\ast(t)(u)]\,\diff \eta(u).
	\end{align*}
	Introduce now the deviation measure $\delta(t) \assign \|\mathbf{x}(t)-\mathbf{x}^{\ast \ast}(t)\|_N$ and obtain ($t$-almost everywhere, due to $\delta$ being Lipschitz)
	\begin{align}
	\partial_t \delta(t) & \leq \|\partial_t \mathbf{x}(t)-\partial_t \mathbf{x}^{\ast \ast}(t)\|_N
	\nonumber \\
	& \leq \|\mathbf{v}^J(\mathbf{x}(t))-\mathbf{v}^J(\mathbf{x}^{\ast \ast}(t))\|_N + \sum_{i=1}^N \frac{1}{N} \left\|
	\int_U e(x_i(t),u)\,[\sigma_i(t)(u)-\sigma_i^\ast(t)(u)]\,\diff \eta(u) \right\|
	\nonumber \\
	& \leq \Lip(\mathbf{v}^J) \cdot \underbrace{\|\mathbf{x}(t)-\mathbf{x}^{\ast \ast}(t)\|_N}_{= \delta(t)}
	+ \|e\|_\infty \cdot \sum_{i=1}^N \frac{1}{N} 
	\!\!\!\!\! \!\!\! \!\!\! \underbrace{\|\sigma_i(t)-\sigma_i^\ast(t)\|}_{
		\leq C_{t,\lambda} \assign\left[C \left[\tfrac{1}{\lambda}+
		\exp\left(-\tfrac{\lambda\,t}{C}\right)\right]\right]^{1/2}}
	\label{eq:QuasiStaticDeltaBound}
	\end{align}
	For every $\tau>0$ we get $\delta(t) \leq 2\,\tau\,\|e\|_\infty$ on $t \in [0,\tau]$ (due to $\|\partial_t x_i\|, \|\partial_t x^{\ast \ast}_i\| \leq \|e\|_\infty$) and for $t \in [\tau,\infty)$ from \eqref{eq:QuasiStaticDeltaBound} via Grönwall's lemma (with a change of variables as above)
	\begin{align}
	\delta(t) \leq \left( 2\,\tau\,\|e\|_\infty + \tfrac{C_{\tau,\lambda}}{\Lip(\mathbf{v}^J)} \right)
	\cdot \exp(\Lip(\mathbf{v}^J) \cdot (t-\tau))-\tfrac{C_{\tau,\lambda}}{\Lip(\mathbf{v}^J)}.
	\end{align}
	Setting now $\tau=\sqrt{\lambda}^{-1}$, there is a suitable $C$ (depending on $a$, $b$, $R$, $J$, $e$, but not on $N$ or $i$) such that
	\begin{align}
	\label{eq:QuasiStaticDeltaBoundFinal}
	\delta(t) \leq \frac{C}{\sqrt{\lambda}} \cdot \exp(t \cdot C).
	\end{align}
	From Lemma \ref{lemma:UniformSigma} it follows that (cf.~Lemma \ref{lemma:LipschitzVelocityMap})
	\begin{align}
	\label{eq:QuasiStaticSigmaLip}
	\|\sigma_i^\ast(t)-\sigma_i^{\ast\ast}(t)\|_{L^2_\eta(U)} = \|\sigma_i^J(\mathbf{x}(t))-\sigma_i^{J}(\mathbf{x}^{\ast\ast}(t))\|_{L^2_\eta(U)} \leq L \cdot \|\mathbf{x}(t))-\mathbf{x}^{\ast\ast}(t)\|_N
	\end{align}
	for some $L<\infty$ depending on $J$ and $\veps$.
	Combining \eqref{eq:QuasiStaticSigmaBound}, \eqref{eq:QuasiStaticDeltaBoundFinal} and \eqref{eq:QuasiStaticSigmaLip} one obtains that (for different $C$ with same dependency structure)
	\begin{align}
	\|\sigma_i(t)-\sigma_i^{\ast\ast}(t)\|_{L^2_\eta(U)} \leq \frac{C}{\sqrt{\lambda}} \cdot \exp(t \cdot C) + 
	\left[C \left[\tfrac{1}{\lambda}+ \exp\left(-\tfrac{\lambda\,t}{C}\right)\right]\right]^{1/2}.
	\end{align}
	\emph{Part 3: Mean-field setting.}
	For $\Sigma_1, \Sigma_2 \in \prob(\fullSpace)$ one has
	\begin{align}	
	\label{eq:WassersteinProjection}
	W_1(\proj_\sharp \Sigma_1, \proj_\sharp \Sigma_2) \leq W_1(\Sigma_1,\Sigma_2)
	\end{align}
	due to
	\begin{align*}
	\|x_1-x_2\|=\|P(x_1,\sigma_1)-P(x_2,\sigma_2)\| \leq \|x_1-x_2\| + \|\sigma_1-\sigma_2\|_{L^p_\eta(U)}=\|(x_1,\sigma_1)-(x_2,\sigma_2)\|_Y.
	\end{align*}
	
	Let $(\bar{\Sigma}^N)_{N}$ be a sequence of initial empirical measures with $W_1(\bar{\Sigma}^N,\bar{\Sigma}) \to 0$ as $N \to \infty$ and let $\Sigma^N$ be the corresponding solution of \eqref{eq:ContEqY}, as discussed in Theorem \ref{thm:wellPosedEntropic}, point \ref{item:wellPosedEntropicDiscrete}.
	Set $\bar{\mu}^N \assign \proj_\sharp \bar{\Sigma}^N$, which consequently satisfies $W_1(\bar{\mu}^N,\bar{\mu}) \to 0$ as $N \to \infty$ and let $\mu^N$ be the corresponding solution of \eqref{eq:ContEqR}.
	By Theorem \ref{thm:wellPosedEntropic}, point \ref{item:wellPosedEntropicStability} and Theorem \ref{thm:wellPosedFastReaction}, point \ref{item:wellPosedEntropicStability} (both for $s = 0$), for any $\alpha>0$ there is some $N$ such that
	\begin{align}
	\label{eq:QuasiStaticConvergenceDiscreteApprox}
	W_1(\Sigma^N(t),\Sigma(t)) & \leq \alpha, &
	W_1(\mu^N(t),\mu(t)) & \leq \alpha
	\end{align}
	for $t \in [0,T]$.
	
	The points in $\Sigma^N$ evolve according to the full discrete model, the points in $\mu^N$ evolve according to the fast-reaction limit. Therefore, by virtue of \eqref{eq:QuasiStaticDeltaBoundFinal} one has
	\begin{align}
	\label{eq:QuasiStaticConvergenceDiscreteWasserstein}
	W_1(\proj_\sharp \Sigma^N(t),\mu^N(t)) \leq \frac{C}{\sqrt{\lambda}} \cdot \exp(t \cdot C)
	\end{align}
	with the same $C$ as in \eqref{eq:QuasiStaticDeltaBoundFinal}.
	Combining \eqref{eq:QuasiStaticConvergenceDiscreteWasserstein}, \eqref{eq:QuasiStaticConvergenceDiscreteApprox} and \eqref{eq:WassersteinProjection} one obtains
	\begin{multline*}
	W_1(\proj_\sharp \Sigma(t),\mu(t)) \leq W_1(\proj_\sharp \Sigma(t), \proj_\sharp \Sigma^N(t))
	+ W_1(\proj_\sharp \Sigma^N(t),\mu^N(t)) + W_1(\mu^N(t),\mu(t)) \\
	\leq 2\alpha + \frac{C}{\sqrt{\lambda}} \cdot \exp(t \cdot C).
	\end{multline*}
	The fact that $\alpha>0$ is arbitrary, means the inequality also holds in the limit $\alpha \to 0$ and thus establishes \eqref{eq:QuasiStaticConvergenceMeanField}.
\end{proof}

\subsection{Undisclosed fast-reaction inference functionals}
\label{sec:InferenceProofs}
We report in this section some technical proofs for statements in Section \ref{sec:InferenceNTFR}. We start with proving that observations generated through explicit simulations of the undisclosed fast-reaction system are indeed admissible in the sense of Assumption \ref{ass:Observations}.

\begin{proof}[Proof of Lemma \ref{lemma:ExplicitForward}]
	\hfill\\
	The first statement is just point \ref{item:QuasiStaticDiscrete} in Theorem \ref{thm:wellPosedFastReaction}. The proof of the second statement is a direct by-product of Theorem \ref{thm:wellPosedFastReaction} combined with Lemma \ref{lemma:UniformSigma}. Let $\mu^N$ and $\mu$ be the unique solutions of \eqref{eq:ContEqR} for initial conditions $\bar{\mu}^N$ and $\bar{\mu}$, as provided by Theorem \ref{thm:wellPosedFastReaction}, point \ref{item:wellPosedEntropicContinuous}. By the stability estimate in Theorem \ref{thm:wellPosedFastReaction}, point \ref{item:wellPosedEntropicStability}, one has for any $t \in [0,T]$
	\begin{equation}\label{eq:ConvergenceToMuInf}
	W_1(\mu^N(t), \mu(t)) \leq C \, W_1(\bar{\mu}^N, \bar{\mu}) \to 0 \text{ as } N \to \infty.
	\end{equation}
	Hence, we claim $\mu^\infty = \mu$ and in particular $v^\infty = v^J(\mu)$ and $\sigma^\infty = \sigma^J(\mu)$ (as defined in \eqref{eq:MeanFieldVel} and \eqref{eq:MeanFieldSigma}). Regarding convergence of velocities, observe that by definition
	\[
	v_i^N(t) = v_i^J(x_1^N(t), \dots, x_N^N(t)) = v^J(\mu^N(t))(x_i^N(t))
	\]
	while \eqref{eq:VelocityLip} in Lemma \ref{lemma:UniformSigma} provides the uniform bound
	\[
	\left\| v^{J}(\mu^N(t))(x) - v^{J}(\mu(t))(x) \right\| \leq C\, W_1(\mu^N(t),\mu(t)).
	\]
	This yields directly \eqref{eq:ObservationsVWeak} while \eqref{eq:ObservationsVConvex} follows thanks to the uniform bound $\|v^J(\mu(t))(x)\| \leq \|e\|_\infty$. An analogous argument applies to prove \eqref{eq:ObservationsSigmaWeak} and \eqref{eq:ObservationsSigmaConvex}: first recall that
	\[
	\sigma_i^N(t) = \sigma_i^J(x_1^N(t), \dots, x_N^N(t)) = \sigma^J(\mu^N(t))(x_i^N(t),\cdot)
	\]
	and thanks to \eqref{eq:SigmaLip} in Lemma \ref{lemma:UniformSigma} one has the uniform bounds
	\[
	\left| \sigma^{J}(\mu^N(t))(x,u) - \sigma^{J}(\mu(t))(x,u) \right| \leq C\, W_1(\mu^N(t),\mu(t))
	\]
	and $1/C < \sigma^{J}(\mu^N)(x,u) < C$, which immediately provides the conclusions when combined with uniform continuity of $s \mapsto s\log(s)$ over $[1/C, C]$.
\end{proof}

Next, following the pipeline provided by \cite[Proposition 1.1]{BFHM16}, we show that the obtained minimal inference functional value provides an upper bound on the accuracy of the trajectories that are simulated with the inferred $\hat{J}$.
\begin{proof}[Proof of Theorem \ref{thm:BoundOnTrajectories}]
	\hfill\\
	\emph{Part 1: discrete velocity-error functional \eqref{eq:energyJNxdot}.}
	Let $\mu^N(t) = \frac1N \sum_{i=1}^N \delta_{x_i^N(t)}$. For every $t \in [0,T]$, we estimate
	\begin{align}
	&\|\mathbf{x}^N(t)-\hat{\mathbf{x}}^N(t)\|_N = \left\| \int_0^t (\partial_t \mathbf{x}^N(s) - \partial_t \hat{\mathbf{x}}^N(s))\,\diff s\right\|_N \leq \int_0^t \left\| \partial_t \mathbf{x}^N(s) - \partial_t \hat{\mathbf{x}}^N(s)\right\|_N\,\diff s \nonumber\\
	&= \int_0^t \frac1N \sum_{i=1}^N \left\| v_{i}^N(s) - v_i^{\hat{J}}(\hat{\mathbf{x}}^N(s)) \right\|\,\diff s\nonumber\\
	&\leq \int_0^t \Bigg[ \frac1N \sum_{i=1}^N \left\| v_i^N(s) - v_i^{\hat{J}}(\mathbf{x}^N(s)) \right\| +\frac1N \sum_{i=1}^N \left\| v_i^{\hat{J}}({\mathbf{x}}^N(s)) - v_i^{\hat{J}}(\hat{\mathbf{x}}^N(s)) \right\|\Bigg]\diff s\nonumber\\
	&\leq T\left( \frac{1}{TN} \int_0^T \sum_{i=1}^N \left\| v_{i}^N(s) - v_i^{\hat{J}}(\mathbf{x}^N(s)) \right\|^2 \,\diff s \right)^{1/2} + L \int_0^t  \|\mathbf{x}^N(s)-\hat{\mathbf{x}}^N(s)\|_N\,\diff s\nonumber\\
	&\leq T  \sqrt{\energy_v^{N}(\hat{J})} + L \int_0^t \|\mathbf{x}^N(s)-\hat{\mathbf{x}}^N(s)\|_N\,\diff s\nonumber
	\end{align}
	where we used Jensen's inequality (or Cauchy-Schwarz) from the third to the fourth line and with $L = L(\hat{J}, e, \veps)$ denoting the Lipschitz constant of the velocities (cf. Lemma \ref{lemma:LipschitzVelocityMap}).
	An application of Gr\"onwall's inequality yields
	\begin{equation}\label{eq:TrajectoryBoundV}
	\|\mathbf{x}^N(t)-\hat{\mathbf{x}}^N(t)\|_N \leq Te^{LT} \cdot \sqrt{\energy_v^{N}(\hat{J})}.
	\end{equation}
	
	\noindent
	\emph{Part 2: discrete $\sigma$-error functional \eqref{eq:energyJNsigma}.}
	To obtain the same estimate for the discrete $\sigma$-error functional \eqref{eq:energyJNsigma} we first recall Pinsker's inequality (cf. \cite[Lemma 2.5]{TsybakovNonparametricEstimation}): for $\sigma_1,\sigma_2 \in \sigmaSet$ one has
	\[
	\|\sigma_1 - \sigma_2 \|_{L^1_\eta(U)} \leq \sqrt{2\KL(\sigma_1\mid\sigma_2)}.
	\]
	Thus, for any $t \in [0,T]$, we estimate
	\begin{align}
	&\left\| v_i^{N}(t) - v_i^{\hat{J}}({\mathbf{x}^N(t)}) \right\|^2 = \left\| \int_U e(x_i^N(t),u) \left(\sigma_i^N(t)(u) - \sigma_i^{\hat{J}}(\mathbf{x}^N(t))(u) \right)\,d\eta(u) \right\|^2 \nonumber\\
	& \leq \left( \int_U \|e(x_i^N(t),u)\| \left|\sigma_i^N(t)(u) - \sigma_i^{\hat{J}}(\mathbf{x}^N(t))(u) \right|\,d\eta(u) \right)^2 \nonumber\\
	& \leq \|e\|_\infty^2 \| \sigma_i^N(t) - \sigma_i^{\hat{J}}(\mathbf{x}^N(t)) \|_{L^1_\eta(U)}^2 \leq 2\|e\|^2_\infty \KL(\sigma_i^N(t) \mid \sigma_i^{\hat{J}}(\mathbf{x}^N(t)))\nonumber
	\end{align}
	This provides
	$
	\energy_v^{N}(\hat{J}) \leq 2\|e\|^2_\infty \energy_\sigma^{N}(\hat{J})
	$
	and the conclusion follows from \eqref{eq:TrajectoryBoundV}.
	
	\noindent
	\emph{Part 3: extension to continuous functionals.}
	Let $(\hat{\textbf{x}}^N)_N$ be the sequence of discrete solutions with initial conditions $(\textbf{x}^N(0))_N$ induced by $\hat{J}$ and $(\hat{\mu}^N)_N$ the corresponding sequence of empirical measures. By Assumption \ref{ass:Observations}, point \ref{item:ObservationsLimit}, we have $W_1(\mu^N(t),\mu^\infty(t)) \to 0$ uniformly for $t \in [0,T]$, in particular $W_1(\mu^N(0),\mu^\infty(0)) \to 0$. Hence, by Theorem \ref{thm:wellPosedFastReaction}, point \ref{item:QuasiStaticStability}, $W_1(\hat{\mu}^N(t),\hat{\mu}(t)) \leq C W_1(\mu^N(0),\mu^\infty(0))  \to 0$ as $N \to \infty$ uniformly for $t \in [0,T]$. Then
	\begin{align*}
	W_1(\mu^\infty(t),\hat{\mu}(t)) & = \lim_{N \to \infty} W_1(\mu^N(t),\hat{\mu}^N(t))
	\leq \liminf_{N \to \infty} \|\textbf{x}^N(t)-\hat{\textbf{x}}^N(t)\|_N \\
	& \leq \liminf_{N \to \infty} C \sqrt{\energy^{N}_v(\hat{J})}
	= C \sqrt{\energy_v(\hat{J})}
	\end{align*}
	where we applied the discrete result in the third step and used Lemma \ref{lemma:UniformConvergence}, for the constant sequence $(\hat{J})_N$, in the fourth step. The same argument covers the $\sigma$-functionals \eqref{eq:energysigma}.
\end{proof}

We close this section by proving convexity of the $\sigma$-based inference functionals $\energy^{N}_\sigma$ and $\energy_\sigma$.
\begin{proposition}[Convexity]
	\label{prop:energysigmaConvex}
	$\energy^{N}_\sigma$ and $\energy_\sigma$, \eqref{eq:energysigma}, are convex on $X$.
\end{proposition}

\begin{proof}
	We fix a time $t \in [0,T]$ and ignore the $t$-dependency for notational simplicity.
	Also recall that all mixed strategy densities in $\sigmaSet$ are strictly bounded away from zero, making quotients and logarithms in the following well-defined.
	For $\textbf{x} = (x_1,\dots,x_N)\in [\R^d]^N$ and any $i \in \{1,\dots,N\}$ we compute
	\begin{align*}
	& \KL\big(\sigma_i|\sigma_i^{\hat{J}}(x_1,\ldots,x_N)\big)
	= \int_U \log\left( \frac{\sigma_i(u)}{\sigma_i^{\hat{J}}(x_1,\ldots,x_N)(u)} \right)\sigma_i(u) \diff \eta(u)
	\\
	& = \int_U \left[
	\sigma_i(u)\log\sigma_i(u)
	- \sigma_i(u)\log\left(\sigma_i^{\hat{J}}(x_1,\ldots,x_N)(u)\right)
	\right]
	\diff \eta(u) \\
	& = \int_U \left[
	\sigma_i(u)\log\sigma_i(u)
	- \sigma_i(u)\log\left(\tfrac{\exp\left(\tfrac{1}{\veps N}\sum_{j=1}^N \hat{J}(x_i,u,x_j)\right)}
	{\int_U \exp\left(\tfrac{1}{\veps N}\sum_{j=1}^N \hat{J}(x_i,v,x_j)\right)\,\diff \eta(v)}\right)
	\right]
	\diff \eta(u) \\
	& = \log\left(
	\int_U \! \exp\!\left(\frac{1}{\veps N}\sum_{j=1}^N \hat{J}(x_i,v,x_j)\right)\! \diff \eta(v)
	\right)
	\!+\! \int_U \sigma_i(u)\! \left(
	\log\sigma_i(u) - \frac{1}{\veps N} \sum_{j=1}^N \hat{J}(x_i,u,x_j)
	\right)\!\diff \eta(u).
	\end{align*}
	Convexity of the first term follows by H\"older's inequality: let $\hat{J}_1, \hat{J}_2 \in X$ and set $S_1(v) \assign \frac{1}{\veps N}\sum_{j=1}^N \hat{J}_1(x_i,v,x_j)$ and $S_2(v) \assign \frac{1}{\veps N}\sum_{j=1}^N \hat{J}_2(x_i,v,x_j)$, for $\alpha \in (0,1)$ we have
	\begin{align*}
	& \log\left( \int_U \exp\left(\alpha S_1(v) + (1-\alpha)S_2(v)\right)\,\diff \eta(v) \right) = \log\left( \int_U \exp(\alpha S_1(v)) \cdot \exp((1-\alpha)S_2(v))\,\diff \eta(v) \right) \\
	&\leq \log\left( \left( \int_U \exp(\alpha S_1(v))^{1/\alpha}\,\diff\eta(v)\right)^\alpha \cdot \left( \int_U \exp((1-\alpha) S_2(v))^{1/(1-\alpha)}\,\diff\eta(v)\right)^{1-\alpha} \right) \\
	&= \alpha \log\left( \int_U \exp( S_1(v) )\,\diff \eta(v) \right) + (1-\alpha) \log\left( \int_U \exp( S_2(v) )\,\diff \eta(v) \right).
	\end{align*}
	The second term is constant and the last term is linear in $\hat{J}$, and so $\KL\big(\sigma_i|\sigma_i^{\hat{J}}(x_1,\ldots,x_N)\big)$ is convex in $\hat{J}$.
	Summing over $i$ (for \eqref{eq:energyJNsigma}) or integrating over $\mu^\infty(t)$ (for \eqref{eq:energyJsigma}) and then integrating in time preserves convexity and concludes the proof.
\end{proof}

\bibliography{references}{}

\begin{thebibliography}{10}

\bibitem{pers_comm}
V.~Agostiniani, Riccarda Rossi, and Giuseppe Savar\'{e}.
\newblock Singular vanishing-viscosity limits of gradient flows in {H}ilbert
  spaces.
\newblock {\em personal communication: in preparation}, 2018.

\bibitem{MR3705699}
Virginia Agostiniani and Riccarda Rossi.
\newblock Singular vanishing-viscosity limits of gradient flows: the
  finite-dimensional case.
\newblock {\em J. Differential Equations}, 263(11):7815--7855, 2017.

\bibitem{DDvolatility}
Vinicius Albani, Uri~M. Ascher, Xu~Yang, and Jorge~P. Zubelli.
\newblock Data driven recovery of local volatility surfaces.
\newblock {\em Inverse Probl. Imaging}, 11(5):799--823, 2017.

\bibitem{2639}
Stefano Almi, Massimo Fornasier, and Richard Huber.
\newblock Data-driven evolutions of critical points.
\newblock {\em Foundations of Data Science}, 2(3):207--255, 2020.

\bibitem{almi2020alternate}
Stefano Almi, Marco Morandotti, and Francesco Solombrino.
\newblock A multi-step lagrangian scheme for spatially inhomogeneous
  evolutionary games.
\newblock {\em Journal of Evolution Equations}, 21(2):2691--2733, 2021.

\bibitem{AmFoMoSa2018}
Luigi Ambrosio, Massimo Fornasier, Marco Morandotti, and Giuseppe Savar{\'e}.
\newblock Spatially inhomogeneous evolutionary games.
\newblock {\em Communications on Pure and Applied Mathematics},
  74(7):1353--1402, 2021.

\bibitem{bailo2018pedestrian}
Rafael Bailo, Jos{\'e}~A Carrillo, and Pierre Degond.
\newblock Pedestrian models based on rational behaviour.
\newblock In {\em Crowd Dynamics, Volume 1}, pages 259--292. Springer, 2018.

\bibitem{BCCCCGLOPPVZ09}
M.~Ballerini, N.~Cabibbo, R.~Candelier, A.~Cavagna, E.~Cisbani, L.~Giardina,
  L.~Lecomte, A.~Orlandi, G.~Parisi, A.~Procaccini, M.~Viale, and
  V.~Zdravkovic.
\newblock Interaction ruling animal collective behavior depends on topological
  rather than metric distance: evidence from a field study.
\newblock {\em PNAS}, 105(4):1232--1237, 2008.

\bibitem{BellomoCrowdReview11}
Nicola Bellomo and Christian Dogbe.
\newblock On the modeling of traffic and crowds: A survey of models,
  speculations, and perspectives.
\newblock {\em SIAM Rev.}, 53(3):409--463, 2011.

\bibitem{BFHM16}
Mattia Bongini, Massimo Fornasier, Markus Hansen, and Mauro Maggioni.
\newblock Inferring interaction rules from observations of evolutive systems
  {I}: The variational approach.
\newblock {\em Math. Models Methods Appl. Sci.}, 27(5):909--951, 2017.

\bibitem{CDFSTB03}
S.~Camazine, J.L. Deneubourg, N.R. Franks, J.~Sneyd, G.~Theraulaz, and
  E.~Bonabeau.
\newblock Self-organization in biological systems.
\newblock {\em Princeton University Press}, 2003.

\bibitem{CCH13}
J.~A. Carrillo, Y.-P. Choi, and M.~Hauray.
\newblock {The derivation of swarming models: mean-field limit and Wasserstein
  distances}.
\newblock In Adrian Muntean and Federico Toschi, editors, {\em Collective
  Dynamics from Bacteria to Crowds}, CISM International Centre for Mechanical
  Sciences, pages 1--46. Springer, 2014.

\bibitem{MR2507454}
J.~A. Carrillo, M.~R. D'Orsogna, and V.~Panferov.
\newblock Double milling in self-propelled swarms from kinetic theory.
\newblock {\em Kinet. Relat. Models}, 2(2):363--378, 2009.

\bibitem{CCP16}
Jos{\'e}~A. Carrillo, Young-Pil Choi, and Sergio~P. Perez.
\newblock {\em A Review on Attractive--Repulsive Hydrodynamics for Consensus in
  Collective Behavior}, pages 259--298.
\newblock Springer International Publishing, Cham, 2017.

\bibitem{cafotove10}
Jos\'e~A. Carrillo, Massimo Fornasier, Giuseppe Toscani, and Francesco Vecil.
\newblock Particle, kinetic, and hydrodynamic models of swarming.
\newblock In Giovanni Naldi, Lorenzo Pareschi, Giuseppe Toscani, and Nicola
  Bellomo, editors, {\em Mathematical Modeling of Collective Behavior in
  Socio-Economic and Life Sciences}, Modeling and Simulation in Science,
  Engineering and Technology, pages 297--336. Birkh{\"a}user Boston, 2010.

\bibitem{Carrillo2014}
Jos{\'e}~Antonio Carrillo, Young-Pil Choi, and Maxime Hauray.
\newblock {\em The derivation of swarming models: Mean-field limit and
  Wasserstein distances}, pages 1--46.
\newblock Springer Vienna, Vienna, 2014.

\bibitem{Carrillo2011}
J.~A. Cañizo, J.~A. Carrillo, and J.~Rosado.
\newblock A well-posedness theory in measures for some kinetic models of
  collective motion.
\newblock {\em Mathematical Models and Methods in Applied Sciences},
  21(03):515--539, 2011.

\bibitem{ChuDorMarBerCha07}
Y.L. Chuang, M.~D'Orsogna, Daniel Marthaler, A.L. Bertozzi, and L.S. Chayes.
\newblock State transition and the continuum limit for the 2{D} interacting,
  self-propelled particle system.
\newblock {\em Physica D}, (232):33--47, 2007.

\bibitem{ChuHuaDorBer07}
Y.L. Chuang, Y.R. Huang, M.R. D'Orsogna, and A.L. Bertozzi.
\newblock Multi-vehicle flocking: scalability of cooperative control algorithms
  using pairwise potentials.
\newblock {\em IEEE International Conference on Robotics and Automation}, pages
  2292--2299, 2007.

\bibitem{MR3799091}
S.~Conti, S.~M\"{u}ller, and M.~Ortiz.
\newblock Data-driven problems in elasticity.
\newblock {\em Arch. Ration. Mech. Anal.}, 229(1):79--123, 2018.

\bibitem{CouFra02}
I.D. Couzin and N.R. Franks.
\newblock Self-organized lane formation and optimized traffic flow in army
  ants.
\newblock {\em Proc. R. Soc. Lond.}, B 270:139--146, 2002.

\bibitem{CKFL05}
I.D. Couzin, J.~Krause, N.R. Franks, and S.A. Levin.
\newblock Effective leadership and decision making in animal groups on the
  move.
\newblock {\em Nature}, 433:513--516, 2005.

\bibitem{cre03}
S.~Cr\'{e}pey.
\newblock Calibration of the local volatility in a generalized
  {B}lack-{S}choles model using {T}ikhonov regularization.
\newblock {\em SIAM J. Math. Anal.}, 34(5):1183--1206, 2003.

\bibitem{crpito10}
Emiliano Cristiani, Benedetto Piccoli, and Andrea Tosin.
\newblock {Modeling self-organization in pedestrians and animal groups from
  macroscopic and microscopic viewpoints.}
\newblock In Giovanni Naldi, Lorenzo Pareschi, Giuseppe Toscani, and Nicola
  Bellomo, editors, {\em Mathematical Modeling of Collective Behavior in
  Socio-Economic and Life Sciences}, Modeling and Simulation in Science,
  Engineering and Technology. Birkh{\"a}user Boston, 2010.

\bibitem{crpito11}
Emiliano Cristiani, Benedetto Piccoli, and Andrea Tosin.
\newblock {Multiscale modeling of granular flows with application to crowd
  dynamics.}
\newblock {\em Multiscale Model. Simul.}, 9(1):155--182, 2011.

\bibitem{crpito14}
Emiliano {Cristiani}, Benedetto {Piccoli}, and Andrea {Tosin}.
\newblock {\em {Multiscale Modeling of Pedestrian Dynamics.}}
\newblock Cham: Springer, 2014.

\bibitem{PhysRevLett.108.120503}
Toby~S. Cubitt, Jens Eisert, and Michael~M. Wolf.
\newblock Extracting dynamical equations from experimental data is np hard.
\newblock {\em Phys. Rev. Lett.}, 108:120503, Mar 2012.

\bibitem{CucSmaZho04}
F.~Cucker, S.~Smale, and D.~Zhou.
\newblock Modeling language evolution.
\newblock {\em Found. Comput. Math.}, 4(5):315--343, 2004.

\bibitem{CuckerDong11}
Felipe Cucker and Jiu-Gang Dong.
\newblock A general collision-avoiding flocking framework.
\newblock {\em IEEE Trans. Automat. Control}, 56(5):1124--1129, 2011.

\bibitem{cucker-mordecki}
Felipe Cucker and Ernesto Mordecki.
\newblock Flocking in noisy environments.
\newblock {\em J. Math. Pures Appl. (9)}, 89(3):278--296, 2008.

\bibitem{CS}
Felipe Cucker and Steve Smale.
\newblock Emergent behavior in flocks.
\newblock {\em IEEE Trans. Automat. Control}, 52(5):852--862, 2007.

\bibitem{Degond2013}
P.~Degond, C.~Appert-Rolland, M.~Moussaïd, J.~Pettré, and G.~Theraulaz.
\newblock A hierarchy of heuristic-based models of crowd dynamics.
\newblock {\em Journal of Statistical Physics}, 152(6):1033--1068, September
  2013.

\bibitem{Egger_2005}
Herbert Egger and Heinz~W. Engl.
\newblock Tikhonov regularization applied to the inverse problem of option
  pricing: convergence analysis and rates.
\newblock {\em Inverse Problems}, 21(3):1027--1045, 2005.

\bibitem{FlamEntropicGames2005}
S.D. Fl\r{a}m and E.~Cavazzuti.
\newblock Entropic penalties in finite games.
\newblock {\em Ann Oper Res}, 137:331--348, 2005.

\bibitem{garavello2006traffic}
M.~Garavello and B.~Piccoli.
\newblock {\em Traffic Flow on Networks: Conservation Laws Model}.
\newblock AIMS series on applied mathematics. American Institute of
  Mathematical Sciences, 2006.

\bibitem{GC04}
G.~Gr\'egoire and H.~Chat\'e.
\newblock Onset of collective and cohesive motion.
\newblock {\em Phy. Rev. Lett.}, (92), 2004.

\bibitem{hekr02}
Rainer Hegselmann and Ulrich Krause.
\newblock Opinion dynamics and bounded confidence: models, analysis and
  simulation.
\newblock {\em J. Artificial Societies and Social Simulation}, 5(3), 2002.

\bibitem{zbMATH01169593}
Josef {Hofbauer} and Karl {Sigmund}.
\newblock {\em {Evolutionary Games and Population Dynamics.}}
\newblock Cambridge: Cambridge University Press, 1998.

\bibitem{HCM03}
M.~Huang, P.E. Caines, and R.P. Malham{\'e}.
\newblock {Individual and mass behaviour in large population stochastic
  wireless power control problems: centralized and Nash equilibrium solutions}.
\newblock {\em Proceedings of the 42nd IEEE Conference on Decision and Control
  Maui, Hawaii USA, December 2003}, pages 98--103, 2003.

\bibitem{MR2000132}
Ali Jadbabaie, Jie Lin, and A.~Stephen Morse.
\newblock Correction to: ``{C}oordination of groups of mobile autonomous agents
  using nearest neighbor rules'' [{IEEE} {T}rans. {A}utomat. {C}ontrol {\bf 48}
  (2003), no. 6, 988--1001; {MR} 1986266].
\newblock {\em IEEE Trans. Automat. Control}, 48(9):1675, 2003.

\bibitem{CaMaWo2016}
Stephan Martinand Marie-Therese~Wolfram Jose A.~Carrillo.
\newblock An improved version of the {Hughes} model for pedestrian flow.
\newblock {\em Mathematical Models and Methods in Applied Sciences},
  26(4):671--697, 2016.

\bibitem{KeMinAuWan02}
J.~Ke, J.~Minett, C.-P. Au, and W.-Y. Wang.
\newblock Self-organization and selection in the emergence of vocabulary.
\newblock {\em Complexity}, 7:41--54, 2002.

\bibitem{kese70}
Evelyn~F. Keller and Lee~A. Segel.
\newblock {Initiation of slime mold aggregation viewed as an instability.}
\newblock {\em J. Theor. Biol.}, 26(3):399--415, 1970.

\bibitem{KIRCHDOERFER201681}
T.~Kirchdoerfer and M.~Ortiz.
\newblock Data-driven computational mechanics.
\newblock {\em Comput. Methods Appl. Mech. Engrg.}, 304:81--101, 2016.

\bibitem{KocWhi98}
A.L. Koch and D.~White.
\newblock The social lifestyle of myxobacteria.
\newblock {\em Bioessays 20}, pages 1030--1038, 1998.

\bibitem{kr00}
U.~{Krause}.
\newblock {A discrete nonlinear and non-autonomous model of consensus
  formation.}
\newblock In {\em {Communications in difference equations. Proceedings of the
  4th international conference on difference equations, Poznan, Poland, August
  27--31, 1998}}, pages 227--236. Amsterdam: Gordon and Breach Science
  Publishers, 2000.

\bibitem{kuramoto2003chemical}
Y.~Kuramoto.
\newblock {\em Chemical Oscillations, Waves, and Turbulence}.
\newblock Dover Books on Chemistry Series. Dover Publications, 2003.

\bibitem{lali07}
Jean-Michel Lasry and Pierre-Louis Lions.
\newblock {Mean field games.}
\newblock {\em Jpn. J. Math. (3)}, 2(1):229--260, 2007.

\bibitem{LeoFio01}
N.E. Leonard and E.~Fiorelli.
\newblock Virtual leaders, artificial potentials and coordinated control of
  groups.
\newblock {\em Proc. 40th IEEE Conf. Decision Contr.}, pages 2968--2973, 2001.

\bibitem{lu2019learning}
Fei Lu, Mauro Maggioni, and Sui Tang.
\newblock Learning interaction kernels in heterogeneous systems of agents from
  multiple trajectories.
\newblock {\em to appear J. Mach. Learn. Res.}, 2019.

\bibitem{Lu14424}
Fei Lu, Ming Zhong, Sui Tang, and Mauro Maggioni.
\newblock Nonparametric inference of interaction laws in systems of agents from
  trajectory data.
\newblock {\em Proc. Natl. Acad. Sci. USA}, 116(29):14424--14433, 2019.

\bibitem{Optim}
Patrick~Kofod Mogensen and Asbj{\o}rn~Nilsen Riseth.
\newblock Optim: A mathematical optimization package for {Julia}.
\newblock {\em Journal of Open Source Software}, 3(24):615, 2018.

\bibitem{doi:10.1137/19M1273426}
Marco Morandotti and Francesco Solombrino.
\newblock Mean-field analysis of multipopulation dynamics with label switching.
\newblock {\em SIAM Journal on Mathematical Analysis}, 52(2):1427--1462, 2020.

\bibitem{Niw94}
H.S. Niwa.
\newblock Self-organizing dynamic model of fish schooling.
\newblock {\em J. Theor. Biol.}, 171:123--136, 1994.

\bibitem{NCM10}
M.~Nuorian, P.E. Caines, and R.P. Malham{\'e}.
\newblock {Synthesis of {Cucker-Smale} type flocking via mean field stochastic
  control theory: {Nash} equilibria}.
\newblock {\em Proceedings of the 48th Allerton Conf. on Comm., Cont. and
  Comp., Monticello, Illinois, pp. 814-819, Sep. 2010}, pages 814--815, 2010.

\bibitem{NCM11}
M.~Nuorian, P.E. Caines, and R.P. Malham{\'e}.
\newblock Mean field analysis of controlled {Cucker-Smale} type flocking:
  Linear analysis and perturbation equations.
\newblock {\em Proceedings of 18th IFAC World Congress Milano (Italy) August 28
  - September 2, 2011}, pages 4471--4476, 2011.

\bibitem{PE99}
J.~Parrish and L.~Edelstein-Keshet.
\newblock Complexity, pattern, and evolutionary trade-offs in animal
  aggregation.
\newblock {\em Science}, 294:99--101, 1999.

\bibitem{ParVisGru02}
J.K. Parrish, S.V. Viscido, and D.~Gruenbaum.
\newblock Self-organized fish schools: An examination of emergent properties.
\newblock {\em Biol. Bull.}, 202:296--305, 2002.

\bibitem{5459260}
S.~Pellegrini, A.~Ess, K.~Schindler, and L.~van Gool.
\newblock You'll never walk alone: Modeling social behavior for multi-target
  tracking.
\newblock In {\em 2009 IEEE 12th International Conference on Computer Vision},
  pages 261--268, 2009.

\bibitem{PerGomElo09}
L.~Perea, G.~G\'omez, and P.~Elosegui.
\newblock Extension of the {C}ucker-{S}male control law to space flight
  formations.
\newblock {\em AIAA Journal of Guidance, Control, and Dynamics}, 32:527--537,
  2009.

\bibitem{be07}
Beno\^it Perthame.
\newblock {\em {Transport Equations in Biology.}}
\newblock {Basel: Birkh\"auser}, 2007.

\bibitem{GameTheoryTransportModel2021}
Yalda Rahmati, Alireza Talebpour, Archak Mittal, and James Fishelson.
\newblock Game theory-based framework for modeling human–vehicle interactions
  on the road.
\newblock {\em Transportation Research Record}, 2674(9):701--713, 2020.

\bibitem{Rom96}
W.L. Romey.
\newblock Individual differences make a difference in the trajectories of
  simulated schools of fish.
\newblock {\em Ecol. Model.}, 92:65--77, 1996.

\bibitem{sctrwa17-1}
Hayden Schaeffer, Giang Tran, and Rachel Ward.
\newblock Extracting sparse high-dimensional dynamics from limited data.
\newblock {\em SIAM J. Appl. Math.}, 78(6):3279--3295, 2018.

\bibitem{sctrwa17}
Hayden Schaeffer, Giang Tran, and Rachel Ward.
\newblock Learning dynamical systems and bifurcation via group sparsity,
  arXiv:1709.01558.

\bibitem{SS20172}
G.~Scilla and Francesco Solombrino.
\newblock Delayed loss of stability in singularly perturbed finite-dimensional
  gradient flows.
\newblock {\em Asymptot. Anal.}, 110(1-2):1--19, 2018.

\bibitem{SS2017}
Giovanni Scilla and Francesco Solombrino.
\newblock Multiscale analysis of singularly perturbed finite dimensional
  gradient flows: the minimizing movement approach.
\newblock {\em Nonlinearity}, 31(11):5036--5074, 2018.

\bibitem{MR2438215}
M.~B. Short, M.~R. D'Orsogna, V.~B. Pasour, G.~E. Tita, P.~J. Brantingham,
  A.~L. Bertozzi, and L.~B. Chayes.
\newblock A statistical model of criminal behavior.
\newblock {\em Math. Models Methods Appl. Sci.}, 18(suppl.):1249--1267, 2008.

\bibitem{OSQP}
Bartolomeo Stellato, Goran Banjac, Paul Goulart, Alberto Bemporad, and Stephen
  Boyd.
\newblock {{OSQP}}: {A}n operator splitting solver for quadratic programs.
\newblock {\em Mathematical Programming Computation}, 12(4):637--672, 2020.

\bibitem{SugSan97}
K.~Sugawara and M.~Sano.
\newblock Cooperative acceleration of task performance: {F}oraging behavior of
  interacting multi-robots system.
\newblock {\em Physica D}, 100:343--354, 1997.

\bibitem{TonTu95}
J.~Toner and Y.~Tu.
\newblock Long-range order in a two-dimensional dynamical xy model: {H}ow birds
  fly together.
\newblock {\em Phys. Rev. Lett.}, 75:4326--4329, 1995.

\bibitem{trwa16}
Giang Tran and Rachel Ward.
\newblock Exact recovery of chaotic systems from highly corrupted data.
\newblock {\em Multiscale Model. Simul.}, 15(3):1108--1129, 2017.

\bibitem{TsybakovNonparametricEstimation}
Alexandre~B. Tsybakov.
\newblock {\em Introduction to Nonparametric Estimation}.
\newblock Springer, 2009.

\bibitem{vicsek}
T.~Vicsek, A.~Czirok, E.~Ben-Jacob, I.~Cohen, and O.~Shochet.
\newblock Novel type of phase transition in a system of self-driven particles.
\newblock {\em Phys. Rev. Lett.}, 75:1226--1229, 1995.

\bibitem{viza12}
T.~Vicsek and A.~Zafeiris.
\newblock Collective motion.
\newblock {\em Physics Reports}, 517:71--140, 2012.

\bibitem{PedestrianDataset2019}
Dongfang Yang, Linhui Li, Keith~A. Redmill, and {\"U}mit {\"O}zg{\"u}ner.
\newblock Top-view trajectories: A pedestrian dataset of vehicle-crowd
  interaction from controlled experiments and crowded campus.
\newblock {\em 2019 IEEE Intelligent Vehicles Symposium (IV)}, pages 899--904,
  2019.

\bibitem{YEECBKMS09}
C.~Yates, R.~Erban, C.~Escudero, L.~Couzin, J.~Buhl, L.~Kevrekidis, P.~Maini,
  and D.~Sumpter.
\newblock Inherent noise can facilitate coherence in collective swarm motion.
\newblock {\em Proceedings of the National Academy of Sciences},
  106:5464--5469, 2009.

\bibitem{ZHONG2020132542}
Ming Zhong, Jason Miller, and Mauro Maggioni.
\newblock Data-driven discovery of emergent behaviors in collective dynamics.
\newblock {\em Physica D: Nonlinear Phenomena}, 411:132542, 2020.

\end{thebibliography}
\bibliographystyle{plain}

\end{document}